\documentclass{article}
\usepackage[
backend=biber,
style=alphabetic,
natbib=true,
url=false,
doi=true,
eprint=true,
giveninits=true,
isbn=false,
maxbibnames=12,
]{biblatex}
\usepackage[a4paper, margin=3cm]{geometry}
\usepackage{hyperref}
\usepackage{authblk}
\usepackage{amsthm,amsmath,amsfonts,amssymb}
\usepackage{mathtools}
\usepackage[capitalize,nameinlink,noabbrev]{cleveref} 
\usepackage{enumitem} 
\usepackage{graphicx} 
\usepackage{dsfont} 
\usepackage{tikz} 
\usetikzlibrary{decorations.pathreplacing}
\usetikzlibrary {patterns,patterns.meta}
\usepackage{pgfplots}
\usepgfplotslibrary{colorbrewer}

\definecolor{myOrange}{HTML}{F16E04}
\definecolor{myBlue}{HTML}{A152D6}
\definecolor{metricSpacepColor}{HTML}{000000}

\newcommand{\ol}[2]{\overline{#1#2}}

\newcommand{\tran}{\tau}
\newcommand{\dtran}{\tau\pr}
\newcommand{\ddtran}{\tau\prr}
\newcommand{\dddtran}{\tau\prrr}
\newcommand{\ddltran}{\tau^{\prime\ominus}}
\newcommand{\ddrtran}{\tau^{\prime\oplus}}

\newcommand{\setcc}{\mc S}
\newcommand{\setcciz}{\mc S_0^+}

\newcommand{\ld}{^\ominus}

\newcommand{\rd}{^\oplus}
\newcommand{\woinf}[1]{#1^\ominus}
\newcommand{\wosup}[1]{#1^\oplus}

\newcommand{\semiDerivs}{\partial^\pm}

\newcommand{\GeodUS}{\Gamma_1} 
\newcommand{\geodft}[2]{\gamma_{#1\to#2}} 
\renewcommand{\subset}{\subseteq}
\newcommand{\lcb}{\left\lbrace} 
\newcommand{\rcb}{\right\rbrace} 
\newcommand{\cb}[1]{\lcb #1 \rcb} 
\newcommand{\lab}{\left[} 
\newcommand{\rab}{\right]} 
\newcommand{\ab}[1]{\lab #1 \rab} 
\newcommand{\abOf}[1]{\!\ab{#1}} 
\newcommand{\lb}{\left(} 
\newcommand{\rb}{\right)} 
\newcommand{\br}[1]{\lb #1 \rb} 
\newcommand{\brOf}[1]{\!\br{#1}} 
\newcommand{\abs}[1]{\left| #1 \right|} 
\newcommand*{\E}{\mathbb{E}} 
\let\Pr\relax
\newcommand*{\Pr}{\mathbb{P}} 
\newcommand{\sizedMid}[2]{#1 \ \kern-\nulldelimiterspace\mathopen{}\left| \vphantom{#1}\ #2\right.\mathclose{}\kern-\nulldelimiterspace}
\newcommand{\EOf}[1]{\E\abOf{#1}}
\newcommand{\Eof}[1]{\E[#1]}

\newcommand{\PrOf}[1]{\Pr\mathopen{}\lb #1 \rb\mathclose{}}
\newcommand{\Prof}[1]{\Pr(#1)}

\newcommand{\setByEle}[2]{\cb{\sizedMid{#1}{#2}}}
\newcommand{\setByEleInText}[2]{\{#1 \mid #2\}}
\DeclareMathOperator{\diam}{\mathsf{diam}}

\newcommand{\ball}[2]{\mathrm{B}_{#2}(#1)}

\DeclarePairedDelimiterX\Set[1]{\lbrace}{\rbrace}%
{  #1 }
\newcommand{\Ex}{\E\expectarg}
\DeclarePairedDelimiterX{\expectarg}[1]{[}{]}{%
	\ifnum\currentgrouptype=16 \else\begingroup\fi
	\activatebar#1
	\ifnum\currentgrouptype=16 \else\endgroup\fi
}
\newcommand{\innermid}{\nonscript\;\delimsize\vert\nonscript\;}
\newcommand{\activatebar}{%
	\begingroup\lccode`\~=`\|
	\lowercase{\endgroup\let~}\innermid 
	\mathcode`|=\string"8000
}
\newcommand*{\mc}[1]{\mathcal{#1}}
\newcommand*{\mb}[1]{\mathbb{#1}}

\newcommand*{\ms}[1]{\mathsf{#1}}

\newcommand*{\mf}[1]{\mathfrak{#1}}
\newcommand{\N}{\mathbb{N}}

\newcommand{\R}{\mathbb{R}}
\newcommand{\Rp}{[0, \infty)}
\newcommand{\Rpp}{(0, \infty)}

\newcommand{\transpose}{\!^\top\!}
\newcommand{\tr}{\transpose}
\newcommand{\pr}{^\prime}
\newcommand{\prr}{^{\prime\prime}}
\newcommand{\prrr}{^{\prime\prime\prime}}
\def\integral from #1to #2of #3by #4;{\int_{#1}^{#2} \! #3 \mathrm{d}#4} %
\def\integralMeasure in #1of #2by #3of #4;{\int_{#1} \! #2{#4} #3{\mathrm{d}#4}} %
\def\mapping #1from #2to #3;{#1 \colon #2 \rightarrow #3}
\def\mappingDef #1from #2to #3maps #4to #5;{#1 \colon #2 \rightarrow #3,\ #4 \mapsto #5}
\def\seq #1by #2;{\br{#1}_{#2\in\N}}
\def\seqInText #1by #2;{(#1)_{#2\in\N}}
\newcommand{\innerProduct}[2]{\left\langle#1\,,\, #2\right\rangle}
\newcommand{\ip}[2]{\innerProduct{#1}{#2}}
\newcommand{\lebesgue}{\mathcal{L}}
\newcommand{\lebesguePow}[1]{\lebesgue^{#1}}

\newcommand{\dl}{\mathrm{d}}
\def\converges for #1to #2;{\xrightarrow{#1} #2}
\def\convergesAlmostSurely for #1to #2;{\xrightarrow{#1}_{\mathsf{fs}} #2}
\def\convergesInProbability for #1to #2;{\xrightarrow{#1}_{\mathsf{p}} #2}
\def\convergesInL #1for #2to #3;{\xrightarrow{#2}_{\lebesguePow{#1}} #3}

\newcommand{\ind}{\mathds{1}}
\newcommand{\indOf}[1]{\ind_{\!#1}}%
\newcommand{\indOfOf}[2]{\ind_{\!#1}\!\brOf{#2}}%
\newcommand{\norm}{\left\Vert \cdot \right\Vert}
\newcommand{\normof}[1]{\Vert #1 \Vert}
\newcommand{\normOf}[1]{\left\Vert #1 \right\Vert}
\newcommand{\equationFullstop}{\, .}
\newcommand{\eqfs}{\equationFullstop}
\newcommand{\equationComma}{\, ,}
\newcommand{\eqcm}{\equationComma}

\DeclareMathOperator*{\argmin}{arg\,min}
\DeclareMathOperator*{\argmax}{arg\,max}
\newbox{\myorcidthanksbox}
\sbox{\myorcidthanksbox}{\large\includegraphics[height=1.7ex]{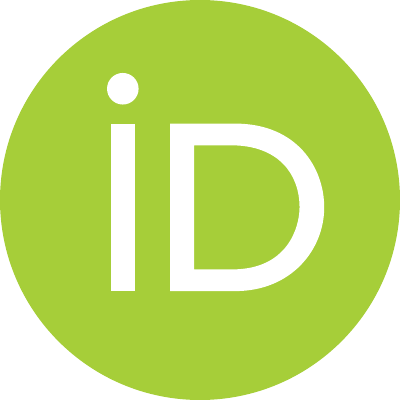}}
\newcommand{\orcidthanks}[1]{%
	\href{https://orcid.org/#1}{\usebox{\myorcidthanksbox}\,#1}}
\usepackage{thmtools}
\def\postBoxSkip{1.0ex}
\def\postBoxSkipCmd{\vskip\postBoxSkip}
\def\preBoxSkip{1.0ex}
\def\preBoxSkipCmd{\vskip\preBoxSkip}
\declaretheoremstyle[
	bodyfont=\normalfont,
	postfoothook={\postBoxSkipCmd},
	preheadhook={\preBoxSkipCmd},
	mdframed={
		backgroundcolor = black!2,
		startcode={},
}]{ruledBoxStyle}
\declaretheoremstyle[
	bodyfont=\normalfont,
	postfoothook={\postBoxSkipCmd},
	preheadhook={\preBoxSkipCmd},
	mdframed={
		backgroundcolor=white,
}]{ruledBoxStyleWhite}
\declaretheoremstyle[
	bodyfont=\normalfont,
	postfoothook={\postBoxSkipCmd},
	preheadhook={\preBoxSkipCmd},
	mdframed={
		backgroundcolor=black!2,
		linecolor = black!2,
		tikzsetting = {
			draw = black,
			line width = 2pt,%
			dashed,%
			dash pattern = on 10pt off 3pt
		},
}]{dashedBoxStyle}
\declaretheoremstyle[
	bodyfont=\normalfont,
	postfoothook={\postBoxSkipCmd},
	preheadhook={\preBoxSkipCmd},
	mdframed={
		linecolor = white,
		startcode={},
		tikzsetting = {
			draw = black,
			line width = 1pt,%
			loosely dotted,
		},
	}
]{dashedStyle}
\declaretheoremstyle[
	bodyfont=\normalfont,
	postfoothook={\postBoxSkipCmd},
	preheadhook={\preBoxSkipCmd},
	mdframed={
		linecolor = black,
		innerlinewidth=1pt,outerlinewidth=1pt,
		middlelinewidth=1pt,
		linecolor=black,middlelinecolor=white,
		startcode={},
	}
]{doubleStyle}
\declaretheoremstyle[
	bodyfont=\normalfont,
	postfoothook={\postBoxSkipCmd},
	preheadhook={\preBoxSkipCmd},
	mdframed={
		backgroundcolor = black!4,
		linecolor = black!4,
		startcode={},
}]{boxStyle}
\declaretheoremstyle[
	headfont=\normalfont\itshape, 
	notefont=\normalfont\itshape, 
	notebraces={}{},
	bodyfont=\normalfont,
	qed=\qedsymbol,
	numbered=no,
	headindent=0pt,
	postheadspace=1ex,
	name={Proof},
	postheadhook={},
	mdframed={
		hidealllines = true,
		innerrightmargin = 0pt,
		innerleftmargin = 0pt,
		innertopmargin = 0pt,
		innerbottommargin = 0pt,
		leftmargin = 0pt,
		rightmargin = 0pt,
	}
]{proofStyle}
\declaretheoremstyle[
	bodyfont=\normalfont,
	postfoothook={\postBoxSkipCmd},
	preheadhook={\preBoxSkipCmd},
	mdframed={
		backgroundcolor = white,
		linecolor = black,
		startcode={},
		leftline = false,
		rightline = false,
}]{tobBottomStyle}
\declaretheoremstyle[
bodyfont=\normalfont,
]{standardStyle}
\declaretheorem[style=ruledBoxStyle,name=Definition, numberwithin=section]{definition}
\declaretheorem[style=ruledBoxStyle,name=Lemma,numberwithin=section,numberlike=definition]{lemma}
\declaretheorem[style=ruledBoxStyle,name=Proposition,numberwithin=section,numberlike=definition]{proposition}
\declaretheorem[style=ruledBoxStyle,name=Theorem,numberwithin=section,numberlike=definition]{theorem}
\declaretheorem[style=ruledBoxStyle,name=Corollary,numberwithin=section,numberlike=definition]{corollary}
\declaretheorem[style=ruledBoxStyle,name=Theorem,numbered=no]{theorem*}
\declaretheorem[style=boxStyle,name=Remark,numberwithin=section,numberlike=definition]{remark}
\declaretheorem[style=boxStyle,name=Notation,numberwithin=section,numberlike=definition]{notation}
\declaretheorem[style=dashedStyle,name=Example,numberwithin=section,numberlike=definition]{example}

\numberwithin{equation}{section}
\title{Variance Inequalities for Transformed Fréchet Means in Hadamard Spaces}
\date{}
\author[1,2]{Christof Sch\"otz\thanks{math@christof-schoetz.de, \orcidthanks{0000-0003-3528-4544}}}
\affil[1]{Potsdam Institute for Climate Impact Research}
\affil[2]{Technical University of Munich}
\begin{document}
\maketitle
\begin{abstract}
	The Fréchet mean (or barycenter) generalizes the expectation of a random variable to metric spaces by minimizing the expected squared distance to the random variable. Similarly, the median can be generalized by its property of minimizing the expected absolute distance. We consider the class of transformed Fréchet means with nondecreasing, convex transformations that have a concave derivative. This class includes the Fréchet median, the Fréchet mean, the Huber loss-induced Fréchet mean, and other statistics related to robust statistics in metric spaces.
We study variance inequalities for these transformed Fréchet means. These inequalities describe how the expected transformed distance grows when moving away from a minimizer, i.e., from a transformed Fréchet mean. Variance inequalities are useful in the theory of estimation and numerical approximation of transformed Fréchet means.
Our focus is on variance inequalities in Hadamard spaces -- metric spaces with globally nonpositive curvature. Notably, some results are new also for Euclidean spaces. Additionally, we are able to characterize uniqueness of transformed Fréchet means, in particular of the Fréchet median.  

\end{abstract}
\tableofcontents
\section{Introduction}
\subsection{Transformed Fréchet Means}
Let $(\mc Q, d)$ be a metric space and use the short hand notation $\ol qp := d(q,p)$. When studying a $\mc Q$-valued random variable $Y$, a basic property to be considered is a mean value. A \emph{Fréchet mean} \cite{frechet48} (also called \emph{barycenter} or \emph{center of mass}) of $Y$ is a point in $\argmin_{q\in\mc Q} \Ex{\ol Yq^2}$. If $\mc Q$ is a Euclidean space, the Fréchet mean is the expectation, $m = \Ex{Y}$. Alternatively, one may consider a generalization of the median, which minimizes the absolute distance instead of the squared distance: A Fréchet median of $Y$ is any element of $\argmin_{q\in\mc Q}\Ex{\ol Yq}$.

We treat both, Fréchet mean and median (and more), in a common framework: Fix an arbitrary reference point $o\in\mc Q$. For a function $\tran\colon\Rp\to \R$, a point $m\in\mc Q$ is a \textit{transformed Fréchet mean} or \textit{$\tran$-Fréchet mean} of $Y$, if
\begin{equation}
    m\in\argmin_{q\in\mc Q} \Ex*{\tran(\ol Yq) - \tran(\ol Yo)}
    \eqfs
\end{equation}
Subtracting $\tran(\ol Yo)$ does not change the set of minimizers if $\Ex{\tran(\ol Yq)}<\infty$, but it allows a meaningful definition in some instances, where this expectation is infinite, e.g., in the case $\tran = (x\mapsto x^2)$, it is enough to have $\Ex{\ol Yq}<\infty$ (for one and hence for all $q\in\mc Q$), and for $\tran = (x\mapsto x)$, no moment condition is needed. We will restrict ourselves to $\tran\in\setcc$, where $\setcc$ is the set of nondecreasing, convex functions $\tran\colon\Rp\to\R$ with concave derivative $\dtran$. This allows for clean results and encompasses the most interesting transformation functions. Aside from $x\mapsto x^2$ (Fréchet mean) and  $x\mapsto x$ (Fréchet median), $\mc S$ contains a wealth of functions (see \cref{fig:transformations}) such as $x\mapsto x^\alpha$ for $\alpha\in[1,2]$ and the Huber loss $\tran_{\ms h,\delta}$ \cite{huber64} for $\delta\in\Rpp$,
\begin{equation}\label{eq:huber}
    \tran_{\ms h,\delta}(x) :=
    \begin{cases}
        \frac12 x^2 & \text{ for } x \leq \delta\eqcm\\
        \delta(x - \frac12 \delta) & \text{ for } x > \delta\eqfs
    \end{cases}
\end{equation}
The Huber loss is of great importance in robust statistics, as it allows to estimate a mean value in the presence of outliers. It is continuously differentiable, but not twice differentiable at $x=\delta$. If more smoothness is required, one may use the pseudo-Huber loss $\tran_{\ms{ph},\delta}\in\setcc$ \cite{charbonnier94},
\begin{equation}\label{eq:pseudohuber}
    \tran_{\ms{ph},\delta}(x) := \delta^2 \br{\sqrt{1 + \frac{x^2}{\delta^2}} - 1}
    \eqfs
\end{equation}
Another smooth element of $\setcc$ used in previous works is $x\mapsto \log(\cosh(x))$ \cite{green90}, where $\log$ is the natural logarithm, and $\cosh$ is the hyperbolic cosine.

This function and the two Huber-losses have a similar form: They look like $x\mapsto x^2$ close to $0$, which allows for an $L^2$-like theory locally, and like an increasing affine function for large $x$, which entails that we do not require finite moments to define the respective transformed Fréchet mean.
All functions in $\setcc$ are in some way between linear and quadratic and enjoy properties that are between those of the mean and the median. Thus, they form a suitable framework for robust statistics in metric spaces.
\begin{figure}
    \begin{center}
        \begin{tikzpicture}
            \begin{axis}[
                ymin=-0,
                ymax=3,
                xmin=-0,
                xmax=3,
                grid=both,
                ylabel={$\tran(x)$},
                xlabel=$x$,
                x label style={at={(axis description cs:0.5,0)},anchor=north},
                y label style={at={(axis description cs:0.15,0.5)},rotate=-90,anchor=east},
                xtick={0,1,2},
                ytick={0,1,2},
                axis lines = left,
                scale only axis=true,
                width=0.35\textwidth,
                height=0.35\textwidth,
                axis line style = thick,
                samples=100
            ]
                \addplot[domain=0:3, Dark2-A, line width=1pt] {x};
                \addplot[domain=0:3, Dark2-B, line width=1pt] {x^1.5/1.5};
                \addplot[domain=0:3, Dark2-C, line width=1pt] {x^2/2};
                \addplot[domain=0:1, Dark2-D, dash pattern={on 6pt off 3pt}, line width=2pt] {x^2/2};
                \addplot[domain=1:3, Dark2-D, dash pattern={on 6pt off 3pt}, line width=2pt] {x-1/2};
                \addplot[domain=0:3, Dark2-E, dash pattern={on 6pt off 3pt}, line width=2pt] {(1+abs(x)^2)^0.5-1};
                \node at (axis cs:0.7,0.7) [anchor=south east, Dark2-A] {$\tran_1$};
                \node at (axis cs:2.35,2.05) [anchor=center, Dark2-B] {$\tran_{3/2}$};
                \node at (axis cs:2.4,2.7) [anchor=south east, Dark2-C] {$\tran_2$};
                \node at (axis cs:1.6,0.5) [anchor=center, Dark2-D] {$\tran_{\ms h,1}$};
                \node at (axis cs:2.5,1.3) [anchor=center, Dark2-E] {$\tran_{\ms{ph},1}$};
            \end{axis}
        \end{tikzpicture}
        \hfill
        \begin{tikzpicture}
            \begin{axis}[
                ymin=-0,
                ymax=3,
                xmin=-0,
                xmax=3,
                grid=both,
                ylabel={$\dtran(x)$},
                xlabel=$x$,
                x label style={at={(axis description cs:0.5,0)},anchor=north},
                y label style={at={(axis description cs:0.15,0.5)},rotate=-90,anchor=east},
                xtick={0,1,2},
                ytick={0,1,2},
                axis lines = left,
                scale only axis=true,
                width=0.35\textwidth,
                height=0.35\textwidth,
                axis line style = thick,
                samples=100
             ]
                \addplot[domain=0:3, Dark2-A, line width=1pt] {1};
                \addplot[domain=0:3, Dark2-B, line width=1pt] {x^0.5};
                \addplot[domain=0:3, Dark2-C, line width=1pt] {x};
                \addplot[domain=0:1, Dark2-D, dash pattern={on 6pt off 3pt}, line width=2pt] {x};
                \addplot[domain=1:3, Dark2-D, dash pattern={on 6pt off 3pt}, line width=2pt] {1};
                \addplot[domain=0:3, Dark2-E, dash pattern={on 6pt off 3pt}, line width=2pt] {x/(1+x^2)^0.5};
                \node at (axis cs:0.5,1.0) [anchor=south, Dark2-A] {$\dtran_1$};
                \node at (axis cs:2.5,1.6) [anchor=south, Dark2-B] {$\dtran_{3/2}$};
                \node at (axis cs:2.35,2.6) [anchor=south, Dark2-C] {$\dtran_{2}$};
                \node at (axis cs:2.5,1.2) [anchor=center, Dark2-D] {$\dtran_{\ms h,1}$};
                \node at (axis cs:1.6,0.6) [anchor=center, Dark2-E] {$\dtran_{\ms{ph},1}$};
            \end{axis}
        \end{tikzpicture}
    \end{center}
    \caption{Different transformation functions $\tran\in\setcc$ and their derivatives. The functions shown are $\tran_\alpha(x) = \alpha^{-1}x^\alpha$ for $\alpha \in \{1, 3/2, 2\}$ as well as the Huber and pseudo-Huber loss functions with threshold $\delta=1$, see \eqref{eq:huber} and \eqref{eq:pseudohuber}, respectively.}\label{fig:transformations}
\end{figure}
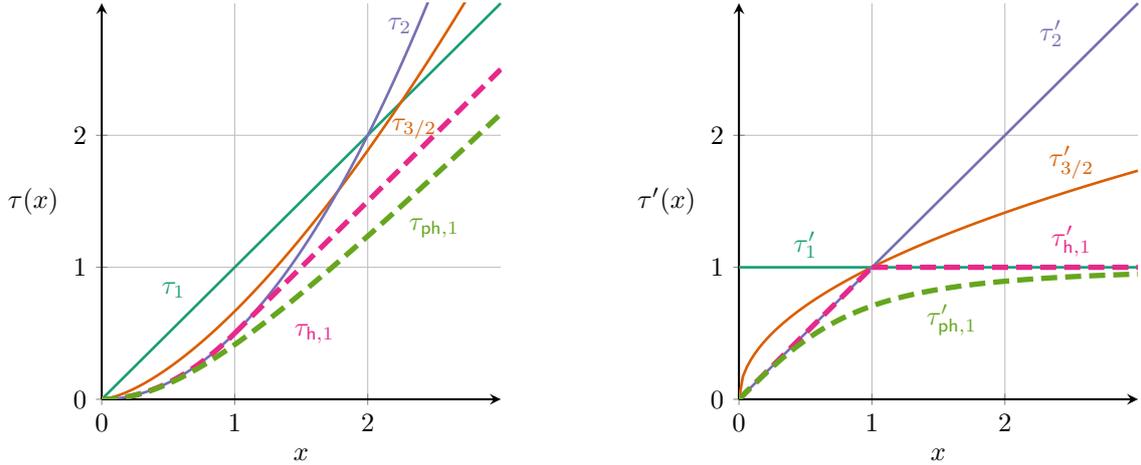
\subsection{Variance Functional and Variance Inequalities}
Our main goal is to investigate properties of the \emph{variance functional} (or \emph{Fréchet functional}) $q\mapsto \Ex{\tran(\ol Yq) - \tran(\ol Yp)}$, $p\in\mc Q$, $\tran\in\setcc$, in certain classes of metric spaces. In particular, we will show \emph{variance inequalities}, i.e., inequalities of the form
\begin{equation}\label{eq:vi:general}
    \Ex{\tran(\ol Yq) - \tran(\ol Ym)} \geq f(\ol qm)
\end{equation}
for $q\in\mc Q$, where $m$ is the $\tran$-Fréchet mean of $Y$ and $f\colon\Rp\to\R$ is some function.

Let us first illustrate the meaning of variance inequalities in a Euclidean space $(\mc Q, \norm)$ for $\tran(x) = x^2$. In this case, the variance inequality is an equality: We have $m = \Ex{Y}$ and
\begin{equation}\label{eq:vareq}
    \Ex*{\normof{Y-q}^2 - \normof{Y-m}^2} = \normof{q - m}^2
\end{equation}
for all $q\in\mc Q$. By noting that $m$ minimizes the variance functional $v(q) := \Ex{\normof{Y-q}^2 - \normof{Y}^2}$, we gain following insight from \eqref{eq:vareq}: If $q$ minimizes the objective $v(q)$ up to at most $\delta$ for a $\delta\geq0$, in the sense that $v(q) \leq \delta + v(m)$, then $q$ can be at most a distance $\sqrt{\delta}$ away from the (unique) minimizer $m$ of $v$. For this statement to hold, it is enough to have \eqref{eq:vareq} as an inequality, i.e., using the metric notation,
\begin{equation}\label{eq:varinequhadamard}
    \Ex*{\ol Yq^2 - \ol Ym^2} \geq \ol qm^2
    \eqfs
\end{equation}

Variance inequalities are an essential ingredient in the theory of estimating the (transformed) Fréchet mean from observations. They typically appear as an assumption, e.g., \cite[assumptions (P2), (U2), (L2), (L3)]{petersen16}, \cite[condition B.5]{Lin2019}, \cite[assumption (A3)]{ahidar20}, \cite[assumption \texttt{Growth}]{schoetz19}, \cite[Theorem 8]{legouic23}, \cite[condition (M)]{Ghosal2023}. In contrast, we prove variance inequalities and provide them in a clean, ready-to-use form.

The study of variance inequalities also yields uniqueness results for the transformed Fréchet mean, including the Fréchet median.
\subsubsection{Transformed Fréchet Mean in General Metric Spaces}
In general metric spaces, we show  (\cref{thm:varineq:general}) that the variance functional behaves like the transformation function for far away points, i.e., $\Ex{\tran(\ol Yq) - \tran(\ol Yp)} \approx \tran(\ol qp)$ for $\tran\in\setcc$ and $\ol qp$ large, and it grows at most linearly for close points, i.e., $\Ex{\tran(\ol Yq) - \tran(\ol Yp)} \lesssim \ol qp\, \Eof{\dtran(\ol Yp)}$ for $\ol qp$ small. Most relevant for statistical results is a lower bound on the variance functional for $q$ close to the transformed Fréchet mean $p = m$, i.e., a variance inequality. This is not covered by these results, but can be established by restricting the geometry of $\mc Q$, as we show below.
\subsubsection{Fréchet Means in Hadamard Spaces}
In a general metric space, it is difficult to obtain variance inequalities. But after restriction to so-called Hadamard spaces, it is well-known that \eqref{eq:varinequhadamard} is true.
\begin{definition}
    A metric space $(\mc Q, d)$ is called \emph{Hadamard space}, if and only if it is complete and for all $y_0, y_1 \in \mc Q$, there exists $m \in \mc Q$ such that
    \begin{equation}\label{eq:def:hadamard}
        \frac12 \ol {y_0}q^2 + \frac12 \ol {y_1}q^2 - \frac14 \ol{y_0}{y_1}^2 \geq \ol qm^2
    \end{equation}
    for all $q\in\mc Q$.
\end{definition}
Hadamard spaces are geodesic metric spaces (each pair of points is connected by a geodesic) of nonpositive curvature (triangles are "thinner" than Euclidean ones). They are also called \emph{global NPC spaces} or \emph{complete $CAT(0)$ spaces}. Inequality \eqref{eq:def:hadamard} is the variance inequality \eqref{eq:varinequhadamard} for a random variable with $\Prof{Y=y_0} = \Prof{Y=y_1} = \frac12$.
\begin{proposition}[{\cite[Proposition 4.4]{sturm03}}]\label{prp:vi:hadamard}
    A complete metric space $(\mc Q, d)$ is Hadamard if and only if, for all $\mc Q$-valued random variables $Y$ with $\Eof{\ol Yo} < \infty$ for one (and hence all) $o\in\mc Q$, we have $\Ex{\ol Yq^2 - \ol Ym^2} \geq \ol qm^2$ for all $q\in\mc Q$, where $m$ is the Fréchet mean of $Y$.
\end{proposition}
\cref{prp:vi:hadamard} implies that the Fréchet mean is unique in Hadamard spaces, which is not guaranteed in general metric spaces.
\subsubsection{Transformed Fréchet Means in Hadamard Spaces}
Here, we generalize \eqref{eq:varinequhadamard} in Hadamard spaces $\mc Q$ from the square function to all nondecreasing, convex functions with concave derivative $\tran\in\setcc$: For a transformed Fréchet mean $m \in \argmin_{q\in\mc Q}\Ex*{\tran(\ol Yq) - \tran(\ol Yo)}$, we have
\begin{equation}\label{eq:glimpsatresult}
    \Ex*{\tran(\ol Yq) - \tran(\ol Ym)} \geq \frac12 \ol qm^2 \Ex*{\ddtran(\max(\ol Ym ,\ol Yq))}
\end{equation}
for all $q\in\mc Q$, assuming $\Ex{\dtran(\ol Yo)} < \infty$, see \cref{thm:varinequ}. This yields an at least quadratic growth of the variance functional close to $m$ for strictly convex transformations.
As the second derivative $\ddtran$ vanishes for linear $\tran$, we treat the Fréchet median separately and obtain a result of the form
\begin{equation}
	\Ex*{\ol Yq - \ol Ym} \geq \frac12 \eta^2 \,\ol qm^2\, \Ex*{\max\brOf{\ol Ym, \ol Yq}^{-1} \indOf{A(m, q, \eta)}(Y)}
	\eqcm
\end{equation}
where $\eta\in(0,1)$ and $A(m, q, \eta)$ restricts the expectation to points $Y$ that have geodesics to $m$ and $q$ hitting the geodesic between $m$ and $q$ at a steep enough angle (depending on $\eta$), see \cref{thm:near:median}.
\subsection{Related Work}
The article \cite{ahidar20} provides a condition using extendable geodesics to obtain variance inequalities. This idea is generalized in \cite{legouic23} using what the authors call a hugging function. For these results, it still may be difficult to verify the condition that yields the variance inequality.
\subsubsection{Hadamard Spaces}
Examples of Hadamard spaces include
\begin{itemize}
    \item the Euclidean spaces and -- more generally -- Hilbert spaces \cite[Proposition 3.5]{sturm03},
    \item Cartan--Hadamard manifolds, i.e., complete simply connected Riemannian manifolds of nonpositive sectional curvature \cite[Proposition 3.1]{sturm03},
    \item $\R$-trees (also called metric trees), i.e., geodesic metric spaces that contain no subset homeomorphic to a circle \cite{evans2008probability},
    \item the space of phylogenetic trees when endowed with the Billera--Holmes--Vogtmann metric \cite{billera01},
    \item the set of symmetric positive definite matrices with a metric that yields a geometric mean of such matrices as its induced Fréchet mean \cite{bhatia06}.
\end{itemize}
New Hadamard spaces can be created from given Hadamard spaces by taking closed convex subsets, as images of isometries, as product spaces, as $L^2$-spaces of functions with values in a Hadamard space, or by gluing together Hadamard spaces \cite[section 3]{sturm03}.
A reverse of the variance inequality has been studied in \cite{Sturm2005} for Hadamard spaces and in \cite{ohta12} for Alexandrov space with curvature bounded below.
\subsubsection{Quadratic Growth}
One often encounters variance inequalities \eqref{eq:vi:general}, where $f$ is quadratic as in \cref{prp:vi:hadamard} or under mild assumptions in \eqref{eq:glimpsatresult}. This can be explained in view of a Taylor expansion of the variance functional: For simplicity, we calculate on the real line. Let $Y$ be an $\R$-valued random variable and $\tran\colon\Rp\to\R$. Denote the variance functional as $v(q) := \Ex{\tran(\abs{Y-q}) - \tran(\abs{Y})}$ and assume the necessary integrability and smoothness conditions. Let $m\in\R$ denote a $\tran$-Fréchet mean of $Y$. The Taylor expansion yields
\begin{equation}\label{eq:taylor:thrid}
    v(q) = v(m) + (q-m)v\pr(m) + \frac12(q-m)^2 v\prr(m) + O(\abs{q-m}^3)
    \eqcm
\end{equation}
where $v\pr(m) = 0$ as $m$ is a minimizer of $v$. If $v\prr(m) \neq 0$ and $q$ is close enough to $m$, the quadratic term dominates.
On Riemannian manifolds, this strategy yields central limit theorems for the Fréchet mean \cite[Theorem 2.3]{Bhattacharya05Large}. In \cite[Assumption 2.6]{eltzner19}, the condition of a non-vanishing second order term in the Taylor approximation is weakened to a more general power series expansion. The higher order polynomial then yields a slower than parametric rate in the central limit theorem for generalized Fréchet means, which is termed \emph{smeariness}.
\subsubsection{Median and Huber Loss}
The Fréchet median, $\tran = (x \mapsto x)$, generalizes the geometric median:
If $\mc Q$ is a Banach space with norm $\norm$,
a \emph{geometric median} (also called \emph{spatial median} or \emph{Fermat--Weber point}) is any
\begin{equation}\label{eq:geometricmedian}
    m\in\argmin_{q\in\mc Q} \Ex{\normof{Y-q}-\normof{Y}}
    \eqfs
\end{equation}

There is a large body of literature on the geometric median. We mention only a few and refer the reader to the references therein.
For statistical properties of the geometric median in Euclidean spaces, see \cite{mottonen10}. Its application to robust statistics in Banach spaces is studied in \cite{minsker15}. In \cite{cardot17}, the authors consider properties of an algorithm for computing geometric medians in Hilbert spaces. Recent results on the statistical and numerical properties of the geometric median can be found in \cite{minsker2023geometric}.

The Fréchet median has been studied on Riemannian manifolds, e.g., \cite{arnaudon13}. For computational aspects of the Fréchet median in Hadamard spaces, see \cite{bacak14compute}. In \cite{Barbaresco2013}, the Fréchet median is applied in the context of information geometry. The Fréchet median has certain robustness properties and can be used to improve the robustness of the Fréchet mean via a median-of-means approach \cite{yun2023exponential}.
The Fréchet median or Fermat Weber point is also studied in metric spaces of phylogenetic trees \cite{LinYoshida18}.

In a very recent preprint \cite{lee2024hubermeansriemannianmanifolds}, the authors discuss the transformed Fréchet means induced by the (pseudo) Huber-loss function on Riemannian manifolds.
\subsubsection{Other Generalizations of the Fréchet Mean}
To capture Fréchet mean and Fréchet median in one definition, one may consider power Fréchet means, which are also called $\alpha$-Fréchet means. These are the elements of
\begin{equation}\label{eq:alphafm}
    \argmin_{q\in\mc Q}\Ex{\ol Yq^\alpha - \ol Yo^\alpha}
    \eqcm
\end{equation}
where $\alpha\in\Rpp$. Strong laws of large numbers are known for (sets of) $\alpha$-Fréchet means \cite{Sverdrup1981, schoetz22, evans2023limit, blanchard2022frechet}.
If $\alpha\in[1,2]$, the $\alpha$-Fréchet mean is a transformed Fréchet mean with transformation in $\setcc$, which we consider in this article. For such $\alpha$-Fréchet means, \cite[Proposition 2]{yun2023exponential} presents a sufficient condition for a variance inequality that may still be hard to check.

In a recent and parallel work \cite{romon2023convexgeneralizedfrechetmeans}, the authors consider transformed Fréchet means with convex transformation functions in $\R$-trees (also called metric trees) with finite diameter and a finite number of vertices. These are a special kind of Hadamard spaces. Our results are more general in that we discuss general Hadamard spaces and more restricted in that we assume concavity of the derivative of the transformation function. Furthermore, our focus is on variance inequalities, theirs on a phenomenon called \emph{stickiness} \cite{Lammers2023}, where a plug-in estimator is identical to the transformed Fréchet mean with positive probability. An intersection is the discussion of the Fréchet median set, \cite[Proposition 2.18 (applied to linear losses), Proposition 4.3]{romon2023convexgeneralizedfrechetmeans}, where our result, \cref{thm:medianConcentration}, is more general and precise.

One can generalize the idea of minimizing some loss or cost quite far and replace $d^\alpha$ by a generic function $\mf c\colon\mc Y \times \mc Q\to\R$, where $\mc Y$ is allowed to be a different space than $\mc Q$. Then $m\in\mc Q$ is a \emph{generalized Fréchet mean} of the $\mc Y$-valued random variable $Y$ if
\begin{equation}
    m\in \argmin_{q\in\mc Q}\Ex{\mf c(Y,q)}
    \eqfs
\end{equation}
This can be viewed as the definition of a generic M-estimator, e.g., \cite[section 3.2]{vaart23}. In the context of Fréchet means, it was introduced in \cite{Huckemann2011Intrinsic}. Strong laws of large numbers are available in \cite{schoetz22} and rates of convergence of empirical generalized Fréchet means are treated in \cite{schoetz19}.
\subsubsection{Nondecreasing, Convex Functions with Concave Derivatives }
The class $\setcc$ of nondecreasing, convex functions with concave derivatives has been studied in the context of Hadamard spaces in \cite{schotz2023quadruple}, where a relationship between the transformations of the six distances between any four points in the metric space is shown. In \cite{niculescu2024functionalinequalitiesframeworkbanach}, this class of functions is studied in Banach spaces and connected to higher order convexity. Strong laws of large numbers for most instances of the respective class of $\tran$-Fréchet means, $\tran\in\setcc$, have been derived in \cite[Section 4]{schoetz22}.
\subsection{Contribution}
\begin{enumerate}[label=(\Roman*)]
    \item
    We introduce the setting of $\tran$-Fréchet means with $\tran\in\setcc$ in Hadamard spaces, which is an elegant framework for robust statistics: We only require a minimal moment assumption and make no further topological restrictions (such as the common Heine-Borel assumption). Moreover, the framework covers many different transformations, including the prominent Huber loss, and at the same time yields results that do not require elusive conditions.
    \item
    First, we provide asymptotic properties of the variance functional for $\tran$-Fréchet means in general metric spaces in \cref{thm:varineq:general}. The result is useful to show that the set of  $\tran$-Fréchet means is bounded (\cref{lmm:closedBoundedConvex}).
    \item
    To obtain variance inequalities, we restrict ourselves to Hadamard spaces. In these spaces, the distances between a point and a geodesic have a well-known convexity property, which we here call $\mc G$-convexity. In \cref{lmm:semitaylor}, we present second order lower bounds for $\mc G$-convex functions and transformed $\mc G$-convex functions with transformation $\tran\in\setcc$. This bound is related to the idea presented in \eqref{eq:taylor:thrid}. But it uses a second order remainder term instead of a third order one and also takes care of non-Euclidean geometry.
    \item
    \cref{lmm:semitaylor} allows us to derive a ready-to-use variance inequality for $\tran$-Fréchet means in Hadamard spaces (\cref{thm:varinequ} or \eqref{eq:glimpsatresult}). It directly implies uniqueness of the $\tran$-Fréchet mean for most of the transformations $\tran\in\setcc$, see \cref{cor:nonuniquetran} and \cref{cor:uniquetran}. \cref{cor:varinequ} shows quadratic (and faster) growth for strictly convex $\tran$. A variance inequality for the special case where we have positive mass at the location of the $\tran$-Fréchet mean is shown in \cref{thm:point:convex}.
    \item
    As \cref{thm:varinequ} requires the second derivative of the transformation to be positive somewhere, we separately present a variance inequality for the Fréchet median in \cref{thm:near:median}, which yields quadratic (and potentially faster) growth under some mild conditions, e.g., \cref{cor:varinequ:median}. It improves upon the recent work \cite[Theorem 2.3]{minsker2023geometric} in Euclidean spaces and shows a general result in Hadamard spaces. \cref{thm:near:median} can also be applied to eventually affine transformations under certain circumstances as shown in \cref{thm:affinereduction}.
    \item
    The Fréchet median may be non-unique. This case is characterized in \cref{thm:medianConcentration}.
    On Riemannian manifolds, \cite{yang10} shows that the Fréchet median is unique if it is not concentrated on a geodesic. A similar result is known in Hilbert and Banach spaces \cite{Kemperman1987Median}. This criterion (the distribution not being concentrated on a geodesic) is not enough in Hadamard spaces in general. Instead, we require that the random variable is not concentrated on a union of geodesics that all intersect in a common geodesic segment of positive length. Only in non-branching spaces is no concentration on geodesics sufficient for uniqueness of the Fréchet median. These criteria are shown in \cref{cor:unqiueCriteriaGeod}. An alternative sufficient condition for uniqueness is the convexity of the support of the distribution, see \cref{cor:convexUnique}.
    \item
    Finally, \cref{thm:median:geodesic} presents a variance inequality for the Fréchet median when concentrated on a geodesic, a case that is not covered by \cref{thm:near:median}.
\end{enumerate}
\subsection{Outline}
We describe our basic setup and the set of transformations $\setcc$ in section \ref{sec:preliminaries}. Then we consider properties of the variance functional in general metric spaces in section \ref{sec:varineq:general}. In section \ref{sec:quadLowerBound}, we discuss a property called $\mc G$-convexity, which allows us to derive lower bounds on distance functions in Hadamard spaces. With these technical tools, we derive variance inequalities for $\tran$-Fréchet means and medians in Hadamard spaces in sections \ref{sec:varineq:hadmard} and \ref{sec:varineq:hadmard:median}, respectively. In these last sections, we also discuss uniqueness properties.

\section{Preliminaries}\label{sec:preliminaries}
\subsection{Basic Setup}
Throughout the entire article, we will assume following setup (in particular, the meaning of the symbols $\mc Q, d, \Pr, \E, Y, o$) without further mentioning it: Let $(\mc Q, d)$ be a nonempty metric space. For $q,p\in\mc Q$, we denote $\ol{q}{p} := d(q, p)$. This metric space is equipped with its Borel-$\sigma$-algebra. Let $(\Omega, \Sigma_\Omega, \Pr)$ be a probability space. The expectation of measurable functions $X \colon \Omega \to \R$ is denoted as $\Ex{X}$ if it exists. Let $Y$ be a measurable function $Y \colon \Omega \to \mc Q$, i.e., a $\mc Q$-valued random variable. Furthermore, we fix an arbitrary reference point $o\in\mc Q$.
\subsection{Nondecreasing, Convex Functions with Concave Derivative}\label{sec:ncfcd}
We study $\tran$-Fréchet means, which minimize the variance functional  $q \mapsto \Ex{\tran(\ol Yq)-\tran(\ol Yo)}$ in $\mc Q$, where $\tran\colon\Rp\to\R$ is a nondecreasing convex function with concave derivative, i.e., $\tran\in\setcc$ according to following definition.
\begin{definition}
	Let $\setcc$ be the set of nondecreasing convex functions $\tran\colon\Rp\to\R$ that are differentiable on $\Rpp$ with concave derivative $\dtran$. We extend the domain of $\dtran$ to $\Rp$ by setting  $\dtran(0) := \lim_{x\searrow0} \dtran(x)$, which exists as $\dtran$ is nonnegative and nondecreasing.
\end{definition}
Requiring differentiability of $\tran$ is not restrictive, as this is implied by convexity for Lebesgue almost all $x\in\Rpp$.
For technical reasons it is often more convenient to work with $\setcciz\subset\setcc$, the subset of strictly increasing functions $\tran\in\setcc$ with $\tran(0) = 0$,
\begin{align}
	\setcciz
	&:=
	\setByEle{\tran\in\setcc}{\tran(0) = 0 \text{ and } \forall x\in\Rpp\colon \dtran(x)>0}
	\\&=
	\setByEle{x \mapsto \tran(x)-\tran(0)}{\tran\in\setcc}\setminus\cb{x\mapsto 0}
	\eqfs
\end{align}
This is not restrictive, as we essentially only exclude constant functions, for which most results are trivial anyway.
To be able to talk about derivatives of $\tran\in\setcc$ at $0$ and second derivatives, let us recall the definition of the one-sided derivatives.
\begin{notation}
	Let $A\subset \R$ and $f\colon A \to \R$. Let $x_0\in A$ such that there is $\epsilon>0$ such that $(x_0-\epsilon, x_0] \subset A$. Then denote the left derivative of $f$ at $x_0$ as $\partial_- f(x_0) := f\ld(x_0) := \lim_{x \nearrow x_0} \frac{f(x) - f(x_0)}{x-x_0}$ if the limit exists. Similarly, for $x_0\in A$ with $\epsilon>0$ such that $[x_0, x_0+\epsilon) \subset A$, we denote the right derivative of $f$ at $x_0$ as $\partial_+ f(x_0) := f\rd(x_0) := \lim_{x \searrow x_0} \frac{f(x) - f(x_0)}{x-x_0}$ if the limit exists.
\end{notation}
First we show some basic continuity properties of functions in $\setcc$ and existence of one-sided derivatives. Proofs omitted from this section can be found in the appendix \ref{apendix:omitted:ncfcd}.
\begin{lemma}\label{lmm:tran:continuous}
	Let $\tran\in\setcc$. Then
	\begin{enumerate}[label=(\roman*)]
		\item $\tran$ and $\dtran$ are continuous on $\Rp$,
		\item the left and right derivatives of $\dtran$ exist on $\Rpp$, they are nonincreasing, and $\ddltran(x)\geq\ddrtran(x)$ for all $x\in\Rpp$,
		\item $\tran\rd(0)$ exists and $\tran\rd(0) = \dtran(0)$.
	\end{enumerate}
\end{lemma}
Using a Taylor approximation and the convexity and concavity properties of $\tran\in\setcc$, we can see that these functions are, in some sense, between affine and quadratic: If $\tran\in\setcc$ is three times continuously differentiable, we have
\begin{equation}\label{eq:linquad}
	\tran(x_0) + x \dtran(x_0) \leq \tran(x_0 + x) \leq \tran(x_0) +  x \dtran(x_0) + \frac12 x^2 \ddtran(x_0)
\end{equation}
for $x_0\in\Rpp, x\in\Rp$. Fixing $x_0$ in \eqref{eq:linquad} shows a linear lower bound and quadratic upper bound on $x \mapsto \tran(x_0 + x)$.

For $\alpha>0$, let $\tran_\alpha := (x \mapsto x^\alpha)$. Then  $\tran_\alpha\in\setcc$ if and only if $\alpha\in[1,2]$. Further members of $\setcc$ include the Huber loss $\tran_{\ms h,\delta}$ \eqref{eq:huber}, the Pseudo-Huber loss $\tran_{\ms{ph},\delta}$ \eqref{eq:pseudohuber}, as well as $x\mapsto \log(\cosh(x))$.
The set $\setcc$ is a real convex cone in the following sense: If $\tran,\tilde\tran\in\setcc$, $b,\tilde b\in\Rp$, then $(x\mapsto b\tran(x)+\tilde b\tilde\tran(x))\in\setcc$.

The next two lemmas show bounds on the difference of a function in $\setcc$ evaluated at different points. These will be useful in the proofs of the main results.
\begin{lemma}\label{lmm:ccdiff}
	Let $\tran\in\setcc$.
	Let $x,y\in\Rp$, $x \neq y$.
	Then
	\begin{equation}
		\frac{\dtran(x)+\dtran(y)}2 \leq \frac{\abs{\tran(x) - \tran(y)}}{\abs{x-y}} \leq \dtran\brOf{\frac{x+y}{2}}
		\eqfs
	\end{equation}
\end{lemma}
\begin{lemma}\label{lmm:tranlower}
		Let $\tran\in\setcciz$.
		Let $x,y\in\Rp$.
		Then
		\begin{equation}
			\tran(x) + \tran(y)  \leq \tran(\abs{x-y}) + 2 y \dtran(x)
			\eqfs
		\end{equation}
\end{lemma}
\begin{remark}
	In the proofs of \cref{lmm:ccdiff} and \cref{lmm:tranlower} as well as upcoming proofs, we make extensive use of the properties of $\tran\in\setcc$. In particular, the facts that $\tran$, $\dtran$ are nondecreasing and $\dtran$ is subadditive (follows from nonnegativity and concavity, see \cref{lmm:tranconcave}) are applied again and again. This is a technical reason for not considering a larger class of transformations than $\setcc$.
\end{remark}
\section{The Variance Functional in General Metric Spaces}\label{sec:varineq:general}
In this section, we show asymptotic bounds for the variance functional $q\mapsto \Eof{\tran(\ol Yq)-\tran(\ol Yp)}$ in general metric spaces (\cref{thm:varineq:general}). The results are stated for an arbitrary reference point $p\in\mc Q$. But the case $p = m$ where $m$ is a $\tran$-Fréchet mean of the $\mc Q$-valued random variable $Y$ is most illustrating for the purpose of this article. In this general setting, we are not able to derive variance inequalities for $q$ close to $m$. Only after restriction to Hadamard spaces in section \ref{sec:varineq:hadmard} will we be able to find fully satisfying results. Nonetheless, the results here give us a first bound on the variance functional and will be used for showing that the set of $\tran$-Fréchet means of $Y$ is bounded (\cref{lmm:closedBoundedConvex}).

We start here with a trivial statement, which is useful to show that the main result, \cref{thm:varineq:general}, is optimal in some sense.
\begin{proposition}\label{prp:point:trivial}
	Let $q,p \in \mc Q$.
	\begin{enumerate}[label=(\roman*)]
		\item Let $\tran\in \setcciz$. Assume $Y = p$ almost surely. Then
		\begin{equation}
			\Ex*{\tran(\ol Yq) - \tran(\ol Yp)} = \tran(\ol qp)
			\eqfs
		\end{equation}
		\item We have
		\begin{equation}
			\Ex*{\ol Yq - \ol Yp} \geq \ol qp\, \br{\Prof{Y = p} - \Prof{Y \neq p}}
			\eqfs
		\end{equation}
	\end{enumerate}
\end{proposition}
\begin{proof}
	We have
	\begin{equation}
		\Ex*{\br{\ol Y{q} - \ol Yp}\ind_{\cb{Y=p}}} = \ol qp\, \PrOf{Y=p}
	\end{equation}
	and the triangle inequality implies
	\begin{equation}
		\abs{\Ex*{\br{\ol Y{q} - \ol Yp}\ind_{\cb{Y\neq p}}}} \leq \ol qp\, \PrOf{Y\neq p}
		\eqfs
	\end{equation}
\end{proof}
In the following, we will typically make the assumption that $\Ex{\dtran(\ol Yo)}$ is finite. Recall that $o\in\mc Q$ is a fixed arbitrary element of the metric space. This assumption ensures that the variance functional is well-defined and finite:
\begin{proposition}\label{prp:moment}
	Let $\tran\in\setcc$.
    \begin{enumerate}[label=(\roman*)]
        \item We have $\Ex{\dtran(\ol Yo)} < \infty$ if and only if $\Ex{\dtran(\ol Yq)} < \infty$ for all $q\in\mc Q$.
        \item
        Assume $\Ex{\dtran(\ol Yo)} < \infty$.
        For $q,p\in\mc Q$, define
        \begin{equation}\label{eq:altmetricdef}
            D(q,p) := \Ex*{\abs{\tran(\ol Yq) - \tran(\ol Yp)}}
            \eqfs
        \end{equation}
        Then, $(\mc Q, D)$ is a pseudometric space, i.e.,
        \begin{enumerate}[label=(\alph*)]
            \item $D(q,p) \in \Rp$,
            \item $D(q,q) = 0$,
            \item $D(q,p) = D(p,q)$,
            \item $D(q, p) \leq D(q, z) + D(z, p)$,
        \end{enumerate}
        for all $q,p,z\in \mc Q$.
        \item
        Assume $\Ex{\dtran(\ol Yo)} < \infty$.
        Let $r\in\Rpp$. Then there is a constant $C\in\Rpp$ such that, for all $q,p\in \ball or$,
        \begin{equation}\label{eq:equivalentmetric}
            D(q, p) \leq C\, \ol qp
        \end{equation}
        with $D$ as in \eqref{eq:altmetricdef}. Moreover, we can set $C := \Ex{\dtran(\ol Yo)} + \dtran(r)$ in \eqref{eq:equivalentmetric}.
    \end{enumerate}
\end{proposition}
\begin{proof}\mbox{ }
    \begin{enumerate}[label=(\roman*)]
        \item
        Let $q\in\mc Q$.
        By \cref{lmm:tranconcave}, $\dtran$ is subadditive.
        Together with the triangle inequality, we obtain
        \begin{equation}
            \dtran(\ol Yq) \leq \dtran(\ol Yo) + \dtran(\ol qo)
            \eqfs
        \end{equation}
        Taking expectations and exchanging the arbitrary points shows the first claim.
        \item
        Combining the subadditivity of $\dtran$ with the triangle inequality and \cref{lmm:ccdiff} yields
        \begin{equation}
            \abs{\tran(\ol Yq) - \tran(\ol Yp)}
            \leq
            \abs{\ol Yq - \ol Yp} \dtran\brOf{\frac{\ol Yq + \ol Yp}2}
            \leq
            \ol qp \br{\dtran(\ol Yo) + \dtran\brOf{\frac{\ol qo + \ol po }2}}
            \eqfs
        \end{equation}
        After taking expectations, we obtain
        \begin{equation}\label{eq:altmetricbound}
            D(q, p) \leq \ol qp \br{\Ex{\dtran(\ol Yo)} + \dtran\brOf{\frac{\ol qo + \ol po }2}} < \infty\eqfs
        \end{equation}
        Clearly, $D(q,p)\geq 0$, $D(q,q) = 0$, and $D(q,p) = D(p,q)$. The triangle inequality follows from
        \begin{equation}
            \abs{\tran(\ol Yq) - \tran(\ol Yp)} \leq \abs{\tran(\ol Yq) - \tran(\ol Yz)} + \abs{\tran(\ol Yz) - \tran(\ol Yp)}
            \eqfs
        \end{equation}
        \item
        Follows directly from \eqref{eq:altmetricbound}.
    \end{enumerate}
\end{proof}
The somewhat technical inequalities derived in the next lemma allow us to prove the main result of this section, \cref{thm:varineq:general}, which provides bounds of the variance functional for points $q$ close to the reference point $p$ and $q$ far away from $p$.
\begin{lemma}\label{lmm:generalBound}
	Let $\tran\in\setcciz$.
	Assume $\Ex{\dtran(\ol Yo)} < \infty$.
	Let $q,p\in\mc Q$. 	Let $s\in\Rp$.
	\begin{enumerate}[label=(\roman*)]
		\item \label{lmm:generalBound:upper}Then
		\begin{align}
			\Ex*{\tran(\ol Yq) - \tran(\ol Yp)}
			\leq
			&\ \ol qp\, \Ex*{\dtran\brOf{\frac{\ol qp}{2} + \ol Yp } \indOfOf{[s,\infty)}{\ol Yp}}
			\\& +
			\tran(\ol pq + s) \PrOf{\ol Yp < s}
			\eqfs
		\end{align}
		\item \label{lmm:generalBound:upper:near}  Assume $\ol qp \leq s$ and $s>0$. Then
		\begin{align}
			\Ex*{\tran(\ol Yq) - \tran(\ol Yp)}
			\leq&\
			\PrOf{Y = p} \tran(\ol qp)
			\\& +
			\frac32 \, \ol qp\, \dtran\brOf{s}\PrOf{\ol Yp \in (0,s)}
			\\& +
			\ol qp\, \br{\frac{\ol qp}{2 s} + 1} \Ex*{\dtran\brOf{\ol Yp} \indOfOf{[s,\infty)}{\ol Yp}}
			\eqfs
		\end{align}
		\item \label{lmm:generalBound:lower} Assume $s \leq \ol qp$. Then
		\begin{align}
			\Ex*{\tran(\ol Yq) - \tran(\ol Yp)}
			\geq&\
			\Ex*{\br{\tran(\ol qp) - 2\,\ol qp\, \dtran(\ol Yp)}\indOfOf{[s,\infty)}{\ol Yp}}
			\\& +
			\br{\tran(\ol qp - s) - \tran(s)}\PrOf{\ol Yp < s}
			\eqfs
		\end{align}
	\end{enumerate}
\end{lemma}
\begin{proof}
	First note that  $\Ex{\dtran(\ol Yp)}$ is finite because of $\Ex{\dtran(\ol Yo)} < \infty$, see \cref{prp:moment}.
	\begin{enumerate}[label=(\roman*)]
		\item
		On one hand, \cref{lmm:ccdiff}, the triangle inequality, and $\dtran$ being nondecreasing imply
		\begin{align}
			\tran(\ol Yq) - \tran(\ol Yp)
			&\leq
			\ol qp\, \dtran\brOf{\frac{\ol Yq + \ol Yp}{2}}
			\\&\leq
			\ol qp\, \dtran\brOf{\frac{\ol qp}{2} + \ol Yp}
			\eqfs
		\end{align}
		On the other hand, if $\ol Yp < s$, then the triangle inequality and $\tran$ being nondecreasing imply
		\begin{align}
			\tran(\ol Yq) - \tran(\ol Yp)
			&\leq
			\tran(\ol Yp + \ol pq)
			\\&\leq
			\tran(s + \ol pq)
			\eqfs
		\end{align}
		Taking expectations of these two inequalities on $\ol Yp \geq s$ and $\ol Yp < s$, respectively, we obtain the desired result.
	\item
		We split $\Ex{\tran(\ol Yq) - \tran(\ol Yp)}$ into the three expectations on the events $Y = p$, $\ol Yp \in (0, s)$, and $\ol Yp \geq s$. If $Y=p$, then $\tran(\ol Yq) - \tran(\ol Yp) = \tran(\ol qp)$.
		In general, as before
		\begin{equation}
			\tran(\ol Yq) - \tran(\ol Yp) \leq \ol qp\, \dtran\brOf{\frac{\ol qp}2 + \ol Yp}
			\eqfs
		\end{equation}
		If $0 < \ol Yp < s$ and $\ol qp \leq s$, then, with \cref{lmm:tranconcave} \ref{lmm:tranconcave:factor},
		\begin{align}
			 \dtran\brOf{\frac{\ol qp}2 + \ol Yp}
			 &\leq
			  \dtran\brOf{\frac32 s}
			 \\&\leq
			 \frac32 \dtran\brOf{s}
			 \eqfs
		\end{align}
		If $\ol Yp \geq s$, then, again using \cref{lmm:tranconcave} \ref{lmm:tranconcave:factor},
		\begin{align}
			\dtran\brOf{\frac{\ol qp}2 + \ol Yp}
			&=
			\dtran\brOf{\br{\frac{\ol qp}{2\ol Yp} + 1}\ol Yp}
			\\&\leq
			\dtran\brOf{\br{\frac{\ol qp}{2s} + 1}\ol Yp}
			\\&\leq
			\br{\frac{\ol qp}{2s} + 1}\dtran\brOf{\ol Yp}
			\eqfs
		\end{align}
	\item
		By the triangle inequality and $\tran$ being nondecreasing, we have $\tran(\ol Yq) \geq \tran(|\ol Yp-\ol qp|)$.
		From this, on one hand, we obtain by \cref{lmm:tranlower},
		\begin{align}
			\tran(\ol Yq) - \tran(\ol Yp)
			&\geq
			\tran(\ol qp) - 2 \,\ol qp\, \dtran(\ol Yp)
			\eqfs
		\end{align}
		On the other hand, if $\ol Yp < s \leq \ol qp$, then, by the triangle inequality,
		\begin{align}
			\tran(\ol Yq) - \tran(\ol Yp)
			&\geq
			\tran(\ol qp-s) - \tran(s)
			\eqfs
		\end{align}
		Taking expectations of these two inequalities on $\ol Yp \geq s$ and $\ol Yp < s$, respectively, we obtain the desired result.
	\end{enumerate}
\end{proof}
\begin{notation}
For $p\in\mc Q$ and $r\in \Rpp$, we denote the \emph{open ball} with center $p$ and radius $r$ as $\ball pr := \setByEleInText{q\in \mc Q}{\ol qp < r}$. The \emph{diameter} of a set $A\subset\mc Q$ is $\diam(A) := \sup_{q,p\in A}\ol qp$.
\end{notation}
\begin{theorem}\label{thm:varineq:general}
	Let $\tran\in\setcciz$. Assume $\Ex{\dtran(\ol Yo)} < \infty$. Fix $p\in\mc Q$.
	\begin{enumerate}[label=(\roman*)]
		\item \label{thm:varineq:general:infty}
		Assume $\diam(\mc Q) = \infty$. Then
		\begin{equation}
			\liminf_{r \to \infty} \inf_{q\in\mc Q \setminus\ball or}  \frac{\Ex*{\tran(\ol Yq) - \tran(\ol Yp)}}{\tran(\ol qp)}
			=
			\limsup_{r \to \infty} \sup_{q\in\mc Q \setminus\ball or}  \frac{\Ex*{\tran(\ol Yq) - \tran(\ol Yp)}}{\tran(\ol qp)}
			=
			1
			\eqfs
		\end{equation}
		\item \label{thm:varineq:general:zero}
		Assume $p$ is an accumulation point of $\mc Q$. Then
		\begin{equation}
			\limsup_{r \to 0} \sup_{q\in\ball pr \setminus\cb{p}}  \frac{\Ex*{\tran(\ol Yq) - \tran(\ol Yp)}}{\ol qp} \leq \Ex{\dtran(\ol Yp)}
			\eqfs
		\end{equation}
	\end{enumerate}
\end{theorem}
\begin{proof}
	\mbox{ }
	\begin{enumerate}[label=(\roman*)]
		\item
		First assume that $T := \sup_{x\in\Rp}\dtran(x) < \infty$. Then, for $s\in[0,\infty)$, \cref{lmm:generalBound} \ref{lmm:generalBound:upper} yields
		\begin{equation}
			\Ex*{\tran(\ol Yq) - \tran(\ol Yp)}
			\leq
			\ol qp\, T \PrOf{\ol Yp \geq s}
			+
			\tran(\ol pq + s) \PrOf{\ol Yp < s}
			\eqfs
		\end{equation}
		For $\epsilon\in \Rpp$, there is $s = s_\epsilon\in \Rpp$, such that $\max(1,T) \PrOf{\ol Yp \geq s} \leq \epsilon$. Let $r = r_{s, \epsilon}$ be large enough so that $q\in\mc Q \setminus\ball or$ fulfills $s \leq \epsilon \, \ol qp$. Then
		\begin{equation}
			\ol qp\, T \PrOf{\ol Yp \geq s}
			+
			\tran(\ol pq + s) \PrOf{\ol Yp < s}
			\leq
			\epsilon \, \ol qp + \tran((1+\epsilon)\ol qp)
		\end{equation}
		for all $q\in\mc Q \setminus\ball or$.
		Dividing by $\tran(\ol qp)$ and using  \cref{lmm:limittran:finiteT} with $0 < T < \infty$ yields
		\begin{equation}
			\limsup_{r \to \infty} \sup_{q\in\mc Q \setminus\ball or}\frac{\epsilon \, \ol qp + \tran((1+\epsilon)\ol qp)}{\tran(\ol qp)}
			=
			\frac{\epsilon}{T} + 1+\epsilon
			\eqfs
		\end{equation}
		As $\epsilon\in\Rpp$ can be chosen arbitrarily, we obtain
		\begin{equation}
			\limsup_{r \to \infty} \sup_{q\in\mc Q \setminus\ball or}\frac{\Ex*{\tran(\ol Yq) - \tran(\ol Yp)} }{\tran(\ol qp)} \leq 1
			\eqfs
		\end{equation}
		To show a corresponding lower bound, we use \cref{lmm:generalBound} \ref{lmm:generalBound:lower} to obtain
		\begin{equation}\label{eq:lower:step1}
			\Ex*{\tran(\ol Yq) - \tran(\ol Yp)}
			\geq
			\br{\tran(\ol qp) - 2\,\ol qp\, T}\PrOf{\ol Yp \geq s}
			+
			\br{\tran(\ol qp - s) - \tran(s)}\PrOf{\ol Yp < s}
			\eqfs
		\end{equation}
		For $\epsilon\in\Rpp$, we can choose $s= s_\epsilon$ large enough so that $\Prof{\ol Yp \geq s} < \epsilon$, and then choose $r = r_{s, \epsilon}$ large enough so that $s \leq \epsilon \, \ol qp$ for all $q\in\mc Q \setminus\ball or$.
		Hence, \eqref{eq:lower:step1} yields
        \begin{equation}\label{eq:lower:step2}
            \Ex*{\tran(\ol Yq) - \tran(\ol Yp)}
            \geq
            \br{1-\epsilon}\tran\brOf{\br{1-\epsilon}\ol qp}
            - 2 T \epsilon\,\ol qp
            - \tran\brOf{\epsilon\,\ol qp}
            \eqfs
        \end{equation}
        Dividing by $\tran(\ol qp)$ and using  \cref{lmm:limittran:finiteT} with $0 < T < \infty$ yields
        \begin{equation}
            \liminf_{r \to \infty} \inf_{q\in\mc Q \setminus\ball or} \frac{\br{1-\epsilon}\tran\brOf{\br{1-\epsilon}\ol qp}
            - 2 T \epsilon\,\ol qp
            - \tran\brOf{\epsilon\,\ol qp}}{\tran(\ol qp)} =  (1-\epsilon)^2 - 2\epsilon - \epsilon
            \eqfs
        \end{equation}
		With $\epsilon\in\Rpp$ being arbitrary, we obtain the asymptotic lower bound
		\begin{equation}\label{eq:prp:limit:lower:finite}
			\liminf_{r \to \infty} \inf_{q\in\mc Q \setminus\ball or}   \frac{\Ex*{\tran(\ol Yq) - \tran(\ol Yp)}}{\tran(\ol qp)}
			\geq
			1
			\eqfs
		\end{equation}
		Now assume that $T = \infty$. Then the lower bound \eqref{eq:prp:limit:lower:finite} follows immediately from \cref{lmm:generalBound} \ref{lmm:generalBound:lower} with $s = 0$ as $x/\tran(x) \xrightarrow{x\to\infty} 0$ (as $T=\infty$). For the upper bound, \cref{lmm:generalBound} \ref{lmm:generalBound:upper} yields
		\begin{align}
			&\Ex*{\tran(\ol Yq) - \tran(\ol Yp)}
			\\&\leq
			\ol qp\, \dtran\brOf{\frac{\ol qp}{2}} \PrOf{\ol Yp \geq s} + \ol qp\, \Ex*{\dtran\brOf{\ol Yp}}
			+
			\tran(\ol pq + s) \PrOf{\ol Yp < s}
			\eqcm
		\end{align}
		where we used that $\dtran$ is subadditive (\cref{lmm:tranconcave}). As $x \dtran(x) \leq 2\tran(x)$ (\cref{lem:aux}), and $\tran(x)/x \xrightarrow{x\to\infty} \infty$ (as $T=\infty$), the expression is dominated by $\tran(\ol pq + s) \Prof{\ol Yp < s}$ for large $s$ and larger $\ol qp$. Thus, as $\tran(x + s)/\tran(x) \xrightarrow{x\to\infty} 1$ (see \cref{lmm:limittran}), we have
		\begin{equation}
			\limsup_{r \to \infty} \sup_{q\in\mc Q \setminus\ball or} \frac{\Ex*{\tran(\ol Yq) - \tran(\ol Yp)}}{\tran(\ol qp)}
			\leq
			1
			\eqfs
		\end{equation}
		\item
		Fix $\epsilon > 0$. Choose $s= s_\epsilon \in (0,1]$ small enough so that $\PrOf{\ol Yp \in (0,s)} < \epsilon$. For $q\in\mc Q$ with $\ol qp\leq 2 \epsilon s$, we have,
		by \cref{lmm:generalBound} \ref{lmm:generalBound:upper:near},
		\begin{align}
			&\Ex*{\tran(\ol Yq) - \tran(\ol Yp)}
			\\&\leq
			\PrOf{Y = p} \tran(\ol qp)
			\ +\
			\frac32 \epsilon \, \ol qp\, \dtran\brOf{1}
			\ +\
			\ol qp\, \br{1 + \epsilon} \Ex*{\dtran\brOf{\ol Yp} \ind_{Y\neq p}}
			\eqfs
		\end{align}
		Thus,
		\begin{align}
			&\limsup_{r \to 0} \sup_{q\in\ball pr \setminus\cb{p}} \frac{\Ex*{\tran(\ol Yq) - \tran(\ol Yp)}}{\ol qp}
			\\&\leq
			\PrOf{Y = p} \br{\limsup_{r \to 0} \sup_{q\in\ball pr\setminus\cb{p}}  \frac{\tran(\ol qp)}{\ol qp}}
			\ +\
			\Ex*{\dtran\brOf{\ol Yp} \ind_{Y\neq p}}
			\eqfs
		\end{align}
		As $\tran(0) = 0$, $\lim_{x\to 0} \frac{\tran(x)}{x} = \dtran(0)$. Thus, we have shown the desired result.
	\end{enumerate}
\end{proof}
\cref{thm:varineq:general} shows that, for $q$ far enough away from a reference point, the variance functional $q\mapsto \Ex{\tran(\ol Yq) - \tran(\ol Yp)}$ looks like $q\mapsto \tran(\ol qp)$.
For $q$ very close to the reference point $p$, the variance functional grows at most linearly, and can grow linearly, see \cref{prp:point:trivial}.
Particular interest is in a lower bound for $q$ near the reference point $p$, which is not described by the theorem.
Note that, for arbitrary $p$, the variance functional can be negative for some $q$ close to $p$. By restricting $p$ to be a $\tran$-Fréchet mean of $Y$ and the geometry of the metric space to have nonpositive curvature, we will be able to paint an almost complete picture of the growth behavior of the variance functional in section \ref{sec:varineq:hadmard}. In the next section, we develop the tools we need for this task.

\section{Quadratic Lower Bound of Distance Functions in Hadamard Spaces}\label{sec:quadLowerBound}
\cite[Section 4.4]{burago01} discusses distance functions in Hadamard spaces (and others) and a concept called $\mc E$-convexity. We modify and extend their treatment slightly in section \ref{ssec:gconv} and \ref{ssec:distInHad}, and show its connection to strong convexity in section \ref{ssec:strongconvexity}. With the concepts described there, we are able to obtain quadratic lower bounds of distance functions in Hadamard spaces in sections \ref{ssec:quadLower} (\cref{lmm:semitaylor}), which are the major tool for deriving variance inequalities in Hadamard spaces in sections \ref{sec:varineq:hadmard} and \ref{sec:varineq:hadmard:median}.
Proofs omitted from this section can be found in appendix \ref{apendix:omitted:quadLowerBound}.
\subsection{$\mc G$-convexity}\label{ssec:gconv}
\begin{definition}\label{def:Fconvex}
	Let $I\subset\R$. Let $\mc G$ be a set of functions $I \to \R$.
	A function $f\colon I\to\R$ is called \emph{$\mc G$-convex} if and only if for every $t_0\in I$ there is a function $g\in\mc G$ such that $g(t_0) = f(t_0)$ and $g(t) \leq f(t)$ for all $t\in I$. In this case, the function $g$ is called \emph{$\mc G$-tangent} of $f$ at $t_0$.
\end{definition}
If $\mc G$ is the set of affine functions, then a function is $\mc G$-convex if and only if it is convex.
For the rest of the article, we fix $\mc G$ to be
\begin{equation}
	\mc G := \setByEle{t\mapsto \sqrt{(t-t_0)^2+h^2}}{t_0 \in \R, h\geq 0}
	\eqfs
\end{equation}
The next proposition helps us to better understand this set $\mc G$.
\begin{proposition}[On the set $\mc G$]\label{prp:gconv:set}\mbox{ }
	\begin{enumerate}[label=(\roman*)]
		\item\label{prp:gconv:set:parabola}
		The set $\mc G$ is the set of square roots of nonnegative quadratic polynomials with $1$ as quadratic coefficient, i.e.,
		\begin{equation}
			\mc G = \setByEle{t\mapsto \sqrt{f(t)}}{f(t) = t^2 + at + b\,,\ a\in\R,b\in\left[\frac{a^2}4, \infty \right)}
			\eqfs
		\end{equation}
		\item
		Let $(\mc Q, d)$ be a Hilbert space of at least dimension $2$. Then $\mc G$ is the set of distance functions between a point and a line, i.e.,
		\begin{equation}
			\mc G = \setByEle{t \mapsto \ol{y}{\gamma(t)}}{y \in \mc Q \ \text{ and } \ \gamma(t) = u + v t\,,\ u,v\in\mc Q\,,\ \normof{v} = 1}
			\eqfs
		\end{equation}
	\end{enumerate}
\end{proposition}
When generalizing from Hilbert or Euclidean spaces to Hadamard spaces, the distance functions can look different from the functions in $\mc G$. But they retain some essential properties of the Euclidean distance functions: As stated below in \cref{lem:distFun}, distance functions in Hadamard spaces are $\mc G$-convex. The study of $\mc G$-convex functions is central for proving our main results. The next two propositions provide properties of the functions in $\mc G$ and of $\mc G$-convex functions.

We say a function is \emph{1-Lipschitz} if it is Lipschitz-continuous with Lipschitz constant $1$. The term \emph{nonexpanding} is also used in the literature for this property.
\begin{proposition}[On functions in $\mc G$]\label{prp:gconv:elem}\mbox{ }
	\begin{enumerate}[label=(\roman*)]
		\item
		Every $g \in \mc G$ is $\mc G$-convex.
		\item \label{prp:gconv:elem:equal}
		Let $g_1, g_2 \in \mc G$. If there are $s,t\in\R$, $s\neq t$ such that $g_1(s) = g_2(s)$ and $g_1(t) = g_2(t)$, then $g_1 = g_2$.
		\item \label{prp:gconv:elem:noWiggle}
		Let $g_1, g_2 \in \mc G$. If there are $r,s,t\in\R$, $r<s<t$ such that $g_1(r) \leq g_2(r)$, $g_1(s) \geq g_2(s)$, and $g_1(t) \leq g_2(t)$, then $g_1 = g_2$.
		\item \label{prp:gconv:elem:intersect}
		Let $t_1,t_2\in\R$. If $f\colon\R\to\Rp$ is $1$-Lipschitz with $f(t_1) + f(t_2) \geq \abs{t_1 - t_2}$, then there is a function $g\in\mc G$ such that $g(t_1) = f(t_1)$ and $g(t_2) = f(t_2)$.
	\end{enumerate}
\end{proposition}
\begin{proposition}[On $\mc G$-convex functions]\label{prp:gconv:fun}
	Let $I\subset \R$ be convex. Let $f\colon I\to\R$ be $\mc G$-convex.  Then
	\begin{enumerate}[label=(\roman*)]
		\item\label{prp:gconv:fun:nonneg} $f$ is nonnegative,
		\item\label{prp:gconv:fun:Lipschitz} $f$ is $1$-Lipschitz,
		\item\label{prp:gconv:fun:convex} $f$ is convex,
		\item\label{prp:gconv:fun:addLower} for all $t_1, t_2 \in I$, we have $f(t_1) + f(t_2) \geq \abs{t_1 - t_2}$,
		\item\label{prp:gconv:fun:distOnIsAbs}  if there is $t_0\in I$ such that $f(t_0) = 0$, then $f(t) = \abs{t-t_0}$ for all $t\in I$.
	\end{enumerate}
\end{proposition}
\begin{notation}\label{not:leftright}\mbox{ }
  For $I \subset \R$ convex, denote $\woinf{I} := I \setminus \cb{\inf I}$ and $\wosup{I} := I \setminus \cb{\sup I}$. For a function $f\colon I\to\R$, denote $\semiDerivs f(t) = [\min(f\ld(t), f\rd(t)), \max(f\ld(t), f\rd(t))]$ for $t\in \woinf I \cap \wosup I$. If $\inf I > -\infty$, set $\semiDerivs f(\inf I) = \cb{f\rd(\inf I)}$, and, if $\sup I < \infty$, set $\semiDerivs f(\sup I) = \cb{f\ld(\sup I)}$, assuming the left and right derivatives exist on $\woinf I$ and $\wosup I$, respectively.
\end{notation}
\begin{proposition}[Characterization of $\mc G$-convex functions]\label{prp:gconv:characterization}
	Let $I\subset \R$ be convex.
	\begin{enumerate}[label=(\roman*)]
		\item\label{prp:gconv:characterization:aboveBelow}
			Let $f\colon I\to\R$ be $\mc G$-convex. Let $g\in\mc G$ and $t_1, t_2\in I$, $t_1 < t_2$ such that $g(t_i) = f(t_i)$, $i=1,2$. Then $g(t) \geq f(t)$ for all $t \in [t_1, t_2]$ and $g(t) \leq f(t)$ for all $t \in I \setminus[t_1, t_2]$.
		\item\label{prp:gconv:characterization:lower}
			Let $f\colon I\to\R$. Assume for every $t_1, t_2\in I$, $t_1 < t_2$, there is $g\in\mc G$ such that $g(t_i) = f(t_i)$, $i=1,2$ and $g(t) \geq f(t)$ for all $t \in [t_1, t_2]$. Then the right derivative of $f$ exists on $\wosup I$ and the left derivative on $\woinf I$. Let $t_0\in I$ and $v_0\in\semiDerivs f(t_0)$. Then
			\begin{align}
				s
				&\mapsto
				\sqrt{\br{s - t_0 + f(t_0)v_0}^2 + \br{1 - v_0^2} f(t_0)^2}
				\\& =
				\sqrt{f(t_0)^2 + 2 (s-t_0) f(t_0)v_0 + (s-t_0)^2}
			\end{align}
			is a $\mc G$-tangent of $f$ at $t_0$.
			In particular, $f$ is $\mc G$-convex.
	\end{enumerate}
\end{proposition}
\begin{proof}\mbox{ }
	\begin{enumerate}[label=(\roman*)]
		\item
		Assume there would be $t\in(t_1, t_2)$ with $g(t) < f(t)$. Let $\tilde g$ be the $\mc G$-tangent of $f$ at $t$. Then $\tilde g(t) = f(t) > g(t)$ and $\tilde g(t_i) \leq f(t_i) = g(t_i)$, $i=1,2$. By \cref{prp:gconv:elem} \ref{prp:gconv:elem:noWiggle}, that is not possible.

		Assume there would be $t\in I\setminus[t_1, t_2]$ with $g(t) > f(t)$. Without loss of generality, $t < t_1$. Then, by \cref{prp:gconv:fun} \ref{prp:gconv:fun:nonneg}, \ref{prp:gconv:fun:Lipschitz}, \ref{prp:gconv:fun:addLower} and \cref{prp:gconv:elem} \ref{prp:gconv:elem:intersect}, there is $\tilde g\in\mc G$ with $\tilde g(t) = f(t) < g(t)$, $\tilde g(t_2) = f(t_2) = g(t_2)$. Then, according to the first part of this item, which we just proved, $\tilde g(t_1) \geq f(t_1) = g(t_1)$. By \cref{prp:gconv:elem} \ref{prp:gconv:elem:noWiggle} that is not possible.
		\item
		For $t_1,t_2\in I$, $t_1 < t_2$ denote $g_{t_1, t_2} \in \mc G$ a function with $g_{t_1, t_2}(t_i) = f(t_i)$, $i=1,2$ and $g_{t_1, t_2}(t) \geq f(t)$ for all $t \in [t_1, t_2]$.
		Assume there would be $t\in\R\setminus[t_1, t_2]$ with $g_{t_1, t_2}(t) > f(t)$. Without loss of generality, $t < t_1$. Then we have $g_{t, t_2}\in\mc G$ with $g_{t, t_2}(t) = f(t) < g_{t_1, t_2}(t)$, $g_{t, t_2}(t_2) = f(t_2) = g_{t_1, t_2}(t_2)$ and $g_{t, t_2}(t_1) \geq f(t_1) = g_{t_1, t_2}(t_1)$ as $t_1 \in [t, t_2]$. By \cref{prp:gconv:elem} \ref{prp:gconv:elem:noWiggle} that is not possible. We will use this property later.

		Before, we state two further properties of $f$: The function $f$ is $1$-Lipschitz as $\abs{f(t_1)-f(t_2)} = \abs{g_{t_1, t_2}(t_1) - g_{t_1, t_2}(t_2)} \leq \abs{t_1 -t_2}$. Moreover, $f$ is convex because $f(t)\leq g_{t_1, t_2}(t)$ for all $t\in[t_1,t_2]$ and  $g_{t_1, t_2}$ is convex.

		Let $t_0\in I$. As $f$ is convex and Lipschitz, $\semiDerivs f(t_0)$ exists. For all $v_0\in \semiDerivs f(t_0)$, there are $(t_n^+)_{n\in\N}, (t_n^-)_{n\in\N} \subset I$ with $t_n^+ > t_n^-$ and $t_0\in[t_n^-, t_n^+]$ such that both sequences converge to $t_0$, and
		\begin{equation}
			\lim_{n\to\infty} \frac{f(t_n^+) - f(t_n^-)}{t_n^+ - t_n^-} = v_0
			\eqcm
		\end{equation}
		see \cref{lmm:middleDeriv}.
		Let $g_n\in\mc G$, such that $g_n(t_n^+) = f(t_n^+)$, $g_n(t_n^-) = f(t_n^-)$. Then $g_n(t) \geq f(t)$ for all $t \in [t_n^-, t_n^+]$ by assumption and $g_n(t) \leq f(t)$ for all $t \in I \setminus[t_n^-, t_n^+]$ as we have shown in the first paragraph of this proof.
		Let $s_n\in\R$, $h_n\in\Rp$ be the parameters of $g_n$, i.e.,
		\begin{equation}
			g_n(x) = \sqrt{(x-s_n)^2 + h_n^2}
			\eqfs
		\end{equation}
		Using \cref{lmm:gconv:solve}, we can explicitly calculate $s_n$ and $h_n$ and their limits $s_\infty := \lim_{n\to\infty} s_n$ and $h_\infty := \lim_{n\to\infty} h_n$: As $f$ is convex, its one-sided derivatives exits. Thus, using continuity of $f$, we obtain
		\begin{align}
			s_n &=
			\frac{\br{t_n^+}^2 - \br{t_n^-}^2 + f(t_n^-)^2 - f(t_n^+)^2 }{2 (t_n^+ - t_n^-)}
			\\&=
			\frac{t_n^+ + t_n^-}2  - \frac{f(t_n^+) + f(t_n^-)}2 \cdot \frac{f(t_n^+) - f(t_n^-)}{t_n^+ - t_n^-}
			\\&\xrightarrow{n\to\infty}
			t_0 - f(t_0)v_0
			\eqfs
		\end{align}
		For the other parameter, we obtain
		\begin{align}
			h_n^2 &=
			\frac14 \br{2\br{f(t_n^+)^2 + f(t_n^-)^2} - (t_n^+ - t_n^-)^2 - \br{\frac{f(t_n^-)^2 - f(t_n^+)^2}{t_n^+ - t_n^-}}^2}
			\\&=
			\frac{f(t_n^+)^2 + f(t_n^-)^2}2 - \br{\frac{t_n^+ - t_n^-}2}^2 - \br{\frac{f(t_n^+) + f(t_n^-)}2\cdot\frac{f(t_n^+) - f(t_n^-)}{t_n^+ - t_n^-}}^2
			\\&\xrightarrow{n\to\infty}
			\br{1 - v_0^2} f(t_0)^2
			\eqfs
		\end{align}
		Define the function
		\begin{equation}
			g_\infty(x) = \sqrt{(x-s_\infty)^2 + h_\infty^2}
			\eqfs
		\end{equation}
		Obviously, $g_\infty\in\mc G$. Let $x\in\R$. As $(s,h) \mapsto \sqrt{(x-s)^2 + h^2}$ is continuous, $g_\infty(x) = \lim_{n\to\infty} g_n(x)$. Furthermore, if $x\neq t_0$, there is $n_0\in\N$ such that $g_n(x) \leq f(x)$ for all $n \geq n_0$. Thus, $g_\infty(x) \leq f(x)$. Lastly, $g_\infty(t_0) = \lim_{n\to\infty} g_n(t_0) \geq f(t_0)$. Thus, $g_\infty\in\mc G$ is a $\mc G$-tangent of $f$ at $t_0$.
	\end{enumerate}
\end{proof}
A central property of $\mc G$-convex functions for the proof of the main result of this section, \cref{lmm:semitaylor}, is the following lower bound on the second derivative.
\begin{notation}\label{not:interior}\mbox{ }
   For a convex set $I\subset \R$, let $\mathring I$ be the \emph{interior interval}, $\mathring I := I\rd \cap I\ld$.
\end{notation}
\begin{lemma}\label{lmm:gconv:second:deriv}
    Let $I\subset \R$ be convex.
	Assume $f\colon I\to \R$ is $\mc G$-convex. Let $s\in \mathring I$ such that $f$ is twice differentiable at $s$. Then $f(s) > 0$ and
	\begin{equation}
		f\prr(s) \geq
		\frac{1-f\pr(s)^2}{f(s)}
		\eqfs
	\end{equation}
\end{lemma}
\begin{proof}
	If $f(s) = 0$, then $f$ is not differentiable at $s$ by \cref{prp:gconv:fun} \ref{prp:gconv:fun:distOnIsAbs}. Thus, $f(s) > 0$.
	Let $g$ be the $\mc G$-tangent of $f$ at $s$.
	Define $h(t) := f(t)-g(t)$. As $g(s) = f(s) > 0$, $g$ is twice differentiable at $s$. Thus, $h$ is twice differentiable at $s$.
	As $g$ is a $\mc G$-tangent of $f$ at $s$, we have $h(s) = 0$ and $h(t)\geq 0$ for all $t\in I$. Thus, as $s\in\mathring I$, $h\pr(s) = 0$ and $h\prr(s) \geq 0$. Thus, $f\pr(s) = g\pr(s)$ and $f\prr(s) \geq g\prr(s)$.
	By \cref{lmm:secondDerivInG}, $g\prr(s) g(s) = 1 - g\pr(s)^2$. Thus,
	\begin{equation}
		f\prr(s) \geq g\prr(s) = \frac{1 - g\pr(s)^2}{g(s)} = \frac{1 - f\pr(s)^2}{f(s)}
		\eqfs
	\end{equation}
\end{proof}
\subsection{Distance Functions in Hadamard Spaces}\label{ssec:distInHad}
We first recall the definition of a geodesic in Hadamard spaces and state some basic properties. In some context, one may want to distinguish between \textit{shortest paths} that are globally minimizing the length of a curve and \textit{geodesics} that are only required to minimize length locally. As our main focus is on Hadamard spaces, where these two terms are identical \cite[Theorem 9.2.2]{burago01}, we only use the term geodesic here. For a full treatment see, e.g., \cite{burago01}.
\begin{definition}\label{def:geodesic}
	Let $I\subset \R$ be convex.
	\begin{enumerate}[label=(\roman*)]
		\item A function $\gamma\colon I \to \mc Q$ is called \emph{geodesic} if and only if
		\begin{equation}
			\ol{\gamma(r)}{\gamma(t)} =  \ol{\gamma(r)}{\gamma(s)} + \ol{\gamma(s)}{\gamma(t)}
		\end{equation}
		for all $r,s,t\in I$ with $r<s<t$.
		\item Let  $\gamma\colon I \to \mc Q$ be a geodesic. If there is $L\in\Rp$ such that $\ol{\gamma(s)}{\gamma(t)} = L\abs{s-t}$ for all $s,t\in I$, then the geodesic is said to have \emph{constant speed}. If $L = 1$, we call $\gamma$ a \emph{unit-speed geodesic}. Denote by $\GeodUS$ the set of unit-speed geodesics in $\mc Q$.
		\item The metric space $(\mc Q, d)$ is called \emph{geodesic space}, if and only if each pair of points in $\mc Q$ is connected by a geodesic.
	\end{enumerate}
\end{definition}
\begin{proposition}\label{prp:geod:basics}
	Assume $(\mc Q, d)$ is Hadamard. Let $q,p\in\mc Q$.
	\begin{enumerate}[label=(\roman*)]
		\item There is a unique constant speed geodesic $\gamma_{q,p}\colon[0,1]\to\mc Q$ with $\gamma_{q,p}(0) = q$ and $\gamma_{q,p}(1) = p$.
		\item Let $t\in[0,1]$, then $(q,p)\mapsto \gamma_{q,p}(t)$ is continuous.
		\item Let $y\in\mc Q$ and $t\in[0,1]$. Then
		\begin{equation}
			\ol{y}{\gamma_{q,p}(t)}^2 \leq (1-t) \,\ol{y}{q}^2 + t\,\ol{y}{p}^2 - t(t-1)\,\ol{q}{p}^2
			\eqfs
		\end{equation}
	\end{enumerate}
\end{proposition}
See, e.g., \cite[Proposition 2.3]{sturm03} for a proof.
\begin{notation}\mbox{ }
	\begin{enumerate}[label=(\roman*)]
		\item For $q,p\in\mc Q$, denote $\geodft qp\in\GeodUS$ the unit speed geodesic $\geodft qp\colon[0, \ol qp] \to \mc Q$ from $\geodft qp(0) = q$ to $\geodft qp(\ol qp) = p$.
		\item If $\gamma$ is a geodesic, we denote its domain by $I_\gamma$, i.e., $\gamma\colon I_\gamma \to \mc Q$.
		\item For a geodesic $\gamma$, we abuse notation and sometimes treat $\gamma$ as its image, i.e., as the set $\gamma(I_\gamma) = \setByEleInText{\gamma(t)}{t \in I_\gamma}$, and write $q \in \gamma$ for $q \in \gamma(I_\gamma)$ and  $\gamma \subset A$ for $A\subset \mc Q$ with $\gamma(I_\gamma) \subset A$.
	\end{enumerate}
\end{notation}
\begin{lemma}\label{lem:distFun}
	Assume $(\mc Q, d)$ is Hadamard. Let $\gamma\in\GeodUS$. Let $y\in\mc Q$. Define the function
	\begin{equation}
		\ol{y}{\gamma} \colon I_\gamma \to \R,\, t\mapsto \ol{y}{\gamma(t)}\eqfs
	\end{equation}
	Then $\ol{y}{\gamma}$ is $\mc G$-convex.
\end{lemma}
\begin{proof}
	Let $t_1, t_2\in I_\gamma$ with $t_1 < t_2$. By \cref{prp:geod:basics}, there is a constant speed geodesic $\tilde\gamma\colon[0, 1]\to\mc Q$ with $\tilde\gamma(0) = \gamma(t_1)$ and $\tilde\gamma(1) = \gamma(t_2)$, with
	\begin{equation}\label{eq:hadamard:geodesic}
		\ol{y}{\tilde\gamma(t)}^2
		\leq
		(1-t) \, \ol{y}{\tilde\gamma(0)}^2 +
		t  \,\ol{y}{\tilde\gamma(1)}^2 -
		t(1-t)\,\ol{\tilde\gamma(0)}{\tilde\gamma(1)}^2
		\eqfs
	\end{equation}
	As such a geodesic is unique, we have $\tilde\gamma(t) = \gamma(t_1 + t (t_2 - t_1))$. Set $s = t_1 + t (t_2 - t_1)$. Then $t = (s-t_1)/(t_2 - t_1)$ and
	\begin{align}
		\ol{y}{\gamma}(s)^2
		&=
		\ol{y}{\gamma(s)}^2
		\\&\leq
		\br{1-\frac{s-t_1}{t_2 - t_1}}  \ol{y}{\gamma(t_1)}^2 +
		\frac{s-t_1}{t_2 - t_1}  \ol{y}{\gamma(t_2)}^2 -
		\br{s - t_1}\br{t_2 - s}
		\\&=: f(s)
	\end{align}
	for $s\in[t_1, t_2]$ with equality at the boundaries.
	The function $f$ is a nonnegative, quadratic polynomial with $1$ as coefficient of the squared term. Thus, $\sqrt{f} \in \mc G$ by \cref{prp:gconv:set} \ref{prp:gconv:set:parabola}. \cref{prp:gconv:characterization} \ref{prp:gconv:characterization:lower} implies that $s \mapsto \ol{y}{\gamma}(s)$ is $\mc G$-convex.
\end{proof}
\subsection{Strong Convexity}\label{ssec:strongconvexity}
The notion of $\mc G$-convexity is related to \textit{strong convexity}, a term possibly more common in the literature \cite{poljak66,merentes2010remarks, Nikodem2014, zhou2018fenchel}. Although we will not use strong convexity later, we want to make the connection clear.
In the following $\norm$ denotes the Euclidean norm in $\R^k$, $k\in\N$.
\begin{definition}
	Let $k\in\N$, $D \subset \R^k$ be convex, $f\colon D \to \R$, and $a\in\Rpp$. The function $f$ is \emph{strongly convex with modulus $a$} if and only if, for all $x,y\in D$ and $t\in[0,1]$, we have
	\begin{equation}
		f(t x + (1-t)y) \leq t f(x) + (1-t) f(y) - a t (1-t) \normof{x-y}^2
		\eqfs
	\end{equation}
\end{definition}
Let $k\in\N$, $D \subset \R^k$ be convex, $f\colon D \to \R$. A \textit{subgradient} of $f$ at $x_0\in D$ is a vector $v\in\R^k$ such that $f(x) - f(x_0) \geq v\tr (x-x_0)$ for all $x\in\R^k$. The set of subgradients at $x_0$ is called \textit{subdifferential} and denoted as $\partial f(x_0)$. It is a convex set.
\begin{proposition}[Characterization of strong convexity]\label{prp:strongconv:chara}
	Let $k\in\N$, $D \subset \R^k$ be convex, $f\colon D \to \R$, and $a\in\Rpp$. Then, the following statements are equivalent:
	\begin{enumerate}[label=(\roman*)]
		\item $f$ is strongly convex with modulus $a$.
		\item $x\mapsto f(x) - a \normof{x}^2$ is convex.
		\item $f(x) \geq f(x_0) + v\tr(x-x_0) + a\normof{x-x_0}^2$ for all $x,x_0\in D$ and $v\in\partial f(x_0)$.
		\item $\br{u-v}\tr(x-z) \geq 2 a \normof{x-z}^2$ for all $x,z\in D$ and $u\in\partial f(x)$, $v\in\partial f(z)$.
	\end{enumerate}
\end{proposition}
See, e.g., \cite{zhou2018fenchel} for a proof.
\begin{proposition}[$\mc G$-convexity vs strong convexity]\label{prp:gconvVSstrong}
	Let $I\subset \R$ be convex.
	\begin{enumerate}[label=(\roman*)]
		\item Let $f\colon I\to\R$. Assume $f$ is $\mc G$-convex. Then $f^2$ is strongly convex with modulus $1$.
		\item Let $f\colon I\to\R$. Assume $f$ is nonnegative and strongly convex with modulus $1$, and $\sqrt{f}$ is $1$-Lipschitz. Then $\sqrt{f}$ is $\mc G$-convex.
	\end{enumerate}
\end{proposition}
See appendix \ref{apendix:omitted:quadLowerBound} for a proof.
According to \cite[Theorem 9.2.19]{burago01} the property that squared distance functions $\ol y\gamma^2$ for $y\in\mc Q$, $\gamma\in\GeodUS$ are strongly convex with modulus $1$ characterizes Hadamard spaces among complete metric spaces with so-called \emph{strictly intrinsic metric} \cite[Definition 2.1.10]{burago01}.
\subsection{Quadratic Lower Bound}\label{ssec:quadLower}
For $\mc G$-convex functions $f$, we derive a quadratic lower bound of $f$ and of $\tran \circ f$ for $\tran\in\setcc$ in \cref{lmm:semitaylor} below. The inequality for  $\tran \circ f$ depends on a second derivative, which is bounded in the next lemma.
\begin{lemma}\label{lmm:secondDeriv}
	Let $\tran\in\setcc$.
	Let $I\subset \R$ be convex.
	Let $f\colon I\to\R$ be $\mc G$-convex.
	Let $t\in \mathring I$.
	Assume that $f$ is twice differentiable at $t$.
	Then
	\begin{enumerate}[label=(\roman*)]
		\item $(\tran \circ f)^{\prime\ominus}  (t) \geq \ddrtran (f(t))$ and
		\item $	(\tran \circ f)^{\prime\oplus}  (t) \geq \ddrtran (f(t))$.
	\end{enumerate}
\end{lemma}
\begin{proof}
	If $\tran$ is constant, the result is trivial. Assume $\tran$ is not constant. Then $\dtran(x) > 0$ for all $x \in \Rpp$ as $\dtran$ is concave, nondecreasing, and nonnegative.

	If $f(t) = 0$, then $f(s) = \abs{s-t}$, which is not differentiable at $t$. As we assume that $f$ is twice differentiable at $t$, we have $f(t) \neq 0$. The product rule applies for the left (and right) derivative, so that we obtain
	\begin{align}
		(\tran \circ f)^{\prime\ominus}  (t)
		&=
		\br{(\dtran \circ f) \cdot f\pr}\ld (t)
		\\&=
		(\dtran \circ f)\ld(t) f\pr(t) + \dtran(f(t)) f\prr(t)
		\eqfs
	\end{align}
	If $f$ is nondecreasing in $[t-\epsilon, t]$ for some $\epsilon>0$, we can apply the chain rule to obtain $(\dtran \circ f)\ld(t) = \ddltran(f(t)) f\pr(t)$. If $f$ is nonincreasing in $[t-\epsilon, t]$, the left derivative becomes a right derivative in the chain rule, i.e., we have $(\dtran \circ f)\ld(t) = \ddrtran(f(t)) f\pr(t)$. We can choose $\epsilon$ small enough so that one of the two cases is true. As we always have $\ddltran(f(t)) \geq \ddrtran(f(t))$, we obtain in both cases
	\begin{equation}
		(\tran \circ f)^{\prime\ominus}(t) \geq \ddrtran(f(t)) f\pr(t)^2 + \dtran(f(t)) f\prr(t)
		\eqfs
	\end{equation}
	Note that $\dtran(f(t)) > 0$ as $f(t) > 0$.
	We use \cref{lmm:gconv:second:deriv} to get
	\begin{equation}
		\dtran(f(t)) f\prr(t) \geq \dtran(f(t)) \frac{1 - f\pr(t)^2}{f(t)}
		\eqfs
	\end{equation}
	Thus,
	\begin{align}
		(\tran \circ f)^{\prime\ominus} (t)
		&\geq
		\br{\frac{\ddrtran(f(t))f(t)}{\dtran(f(t))} f\pr(t)^2 + \br{1 - f\pr(t)^2}}\frac{\dtran(f(t))}{f(t)}
		\\&=
		\br{1 - \br{1 - \frac{f(t)\ddrtran(f(t))}{\dtran(f(t))}} f\pr(t)^2}\frac{\dtran(f(t))}{f(t)}
		\eqfs
	\end{align}
	With $\dtran$ is nondecreasing, $\ddrtran$ nonincreasing ($\dtran$ is concave), and \cref{lmm:integration}, we obtain
	\begin{equation}
		\dtran(x) - \dtran(0) \geq \int_0^x \ddrtran(z) \dl z \geq x \ddrtran(x)
	\end{equation}
	for all $x\in\Rp$.
	Thus, with $\dtran(0) \geq 0$, we get
	\begin{equation}
		1 - \frac{x \ddrtran(x)}{\dtran(x)}  \geq 0
		\eqfs
	\end{equation}
	Furthermore, $f\pr(t)^2\leq 1$ as $f$ is $1$-Lipschitz. Thus,
	\begin{equation}
		1 - \br{1 - \frac{f(t)\ddrtran(f(t))}{\dtran(f(t))}} f\pr(t)^2 \geq \frac{f(t)\ddrtran(f(t))}{\dtran(f(t))}
		\eqfs
	\end{equation}
	Finally, we get
	\begin{equation}
		(\tran \circ f)\ld (t)
		\geq
		\br{1 - \br{1 - \frac{\ddrtran(f(t))f(t)}{\dtran(f(t))}} f\pr(t)^2}\frac{\dtran(f(t))}{f(t)}
		\geq
		\ddrtran(f(t)) \eqfs
	\end{equation}
	The calculations for $(\tran \circ f)\rd (t)$ are essentially the same.
\end{proof}
\begin{theorem}\label{lmm:semitaylor}
	Let $I\subset \R$ be convex with $0\in I$.
	Let $f\colon I\to\R$ be $\mc G$-convex.
	\begin{enumerate}[label=(\roman*)]
		\item\label{lmm:semitaylor:median} Then, for all $t\in I$, $t\neq0$,
		\begin{equation}\label{eq:semitaylor:median}
			f(t) - f(0)
			\geq
			tf\rd(0) + \frac12 t^2 \frac{1 - \max\brOf{f\rd(0)^2, f\ld(t)^2}}{\max(f(0), f(t))}
			\eqfs
		\end{equation}
		\item\label{lmm:semitaylor:tran} Let $\tran\in\setcc$.
		Then, for all $t\in I$, $t\neq0$,
		\begin{equation}\label{eq:semitaylor:tran}
			(\tran \circ f)(t) - (\tran \circ f)(0) \geq t (\tran \circ f)\rd(0) + \frac12 t^2 \ddrtran\brOf{\max\brOf{f(0), f(t)}}
			\eqfs
		\end{equation}
	\end{enumerate}
\end{theorem}
\begin{proof}
 	If $t<0$, we can replace $f$ by $t\mapsto f(-t)$. Thus, we can assume $t>0$ without loss of generality.
	\begin{enumerate}[label=(\roman*)]
		\item
		As $f$ is $\mc G$-convex, it is convex, see \cref{prp:gconv:fun} \ref{prp:gconv:fun:convex}.
		Thus, $f$ is twice differentiable almost everywhere (Alexandrov's theorem, cf.\ \cref{thm:alexandrov}).
		Let $B\subset I$ be the set of points where $f$ is not twice differentiable including potential boundary points of $I$.
		As $f$ is Lipschitz continuous, it is absolutely continuous. Thus, by \cref{lmm:integration},
		\begin{equation}
			f(t) - f(0) = \int_{[0,t]} f\rd(s) \dl s\eqfs
		\end{equation}
		As $f$ is convex, $f\rd$ is nondecreasing. Thus, by \cref{lmm:integration},
		\begin{equation}
			f\rd(s) - f\rd(0) \geq \int_{[0,s]\setminus B} f\prr(r) \dl r
			\eqfs
		\end{equation}
		By \cref{lmm:gconv:second:deriv}, for $r\in I\setminus B \subset \mathring I$,
		\begin{equation}
			f\prr(r) \geq \frac{1 - f\pr(r)^2}{f(r)}
			\eqfs
		\end{equation}
		Thus,
		\begin{align}
			f(t) - f(0)
			&\geq
			\int_{[0,t]} \br{ f\rd(0) + \int_{[0,s]\setminus B} \frac{1 - f\pr(r)^2}{f(r)}  \dl r } \dl s
			\\&\geq
			t f\rd(0) + \frac12 t^2 \inf_{s\in[0,t]\setminus B} \frac{1 - f\pr(s)^2}{f(s)}
			\eqfs
		\end{align}
		As $f\pr$ is nondecreasing, $([0,t]\setminus B) \to \R, s\mapsto  f\pr(s)^2$ is maximized at the extremes of the domain.
		As $f$ is convex, $[0,t]\to \R, s\mapsto f(s)$ is also maximized at the extremes. Thus,
		\begin{align}
			\inf_{s\in[0,t]\setminus B} \frac{1 - f\pr(s)^2}{f(s)}
			&\geq
			\frac{1 - \max\brOf{f\rd(0)^2, f\ld(t)^2}}{\max(f(0), f(t))}
			\eqfs
		\end{align}
		\item
		As $\tran$ and $f$ are convex and $\tran$ is nondecreasing, $\tran \circ f$ is convex. Thus, $\tran \circ f$ is twice continuously differentiable almost everywhere by Alexandrov's theorem, (cf.\ \cref{thm:alexandrov}). Let $B\subset I$ be the set of points where $\tran \circ f$ is not twice differentiable.
		As $\tran \circ f$ convex on $\R$, it is Lipschitz on $[0, t]$ and, hence, absolutely continuous on $[0, t]$. Thus, by \cref{lmm:integration},
		\begin{equation}
			(\tran \circ f)(t) - (\tran \circ f)(0) = \int_0^t (\tran \circ f)\rd(s) \dl s
			\eqfs
		\end{equation}
		Furthermore, as $(\tran \circ f)\rd$ is nondecreasing, by \cref{lmm:integration},
		\begin{equation}
			(\tran \circ f)\rd(s) - (\tran \circ f)\rd(0) \geq \int_{[0,s]\setminus B} (\tran \circ f)\prr(r) \dl r
			\eqfs
		\end{equation}
		For $r \in I\setminus B \subset \mathring I$, by \cref{lmm:secondDeriv},
		\begin{equation}
			(\tran \circ f)\prr(r) \geq \ddrtran(f(r))
			\eqfs
		\end{equation}
		As $\dtran$ is concave, $\ddrtran$ is nonincreasing. Thus, if $r\in[0, t]$, then
		\begin{equation}
			\ddrtran(f(r)) \geq \ddrtran\brOf{ \sup_{s\in [0, t]} f(s)} \geq \ddrtran\brOf{\max\brOf{f(0), f(t)}}
		\end{equation}
		as $f$ is convex.
		Hence,
		\begin{align}
			(\tran \circ f)(t) - (\tran \circ f)(0)
			&\geq
			\int_{[0,t]} \br{  (\tran \circ f)\rd(0) + \int_{[0,s]\setminus B} \ddrtran(f(r)) \dl r } \dl s
			\\&\geq
			t (\tran \circ f)\rd(0) + \frac12 t^2 \ddrtran\brOf{\max\brOf{f(0) , f(t)}}
			\eqfs
		\end{align}
	\end{enumerate}
\end{proof}
\begin{remark}[on \cref{lmm:semitaylor}]
	The inequalities seem rather sharp. For $f(t) = \abs{t}$, \eqref{eq:semitaylor:median} is an equality. Additionally, for $\tran = (x\mapsto x^\alpha)$, $\alpha\in[1,2]$, \eqref{eq:semitaylor:tran} is
	\begin{equation}
		t^\alpha \geq \frac{\alpha(\alpha-1)}{2} t^\alpha
		\eqcm
	\end{equation}
	which is an equality for $\alpha=2$ and trivial for $\alpha = 1$.
\end{remark}
\section{Transformed Fréchet Means in Hadamard Spaces}\label{sec:varineq:hadmard}
In this section, we restrict our discussion to Hadamard spaces. Basic properties of the $\tran$-Fréchet mean, $\tran\in\setcciz$, and its variance functional are discussed in section \ref{ssec:tran:basics} (\cref{lmm:closedBoundedConvex}). The variance inequality is shown in section \ref{ssec:tran:varineq} (\cref{thm:varinequ}). To complete the discussion, \cref{thm:point:convex} in section \ref{ssec:tran:pointmass} shows a result for $\Prof{Y = m} > 0$.
\subsection{Basics}\label{ssec:tran:basics}
The notion of convexity can be transferred to Hadamard spaces, see, e.g., \cite[chapter 2]{bacak14convex}. We use the term \emph{convex} here, but some authors prefer \emph{geodesically convex} in this context.
\begin{definition}
	Assume $(\mc Q, d)$ is Hadamard.
	\begin{enumerate}[label=(\roman*)]
		\item A set $A\subset \mc Q$ is called \emph{convex} if and only if for any $q,p\in A, q\neq p$, have $\geodft qp\subset A$.
		\item A function $f\colon\mc Q\to\R$ is called \emph{convex} if and only if for any $q,p\in\mc Q, q\neq p$, we have $f \circ \geodft qp$ is convex.
	\end{enumerate}
\end{definition}
\begin{proposition}\label{lmm:closedBoundedConvex}
	Let $\tran\in\setcciz$.
	Assume $(\mc Q, d)$ is Hadamard.
	Assume $\Ex{\dtran(\ol Yo)} < \infty$.
	\begin{enumerate}[label=(\roman*)]
		\item The variance 	functional $v \colon \mc Q \to \R$, $v(q) := \Ex{\tran(\ol Yq) - \tran(\ol Yo)}$ is convex.
		\item Let $M := \argmin_{q\in\mc Q} v(q)$ be the set of $\tran$-Fréchet means.
		Then $M$ is nonempty, closed, bounded, and convex.
        Furthermore, $M$ does not depend on the choice of the reference point $o$.
        \item\label{lmm:closedBoundedConvex:convexsup} If $Y$ is concentrated on a closed convex set $\mc Y\subset \mc Q$, i.e., $\Prof{Y\in\mc Y}=1$, then $M \subset \mc Y$.
	\end{enumerate}
\end{proposition}
\begin{proof}
	For a geodesic $\gamma\in\GeodUS$, the distance function $\ol y\gamma$ is $\mc G$-convex, see \cref{lem:distFun}. Thus, $t\mapsto\tran(\ol Y\gamma(t)) - \tran(\ol Yp)$ is convex. Hence, $\Ex{\dtran(\ol Yo)} < \infty$ implies that the variance functional $v$ is convex.

	By \cref{prp:moment}, the bound $\Ex{\dtran(\ol Yo)} < \infty$ yields continuity of $v$ via $\abs{v(q) - v(p)}\leq D(q,p)$ with $D$ from \eqref{eq:altmetricdef}.
	Using \cite[Example 2.1.3]{bacak14convex}, the sublevel sets of $v$, $\setByEleInText{q\in\mc Q}{v(q) \leq a}$ for $a\in\R$, are closed and convex.
	Furthermore, by \cref{thm:varineq:general} \ref{thm:varineq:general:infty}, $v(q) \to \infty$ whenever $\ol qo\to\infty$.
	Following the proof of \cite[Lemma 2.2.19]{bacak14convex} and using \cite[Proposition 2.1.16]{bacak14convex}, we obtain that $M$ is nonempty, closed, bounded, and convex. It does not depend on $o$, as all expectations are finite independent of the reference point (\cref{prp:moment}).

	Now consider a closed and convex set $\mc Y\subset \mc Q$ such that $\Prof{Y \in \mc Y} = 1$.
	By \cite[Theorem 2.1.12]{bacak14convex}, for every $q\in \mc Q \setminus\mc Y$, there is $p\in\mc Y$ such that
	\begin{equation}
		\ol yq^2 \geq \ol yp^2 + \ol qp^2
	\end{equation}
	for all $y\in\mc Y$. Thus, $\ol yq > \ol yp$. As we exclude constant transformations, $\dtran$ is concave, and $\dtran(x) \geq 0$ for all $x\in\Rp$, we obtain that $\tran$ is strictly increasing and, therefore, $\tran(\ol yq) > \tran(\ol yp)$. Integration over $Y \in \mc Y$ yields $\Ex{\tran(\ol Yq) - \tran(\ol Yp)} > 0$. Hence, $q\not\in M$.
\end{proof}
\begin{remark}\mbox{ }
    \begin{enumerate}[label=(\roman*)]
        \item
        By the Hopf--Rinow theorem \cite[Theorem 2.5.28]{burago01}, if $\mc Q$ is a locally compact Hadamard space, every closed and bounded subset of $\mc Q$ is compact (Heine--Borel property). Hence, \cref{lmm:closedBoundedConvex} implies compactness of the $\tran$-Fréchet mean set in locally compact Hadamard spaces.
        \item By \cref{lmm:closedBoundedConvex} \ref{lmm:closedBoundedConvex:convexsup}, if $\mc Q$ is separable, then $M$ is contained in the closed convex hull of the support of $Y$, i.e., in the intersection of all closed convex supersets of the support of $Y$.
        Note that, if $\mc Q$ is not separable, $Y$ can have an empty support.
    \end{enumerate}
\end{remark}
\subsection{Variance Inequality}\label{ssec:tran:varineq}
\begin{theorem}\label{thm:varinequ}
	Let $\tran\in\setcciz$.
	Assume $(\mc Q, d)$ is Hadamard.
	Assume $\Ex{\dtran(\ol Yo)} < \infty$.
	Let $m \in\argmin_{q\in\mc Q}\Ex{\tran(\ol Yq) - \tran(\ol Yo)}$. Let $q\in\mc Q\setminus\cb{m}$.
	Then
	\begin{equation}
		\Ex*{\tran(\ol Yq) - \tran(\ol Ym)} \geq \frac12 \ol qm^2 \Ex*{\ddrtran(\max(\ol Ym, \ol Yq))}\eqfs
	\end{equation}
\end{theorem}
\begin{proof}
	Let $\gamma := \geodft mq$. For $t\in[0, \ol qm]$ and $y\in \mc Q$, define $V_y(t) := \tran(\ol y\gamma(t)) - \tran(\ol y\gamma(0))$. Thus, $V_y(0) = 0$.
	By \cref{lmm:semitaylor} \ref{lmm:semitaylor:tran},
	\begin{equation}\label{eq:taylorraw}
		V_y(t)
		\geq
		t  V_y\rd(0) + \frac12 t^2 \ddrtran\brOf{\max(\ol y\gamma(0), \ol y\gamma(t))}
		\eqfs
	\end{equation}
	By \cref{prp:moment}, $\Ex{\dtran(\ol Yo)} < \infty$ yields $\Ex*{\abs{V_Y(t)}} < \infty$.
	Hence, $G(t) := \Ex{V_Y(t)}$ is well-defined. Furthermore, as $\ol y\gamma$ is $1$-Lipschitz, we have, for $s\in[0,\ol qm)$,
	\begin{align}
		\abs{\partial_+ \tran(\ol y\gamma(s))}
		&\leq
		\max\brOf{\abs{\dtran(\ol y\gamma(s))\ol y\gamma\rd(s)}, \abs{\dtran(\ol y\gamma(s))\ol y\gamma\ld(s)}}
		\\&\leq
		\dtran(\ol y\gamma(s))
		\\&\leq
		\dtran(\ol ym + s)
		\\&\leq
		\dtran(\ol ym) + \dtran(s)
		\eqfs
	\end{align}
	Thus,
	\begin{equation}
		\Ex*{\sup_{s\in[0,\ol qm)} \abs{\partial_{+} \tran(\ol Y\gamma(s))}} \leq \Ex*{\dtran(\ol Ym)} + \dtran(\ol qm) < \infty
		\eqfs
	\end{equation}
	Hence, we can swap integral and derivative to get
	\begin{equation}
		G\rd(t) = \Ex*{\partial_+ \tran(\ol Y\gamma(t))}
	\end{equation}
	for $t\in[0,\ol qm)$.
	As $G$ is convex and  minimized at $0$,
	\begin{equation}
		0 \leq G\rd(0) = \Ex*{\partial_+ \tran(\ol Y\gamma(0))}
		\eqfs
	\end{equation}
	Thus, integrating \eqref{eq:taylorraw} yields,
	\begin{equation}
		G(t) \geq \frac12 t^2 \Ex*{\ddrtran\brOf{\max(\ol Y\gamma(0), \ol Y\gamma(t))}}
		\eqfs
	\end{equation}
	Plugging in $t = \ol qm$ yields the desired result.
\end{proof}
\begin{remark}[on \cref{thm:varinequ}]\mbox{ }
	\begin{enumerate}[label=(\roman*)]
		\item \cref{thm:varinequ} applied to $\tran = (x\mapsto x^2)$, yields exactly the well-known variance inequality \eqref{eq:varinequhadamard}, which is an equality in Euclidean spaces, see \eqref{eq:vareq}. See also \cref{exa:pm1} and \cref{exa:huber} for a further evaluation of the sharpness of the result.
		\item If $\Ex*{\ddrtran(\max(\ol Ym,\ol Yq))} \xrightarrow{q\to m} \infty$, the lower bound in \cref{thm:varinequ} can be steeper than quadratic for $q$ close to $m$. A sufficient condition for such a growth behavior is given in \cref{cor:varinequ} below.
		\item
		On one hand, if $\ddrtran(x) \xrightarrow{x\to\infty} 0$, the lower bound can be less steep than linear for $q$ far away from $m$. It can even vanish for finite distance, e.g., for the Huber loss \eqref{eq:huber}.
		On the other hand, the variance functional is convex, see \cref{lmm:closedBoundedConvex}. Thus, the slope of the variance functional cannot decrease. In other words, if the lower bound in \cref{thm:varinequ} is nontrivial for some $q$ close to $m$, it can be extended by a suitable affine function to all $q$ further away from $m$.
        Furthermore, for $q$ far away from $m$, the variance functional looks like $\tran(\ol qm)$ as stated in \cref{thm:varineq:general}.
		\item By \cite[Lemma C.3]{schoetz22}, for every random variable $Y$ with values in $\mc Q$, there is $\tran\in\setcciz$ with $\ddrtran(x) > 0$ for all $x\in\Rpp$ and $\Ex{\dtran(\ol Yo)} < \infty$. Thus, there is always a $\tran\in\setcciz$ such that \cref{thm:varinequ} does not yield merely a trivial result.
	\end{enumerate}
\end{remark}
\begin{example}\label{exa:pm1}\mbox{ }
	\begin{enumerate}[label=(\roman*)]
		\item\label{exa:pm1:tran} Let $(\mc Q, d)$ be the real line with Euclidean metric $(\R, \abs{\cdot - \cdot})$. Let $\Prof{Y = 1}=\Prof{Y = -1}=\frac12$. Assume $\tran\in\setcc$ is three times continuously differentiable on $\Rpp$. Because of the symmetry of the distribution of $Y$, $m = 0$ is a $\tran$-Fréchet mean of $Y$.
		Then, for $q\in(-1,1)$,
		\begin{equation}
			v(q) := \Ex{\tran(\abs{Y-q}) -\tran(\abs{Y})} = \frac{\tran(1+q)+\tran(1-q)}2 - \tran(1)\eqfs
		\end{equation}
		Using a Taylor approximation, we have
		\begin{equation}
			v(q) = \frac12 q^2 \ddtran(1) + \frac1{12} q^3 \br{\dddtran(1+\xi) -\dddtran(1-\xi)}
		\end{equation}
		for some $\xi\in[0,q]$.
		Evaluating the lower bound in \cref{thm:varinequ}, yields
		\begin{equation}
			\Ex{\tran(\abs{Y-q}) -\tran(\abs{Y})} \geq \frac14 q^2 \br{\ddtran(1)+\ddtran(1+\abs{q})}
			\eqfs
		\end{equation}
		The result obtained from \cref{thm:varinequ} is close to the correct second order term for $q$ close to $0$.
		\item
		Let us consider $\tran = (x \mapsto x^\alpha)$ for $\alpha \in(1,2]$. Then, in the same setup as in \ref{exa:pm1:tran}, we obtain
		\begin{equation}\label{eq:pm1:alpha}
			\frac12 \alpha(\alpha-1) q^2 \leq \Ex{\abs{Y-q}^\alpha -\abs{Y}^\alpha} \leq \frac12 \alpha (\alpha-1) q^2 + \frac43 \abs{q}^3
		\end{equation}
		for $q\in[-\frac12, \frac12]$ via a Taylor approximation and explicit bounds on the third order remainder term. Using \cref{thm:varinequ} as in \ref{exa:pm1:tran} and the bounds on $q$ and $\alpha$, we obtain the lower bound $\frac5{12} \alpha(\alpha-1) q^2$. Let us compare this result with the power transform variance inequality in \cite[Proposition 2]{yun2023exponential}. Applied to our setting here, it establishes an inequality of the form
		\begin{equation}\label{eq:pm1:yun}
			B_\alpha  \abs{q}^\alpha \leq \Ex{\abs{Y-q}^\alpha -\abs{Y}^\alpha}
		\end{equation}
		for some constant $B_\alpha \in\Rp$, which depends on $\alpha$ and the distribution of $Y$. Because of \eqref{eq:pm1:alpha}, we must have $B_\alpha = 0$ for $\alpha\in(1,2)$, i.e., \eqref{eq:pm1:yun} only yields a trivial result here, in contrast to \cref{thm:varinequ}.
 	\end{enumerate}
\end{example}
In the following, the infimum of an empty set is $\infty$.
\begin{corollary}\label{cor:nonuniquetran}
	Let $\tran\in\setcciz$.
	Assume $(\mc Q, d)$ is Hadamard.
	Assume $\Ex{\dtran(\ol Yo)} < \infty$.
	Let $m \in\argmin_{q\in\mc Q}\Ex{\tran(\ol Yq) - \tran(\ol Yo)}$.
	Let $x_0 = \inf\setByEleInText{x\in\Rpp}{\ddrtran(x) = 0}$.
	Assume $\PrOf{\ol Ym < x_0} > 0$.
	Then $m$ is the only $\tran$-Fréchet mean of $Y$.
\end{corollary}
\begin{proof}
	Let $p,m\in M := \argmin_{q\in\mc Q}\Ex{\tran(\ol Yq) - \tran(\ol Yo)}$, $p\neq m$. Then $\gamma := \geodft mp\subset M$ as $M$ is convex by \cref{lmm:closedBoundedConvex}.
	\cref{thm:varinequ} yields
	\begin{equation}
		0 = \Ex*{\tran(\ol Y\gamma(t)) - \tran(\ol Ym)} \geq \frac12 \ol m\gamma(t)^2 \Ex*{\ddrtran\brOf{\max(\ol Ym, \ol Y\gamma(t))}}\eqfs
	\end{equation}
	for $t\in (0, \ol pm]$.
	As $\ddrtran$ is nonnegative and nonincreasing, we obtain, using the triangle inequality,
	\begin{equation}\label{eq:zeroddtran}
		0 = \ddrtran\brOf{\max(\ol Ym, \ol Y\gamma(t))} \geq  \ddrtran\brOf{\ol Ym + t} \geq 0
	\end{equation}
	almost surely.
	If $\Prof{\ol Ym < x_0} > 0$, then we can make $t>0$ small enough so that $\Prof{\ol Ym + t < x_0} > 0$.
	But then $\Prof{\ddrtran\brOf{\ol Ym + t} > 0} > 0$ by the definition of $x_0$ in contradiction to \eqref{eq:zeroddtran}.
	Hence, $M$ cannot contain two different elements.
\end{proof}
\begin{corollary}[Unique $\tran$-Fréchet mean]\label{cor:uniquetran}
	Let $\tran\in\setcciz$.
	Assume $(\mc Q, d)$ is Hadamard.
	Assume $\Ex{\dtran(\ol Yo)} < \infty$.
	Assume $\ddrtran(x) > 0$ for all $x\in\Rpp$.
	Then the $\tran$-Fréchet mean is unique.
\end{corollary}
\begin{proof}
	This directly follows from \cref{cor:nonuniquetran}.
\end{proof}
The following corollary draws a rather complete picture of the behavior of the variance functional in form of upper and lower bounds for most settings where $\tran$ is not linear. It distinguishes the growth behavior close to $m$ from the one far away.
\begin{corollary}\label{cor:varinequ}
	Let $\tran\in\setcciz$.
	Assume $(\mc Q, d)$ is Hadamard.
	Assume $\Ex{\dtran(\ol Yo)} < \infty$.
	Let $m \in\argmin_{q\in\mc Q}\Ex{\tran(\ol Yq) - \tran(\ol Yo)}$.
	Let $x_0 = \inf\setByEleInText{x\in\Rpp}{\ddrtran(x) = 0}$. Assume $\PrOf{\ol Ym < x_0} > 0$.
	Set $\beta := 0 $ or assume $\beta\in[0, 1)$ such that
	\begin{equation}\label{eq:fastratecondition}
		\liminf_{x\searrow 0} \frac{\ddrtran(2x) \PrOf{\ol Ym \leq x}}{x^{-\beta}} > 0
		\eqfs
	\end{equation}
	Then, there are $c_1,\dots,c_4,\delta\in\Rpp$ with the following property:
	\begin{enumerate}[label=(\roman*)]
		\item For all $q\in\ball{m}{\delta}$
		\begin{equation}
			c_1\ol qm
			\geq
			\EOf{\tran(\ol Y{q}) - \tran(\ol Ym)}
			\geq
			c_2 \ol qm^{2-\beta}
			\eqfs
		\end{equation}
		\item For all $q\in\mc Q\setminus \ball{m}{\delta}$
		\begin{equation}
			c_3 \tran(\ol qm)
			\geq
			\EOf{\tran(\ol Y{q}) - \tran(\ol Ym)}
			\geq
			c_4 \tran(\ol qm)
			\eqfs
		\end{equation}
	\end{enumerate}
\end{corollary}
\begin{proof}
	Using \cref{thm:varinequ} and the triangle inequality, we obtain
	\begin{equation}
		\Ex*{\tran(\ol Yq) - \tran(\ol Ym)}
		\geq
		\frac12 \ol qm^2 \Ex*{\ddrtran(\ol Ym + \ol qm)}
		\geq
		\frac12 \ol qm^2 \ddrtran(2 \ol qm) \PrOf{\ol Ym \leq \ol qm}
		\eqfs
	\end{equation}
	Thus, by \eqref{eq:fastratecondition}, there are $\delta_0, c_0 \in\Rpp$ such that
	\begin{equation}\label{eq:trancorproof:inner:lower}
		\Ex*{\tran(\ol Yq) - \tran(\ol Ym)} \geq c_0 \ol qm^{2-\beta}
	\end{equation}
	for all $q\in\ball{m}{\delta_0}$. Furthermore, by \cref{thm:varineq:general} \ref{thm:varineq:general:zero}, there is $\delta_1 \in\Rpp$ such that
	\begin{equation}\label{eq:trancorproof:inner:upper}
		\Ex*{\tran(\ol Yq) - \tran(\ol Ym)} \leq 2\, \ol qm \EOf{\dtran\brOf{\ol Ym}}
	\end{equation}
	for all $q\in\ball{m}{\delta_1}$.
	Moreover, by \cref{thm:varineq:general} \ref{thm:varineq:general:infty}, there is $\delta_2 \in\Rpp$ such that
	\begin{equation}\label{eq:trancorproof:outer}
		2 \tran(\ol qm) \geq \Ex*{\tran(\ol Yq) - \tran(\ol Ym)} \geq \frac12 \tran(\ol qm)
	\end{equation}
	for all $q\in\mc Q\setminus \ball{m}{\delta_2}$.
	Finally, as the variance functional is convex (\cref{lmm:closedBoundedConvex}), we can close the gap between $\ol qm < \min(\delta_0, \delta_1)$, where the bounds \eqref{eq:trancorproof:inner:lower} and \eqref{eq:trancorproof:inner:upper} hold, and $\ol qm \geq \delta_2$, where \eqref{eq:trancorproof:outer} is true, by setting $\delta$ and the constant factors $c_1, \dots, c_4$ in the statement of the corollary appropriately.
\end{proof}
\subsection{Point Mass}\label{ssec:tran:pointmass}
When $\Prof{Y=m}>0$, \cref{thm:varinequ} yields a variance inequality of the form
\begin{equation}
	\Ex*{\tran(\ol Yq) - \tran(\ol Ym)} \geq \frac12 \Prof{Y=m} \ol qm^2 \ddrtran(\ol mq)\eqfs
\end{equation}
As $\dtran(x) \geq 0$ and $\ddrtran$ is nonincreasing, we have $\frac12 x^2 \ddrtran(x) \leq \tran(x)$ for all $x\in\Rpp$ (and this inequality can be strict), so that the following theorem is an improvement.
\begin{theorem}\label{thm:point:convex}
	Let $\tran\in\setcciz$ with $\dtran(0) = 0$.
	Assume $(\mc Q, d)$ is Hadamard.
	Assume $\Ex{\dtran(\ol Yo)} < \infty$.
	Then
	\begin{equation}
		\Ex*{\tran(\ol Yq) - \tran(\ol Ym)} \geq \tran(\ol qm) \Prof{Y = m}
	\end{equation}
	for all $q\in\mc Q$.
\end{theorem}
\begin{proof}
	Let $\gamma\in\GeodUS$ with $\gamma(0) = m$. Define
	\begin{align}
		f_1(t) &:= \Ex*{\br{\tran(\ol Y\gamma(t)) - \tran(\ol Ym)}\ind_{\cb{Y=m}}} = \tran(t)\PrOf{Y=m} \eqcm\\
		f_2(t) &:= \Ex*{\br{\tran(\ol Y\gamma(t)) - \tran(\ol Ym)}\ind_{\cb{Y\neq m}}}\eqfs
	\end{align}
	Then
	\begin{equation}
		\Ex*{\tran(\ol Y\gamma(t)) - \tran(\ol Ym)} = f_1(t) + f_2(t)
		\eqfs
	\end{equation}
	We will show $f_2(t) \geq 0$ for all $t\in\Rp$ to prove the theorem.
	The functions
	$f_1, f_2, f_1 + f_2$ are all convex as $\tran$ is convex and $t\mapsto \ol y\gamma(t)$ is convex for all $y\in\mc Q$ (\cref{lem:distFun} and \cref{prp:gconv:fun} \ref{prp:gconv:fun:convex}).
	As $\lim_{t\to 0} f_1(t)/t = 0$ by  $\dtran(0) = 0$ and $f_1 + f_2 \geq 0$,
	\begin{align}\label{eq:pointmass:liminf}
		\liminf_{t\searrow0} \frac{f_2(t)}{t}
		&\geq
		\liminf_{t\searrow0} \frac{f_1(t)+f_2(t)}{t} -
		\limsup_{t\searrow0} \frac{f_1(t)}{t}
		\geq
		0\eqfs
	\end{align}
	Assume there were $t_0\in\Rpp$ such that $f_2(t_0) < 0$.
	As $f_2(0) = 0$ and $f_2$ is convex, we obtain, for all $s\in(0,1]$, $f_2(s t_0) \leq s f_2(t_0)$. Thus, we would have
	\begin{equation}
		\frac{f_2(s t_0)}{s t_0} \leq  \frac{f_2(t_0)}{t_0} < 0
		\eqfs
	\end{equation}
	In other words, $\frac{f_2(t)}{t}$ would be smaller than a negative constant for all $t\in(0,t_0]$.
	But this contradicts \eqref{eq:pointmass:liminf}.
	Hence, $f_2(t) \geq 0$ for all $t\geq 0$.
\end{proof}
\section{The Fréchet Median in Hadamard Spaces}\label{sec:varineq:hadmard:median}
First, we show that focusing only on the Fréchet median is enough to essentially complete the discussion of variance inequalities for $\tran\in\setcciz$ (section \ref{ssec:median:reduction}, \cref{thm:affinereduction}). In section \ref{ssec:median:unique}, we discuss criteria for uniqueness of the Fréchet median (\cref{thm:medianConcentration}). A variance inequality for the Fréchet median is derived in section \ref{ssec:median:varineq} (\cref{thm:near:median}). It requires the random object $Y$ to not be concentrated on a geodesic. The case $\PrOf{Y\in\gamma} = 1$ for a geodesic $\gamma$ is discussed in section \ref{ssec:median:geodesic} (\cref{thm:median:geodesic}).
\subsection{Reduction to the Median}\label{ssec:median:reduction}
The next theorem shows that the cases not covered in the section \ref{sec:varineq:hadmard}, i.e., when $\PrOf{\ol Ym < x_0} = 0$, $x_0 = \inf\setByEleInText{x\in\Rpp}{\ddrtran(x) = 0}$, essentially reduce to the case of the Fréchet median.
\begin{theorem}\label{thm:affinereduction}
    Let $\tran \in \setcciz$.
    Let $x_0 = \inf\setByEleInText{x\in\Rpp}{\ddrtran(x) = 0}$.
    Assume $x_0 < \infty$.
    Let $m \in\argmin_{q\in\mc Q}\Eof{\tran(\ol Yq) - \tran(\ol Yo)}$.
    Define $B_0 := \setByEleInText{q\in\mc Q}{\Prof{\ol Y{q} < x_0} = 0}$.
    Assume $m \in B_0$.
    Then
    \begin{equation}\label{eq:affinemedian:charac}
        \argmin_{q\in\mc Q}\Eof{\tran(\ol Yq) - \tran(\ol Yo)}
        =
        B_0 \cap \argmin_{q\in\mc Q}\Eof{\ol Yq - \ol Yo}
    \end{equation}
    and, for all $q\in\mc Q$,
    \begin{equation}\label{eq:affinemedian:varinequ}
        \EOf{\tran(\ol Yq) - \tran(\ol Ym)}
        \geq
        \dtran(x_0) \EOf{\ol Yq - \ol Ym}
        \eqfs
    \end{equation}
\end{theorem}
\begin{remark}\mbox{ }
    \begin{enumerate}[label=(\roman*)]
        \item \cref{thm:affinereduction} applies not only in Hadamard metric spaces, but in all metric spaces.
        \item If the Fréchet median is unique and we can establish a variance inequality for it, \cref{thm:affinereduction} implies that the $\tran$-Fréchet mean (with the conditions on $\tran$ and $Y$ stated in the theorem) is unique and exhibits the same variance inequality up to a positive constant factor.
    \end{enumerate}
\end{remark}
\begin{proof}[Proof of \cref{thm:affinereduction}]
    As $x_0 < \infty$, we have $b := \sup_{x\in\Rpp}\dtran(x) = \dtran(x_0)$ with $b > 0$ as $\tran$ is nondecreasing and not constant. As $\tran$ is convex with concave derivative, we further have a constant $a\in\R$ such that
    \begin{equation}\label{eq:affinemedian:lower}
        \begin{array}{rl}
            \tran(x) > a + bx &\text{if } x \in [0, x_0)\eqcm\\
            \tran(x) = a + bx &\text{if } x \in [x_0, \infty)\eqfs
        \end{array}
    \end{equation}
    Let $m_0 \in \argmin_{q\in\mc Q}\Eof{\ol Yq - \ol Yo} \cap B_0$.
    As $b>0$ and $m_0$ minimizes $\Eof{\ol Yq - \ol Yo}$ with respect to $q\in\mc Q$, it also minimizes $a + b\Eof{\ol Yq - \ol Yo}$. As we assume $\PrOf{\ol Y{m_0} < x_0} = 0$, we obtain
    \begin{equation}
        \Eof{\tran(\ol Y{m_0}) - b\ol Yo} = a + b\Eof{\ol Y{m_0} - \ol Yo}
        \eqfs
    \end{equation}
    Hence, $m_0$ minimizes $\Eof{\tran(\ol Yq) - b\ol Yo}$ and therefore also $\Eof{\tran(\ol Yq) - \tran(\ol Yo)}$.
    Thus, we have shown
    \begin{equation}\label{eq:affinemedian:sup}
        \argmin_{q\in\mc Q}\Eof{\tran(\ol Yq) - \tran(\ol Yo)}
        \supseteq
        B_0 \cap \argmin_{q\in\mc Q}\Eof{\ol Yq - \ol Yo}
        \eqfs
    \end{equation}
    Because of \eqref{eq:affinemedian:lower}, we further obtain
    \begin{equation}\label{eq:costquality}
        \EOf{\tran(\ol Yq) - \tran(\ol Y{m_0})}
        \geq
        b \EOf{\ol Yq - \ol Y{m_0}}
    \end{equation}
    with equality if and only if $q\in B_0$.
    For $m \in \argmin_{q\in\mc Q}\Eof{\tran(\ol Yq) - \tran(\ol Yo)}$, we have $0 = \EOf{\tran(\ol Y{m}) - \tran(\ol Y{m_0})}$ because of \eqref{eq:affinemedian:sup}.
    Hence, \eqref{eq:costquality} and its condition for equality imply $m \in \argmin_{q\in\mc Q}\Eof{\ol Yq - \ol Yo}$ and $m\in B_0$.
    Thus, we have
    \begin{equation}\label{eq:affinemedian:sub}
        \argmin_{q\in\mc Q}\Eof{\tran(\ol Yq) - \tran(\ol Yo)}
        \subseteq
        B_0 \cap \argmin_{q\in\mc Q}\Eof{\ol Yq - \ol Yo}
        \eqfs
    \end{equation}
    Equations \eqref{eq:affinemedian:sup} and \eqref{eq:affinemedian:sub} imply \eqref{eq:affinemedian:charac}, and \eqref{eq:costquality} then shows \eqref{eq:affinemedian:varinequ}.
\end{proof}
In the next example, we consider the transformed Fréchet mean induced by the Huber loss (the \textit{Huber-Fréchet mean}). We apply either \cref{thm:varinequ} or \cref{thm:affinereduction} depending on the threshold parameter of the Huber loss and compare the resulting lower bounds with the exact variance functionals.
\begin{example}[Huber loss]\label{exa:huber}
    Let us use the real line as our metric space $(\mc Q, d) = (\R, \abs{\cdot - \cdot})$. Let $\delta\in\Rpp$ and set $\tran := \tran_{\ms h,\delta}$, see \eqref{eq:huber}.
    Let $z\in\Rpp$ and let $Y$ be an $\R$-valued random variable with $\Prof{Y = -z} = \Prof{Y = z} = \frac12$. Then $m=0$ is a $\tran$-Fréchet mean and a Fréchet median, and we have the following variance equations for $\tran$:
    For $z\leq \delta$,
    \begin{equation}\label{exa:huber:vareq:unique}
    	\EOf{\tran(\ol Yq) - \tran(\ol Y0)} =
    	\begin{cases}
    		\frac12 q^2 & \text{ for } \abs{q} \in [0, \delta-z]\eqcm\\
    		\frac14 q^2 + \frac12(\delta-z) \abs{q} - \frac14 (\delta - z)^2 &\text{ for } \abs{q} \in [\delta-z, \delta+z]\eqcm\\
    		\delta \abs{q} -\frac12 (\delta^2+z^2) &\text{ for } \abs{q} \in [\delta+z, \infty) \eqcm
    	\end{cases}
    \end{equation}
    for $z\geq \delta$,
    \begin{equation}\label{exa:huber:vareq:median}
        \EOf{\tran(\ol Yq) - \tran(\ol Y0)} =
        \begin{cases}
            0 & \text{ for } \abs{q} \in [0, z-\delta]\eqcm\\
            \frac14 q^2 + \frac12 (\delta-z)\abs{q} + \frac14 (\delta-z)^2  &\text{ for } \abs{q} \in [z-\delta, z+\delta]\eqcm\\
            \delta (\abs{q} -z) &\text{ for } \abs{q} \in [\delta+z, \infty) \eqfs
        \end{cases}
    \end{equation}
    Moreover, for the median variance functional, we have
    \begin{equation}
        \EOf{\ol Yq - \ol Y0} =
        \begin{cases}
            0 & \text{ for } \abs{q} \in [0, z]\eqcm\\
            \abs{q} - z &\text{ for } \abs{q} \in [z, \infty) \eqfs
        \end{cases}
    \end{equation}
    Furthermore, in terms of \cref{thm:affinereduction}, we have $x_0 = \delta$ and
    \begin{equation}
        B_0 =
        \begin{cases}
            (-\infty, -(\delta + z)] \cup [\delta + z, \infty) & \text{ for } z < \delta\eqcm\\
            (-\infty, -(\delta + z)] \cup [\delta-z, z-\delta] \cup [\delta + z, \infty) & \text{ for } z \geq \delta\eqfs
        \end{cases}
    \end{equation}
    The Huber-Fréchet mean (set) is
    \begin{equation}
        \argmin_{q\in\mc Q} \EOf{\tran(\ol Yq) - \tran(\ol Yo)} =
        \begin{cases}
            0 & \text{ for } z \leq \delta\eqcm\\
            [\delta-z, z-\delta] &\text{ for } z \geq \delta \eqfs
        \end{cases}
    \end{equation}
    and the Fréchet median set is
    \begin{equation}
        \argmin_{q\in\mc Q} \EOf{\ol Yq - \ol Yo} = [-z, z]
        \eqfs
    \end{equation}
    In the case $z < \delta$, \cref{thm:varinequ} yields
    \begin{equation}\label{exa:huber:varineq}
        \EOf{\tran(\ol Yq) - \tran(\ol Y0)} \geq
        \begin{cases}
            \frac12 q^2 & \text{ for } \abs{q} \in [0, \delta-z)\eqcm\\
            \frac14 q^2  &\text{ for } \abs{q} \in [\delta-z, \delta+z)\eqcm\\
            0 &\text{ for } \abs{q} \in [\delta+z, \infty) \eqfs
        \end{cases}
    \end{equation}
    In the case $z \geq \delta$, \cref{thm:affinereduction} yields
    \begin{equation}\label{exa:huber:reduc}
        \EOf{\tran(\ol Yq) - \tran(\ol Y0)} \geq \dtran(x_0) \EOf{\ol Yq - \ol Ym} = \begin{cases}
            0 & \text{ for } \abs{q} \in [0, z]\eqcm\\
            \delta(\abs{q} - z) &\text{ for } \abs{q} \in [z, \infty) \eqfs
        \end{cases}
    \end{equation}
\end{example}
\begin{remark}\mbox{ }
    \begin{enumerate}[label=(\roman*)]
        \item \label{rem:exa:huber:convex}One can extend the bound in \ref{exa:huber:varineq} using the convexity of the variance functional (\cref{lmm:closedBoundedConvex}): For $\abs{q} \in [\delta+z, \infty)$, we obtain the lower bound $\frac12(\delta+z)\abs{q} - \frac14(\delta+z)^2$ on $\EOf{\tran(\ol Yq) - \tran(\ol Y0)}$.
        \item Using \ref{rem:exa:huber:convex}, the lower bounds derived from \cref{thm:varinequ} in the case $z < \delta$ show the correct growth behavior of the variance functional up to a constant.
        \item To apply \cref{thm:affinereduction} we require $B_0$ to contain a transformed Fréchet mean. This is only true in the case $z\geq \delta$ but not in the case $z < \delta$.
        \item In the case $z\geq \delta$, \cref{thm:affinereduction} yields a non-trivial lower bound. But, in the specific example here, it does not capture the growth behavior perfectly as the lower bound is constant for a larger interval than in the case of the true variance functional.
    \end{enumerate}
\end{remark}
\subsection{Uniqueness of the Fréchet Median}\label{ssec:median:unique}
To state the theorem on uniqueness of the Fréchet median, we introduce some terminology.
Recall $I\ld = I\setminus\cb{\inf I}$, $I\rd = I\setminus\cb{\sup I}$, and $\mathring I =I\ld \cap I\rd$ for an interval $I\subset \R$ (\cref{not:leftright}).
\begin{notation}\mbox{ }
	\begin{enumerate}[label=(\roman*)]
		\item
			For a set $A$, let $\# A$ be the \emph{cardinality} of $A$.
		\item
			A \emph{geodesic segment} is a closed and bounded set that is an image of a geodesic.
		\item
            For a geodesic $\gamma\colon I_{\gamma} \to \mc Q$, let $\mathring \gamma$ denote the \emph{interior geodesic image}, $\mathring \gamma := \gamma(\mathring I_{\gamma})$.
		\item
			Let $\gamma\in\GeodUS$. Define the \emph{left and right points of $\gamma$} as
			\begin{align}
				L(\gamma) &:= \setByEle{y\in\mc Q}{\forall t \in I\rd_\gamma\colon \ol y\gamma\rd(t) = 1}\eqcm\\
				R(\gamma) &:= \setByEle{y\in\mc Q}{\forall t \in I\ld_\gamma\colon \ol y\gamma\ld(t) = -1}\eqfs
			\end{align}
	\end{enumerate}
\end{notation}
\begin{theorem}\label{thm:medianConcentration}
	Assume $(\mc Q, d)$ is Hadamard.
	Set $M := \argmin_{q\in\mc Q}\Eof{\ol Yq - \ol Yo}$.
	\begin{enumerate}[label=(\roman*)]
		\item Then $M$ is a geodesic segment.
		\item Assume $\#M > 1$. Let $\gamma_M\in\GeodUS$ such that $\gamma_M(I_\gamma) = M$. Then
		\begin{equation}\label{eq:medianLeftRightHalfProb}
			\PrOf{Y\in L(\gamma_M)} = \PrOf{Y\in R(\gamma_M)} = \frac12 \qquad\text{and}\qquad \PrOf{Y \in \mathring \gamma_M} = 0
			\eqfs
		\end{equation}
	\end{enumerate}
\end{theorem}
For the proof of this theorem, we first derive four simple lemmas.
\begin{lemma}\label{lmm:convexAffine}
	Let $r,s\in \R$ with $r \leq s$.
	For $y\in\mc Q$, let $g_y\colon[r,s]\to\R$ be convex.
	Assume $G \colon [r,s]\to\R$ with $G(t) = \Eof{g_Y(t)}$ exists and $\min_{t\in[r,s]} G(t) = G(r) = G(s)$. Then
	$G$ is constant on $[r,s]$ and $g_y$ is affine almost surely, i.e, there are $a_y,b_y\in\R$ (measurable functions $\mc Q \to \R$) such that $g_Y(t) = a_Y + b_Y t$ for all $t\in[r,s]$ almost surely.
\end{lemma}
\begin{proof}
	If $r = s$, the statement is trivial. Assume $r < s$.
	As all $g_y$ are convex, so is $G$. Thus, $G(t) = G(r) = G(s)$ for all $t\in[r,s]$. Thus,
	\begin{align}
		0 &=
		\frac{s-t}{s-r}G(r) + \frac{t-r}{s-r}G(s) - G(t)
		\\&=
		\EOf{\frac{s-t}{s-r}g_Y(r) + \frac{t-r}{s-r}g_Y(s) - g_Y(t)}
		\eqfs
	\end{align}
	As $g_y$ is convex,
	\begin{equation}
		\frac{s-t}{s-r}g_y(r) + \frac{t-r}{s-r}g_y(s) - g_y(t) \geq 0\eqfs
	\end{equation}
	Therefore,
	\begin{equation}
		\PrOf{g_Y(t) = \frac{s-t}{s-r}g_Y(r) + \frac{t-r}{s-r}g_Y(s)} = 1
		\eqfs
	\end{equation}
	Hence, we can set
	\begin{align}
		a_y &:= \frac{s}{s-r}g_y(r) + \frac{-r}{s-r}g_y(s)\eqcm\\
		b_y &:= \frac{-1}{s-r}g_y(r) + \frac{1}{s-r}g_y(s)\eqfs
	\end{align}
\end{proof}
\begin{lemma}\label{lmm:affineDist}
	Assume $(\mc Q, d)$ is Hadamard.
	Let $y\in\mc Q$ and $\gamma\in\GeodUS$. Assume $\ol y\gamma$ is affine on $I_\gamma$. Then
	\begin{enumerate}[label=(\roman*)]
		\item there is $\tilde\gamma\in\GeodUS$ with $\gamma\subset\tilde\gamma$ and $y\in\tilde\gamma$,
		\item there is $t_0\in\R\setminus\mathring{I}_\gamma$ such that $\ol y\gamma(t) = \abs{t-t_0}$ on $I_\gamma$.
	\end{enumerate}
\end{lemma}
\begin{proof}
	If $\# I_\gamma = 1$, the statement is trivial. Assume $I_\gamma$ has positive length. Then there is a point in the interior, $s\in\mathring{I}_\gamma$. As $\ol y\gamma$ is $\mc G$-convex, it has a $\mc G$-tangent $g\in\mc G$ at $s$. Then $g\prr(s)\leq 0$, as $g$ is touched from above by an affine function. The only elements of $\mc G$ that do not have a positive second derivative are of the form $t\mapsto \abs{t-t_0}$ for some $t_0\in\R$. As $g$ is $\mc G$-tangent to $\ol y\gamma$ and $\ol y\gamma$ is affine, we have $g(t) = \ol y\gamma(t) = \abs{t - t_0}$ for $t\in I_\gamma$, where $t_0\not\in\mathring{I}_\gamma$.

	Assume $I_\gamma = [t_-, t_+]$ for $t_-, t_+ \in\R$.
	Without loss of generality, assume $t_0 \leq t_-$ so that $\ol y\gamma(t) = t - t_0$ for $t\in I_\gamma$. Let $p:=\gamma(t_-)$. Let $\tilde\gamma\colon[t_0, t_+]\to\mc Q$ be the concatenation of $\geodft yp$ and $\gamma$, i.e.,
	\begin{equation}
		\tilde\gamma(t) :=
		\begin{cases}
			\geodft yp(t-t_0) & \text{ if } t\in[t_0, t_-]\eqcm\\
			\gamma(t)  & \text{ if } t\in[t_-, t_+]\eqfs
		\end{cases}
	\end{equation}
	Clearly, $y\in\tilde \gamma$ and $\gamma\subset\tilde\gamma$.
	Furthermore, $\ol y{\tilde\gamma}(t) = t- t_0$ for all $t\in I_{\tilde\gamma}$.
	Let $t_1 \in [t_0,  t_-]$ and $t_2\in[t_-, t_+]$. Using the triangle inequality, we obtain on one hand,
	\begin{equation}
		\ol{\tilde\gamma(t_1)}{\tilde\gamma(t_2)} \geq \ol{y}{\gamma}(t_2) - \ol{y}{\geodft yp}(t_1-t_0) = t_2 - t_1
	\end{equation}
	and on the other hand,
	\begin{equation}
		\ol{\tilde\gamma(t_1)}{\tilde\gamma(t_2)} \leq \ol{p}{\geodft yp}(t_1-t_0) + \ol{p}{\gamma}(t_2) = t_2 - t_1
		\eqfs
	\end{equation}
	Thus, $\ol{\tilde\gamma(r)}{\tilde\gamma(s)} = \abs{r - s}$ for all $r, s \in [t_0, t_+]$. Hence, $\tilde\gamma\in\GeodUS$. Similar arguments can be employed if $I_\gamma$ is not a closed and bounded interval of the form $I_\gamma = [t_-, t_+]$.
\end{proof}
\begin{lemma}\label{lmm:geodesicSegmentCharac}
	Assume $(\mc Q, d)$ is Hadamard.
	Let $\mc I\subset\N$ be finite.
	Let $y,p_i\in\mc Q$ for $i\in\mc I$. Let $T \in \Rpp$. Assume there are $\gamma_{ij} \in \GeodUS$ for $i,j\in\mc I$ with $I_{\gamma_{ij}} = [0, T]$ such that $p_i, p_j\in \gamma_{ij}$ and $y = \gamma_{ij}(0)$.
	Then $p_i\in\geodft{y}{p_{\ell}}$ for all $i\in\mc I$, where $\ell := \argmax_{i\in\mc I} \ol yp_i$.
\end{lemma}
\begin{proof}
	Let $\gamma = \geodft{y}{p_{\ell}}$. Then $p_{\ell}\in\gamma$. As geodesics are unique in Hadamard spaces and $\gamma_{j\ell}$ also connects $y$ and $p_\ell$, we have $\gamma\subset\gamma_{j\ell}$ for all $j\in\mc I$. As $\gamma_{j\ell}(\ol yp_j) = p_j$ and $\ol yp_j \leq \ol yp_{\ell}$ for $j\in\mc I$, we have $p_j\in\gamma$.
\end{proof}
\begin{lemma}\label{lmm:geodesicSegmentFour}
	Assume $(\mc Q, d)$ is Hadamard.
	Let $M\subset\mc Q$ be nonempty, closed, bounded, and convex. If every $3$ points in $M$ lie on a common geodesic, then $M$ is a geodesic segment.
\end{lemma}
\begin{proof}
	If $\# M=1$, the statement is trivial. As $M$ is nonempty, we assume $\# M>1$.
	Let $p,q \in M$, $p\neq q$.

	Let $\Lambda\subset \Gamma_1$ be the set of all unit-speed geodesics $\gamma\colon[0, b_\gamma]\to\mc Q$ with $b_\gamma\in[\ol qp, \infty)$ such that $\gamma(0) = q$, $\gamma(\ol qp) = p$, and $\gamma\subset M$. As $M$ is convex, this set contains at least one element, $\geodft qp\in\Lambda$.
	As $M$ is bounded, we have $b_\infty := \sup_{\gamma\in\Lambda} b_\gamma < \infty$.
	Let $(\gamma_n)_{n\in\N}\subset\Lambda$ be a sequence such that $(b_{\gamma_n})_{n\in\N}\subset\R$ is a nondecreasing sequence with $b_{\gamma_n} \xrightarrow{n\to\infty} b_\infty$.
	Let $n,k\in\N$, $k < n$. By the assumption of the lemma, the points $q, \gamma_k(b_k), \gamma_n(b_n)$ lie on a common geodesic, say $\gamma$. As geodesics are unique in Hadamard spaces, $\gamma_k,\gamma_n\subset \gamma$. Hence, the restriction of $\gamma_n$ to $[0, b_k]$ is exactly $\gamma_k$.
	With this, $\gamma_\infty \colon [0, b_\infty) \to \mc Q$ is well-defined by $\gamma_\infty(t) = \gamma_n(t)$ for any $n\in\N$ with $b_n \geq t$.
	As $\mc Q$ is complete, $M$ is closed, and $\gamma_\infty\subset M$, we obtain the existence of $m := \lim_{t\nearrow b_\infty} \gamma_\infty(t)$ with $m\in M$.
	Thus, we can extend the domain of $\gamma_\infty$ to $[0, b_\infty]$ with $\gamma_\infty(b_\infty) = m$. By construction, $\gamma_\infty\in\GeodUS$. Furthermore, for $t > b_\infty$ and any $\gamma\in\GeodUS$ with $\gamma(0) = q$, $\gamma(\ol qp) = p$, and $t\in I_\gamma$, we have $\gamma(t)\not\in M$.

	We can do the same in the other direction (the roles of $p$ and $q$ swapped) to obtain $\bar\gamma\colon[a_\infty,b_\infty]\to\mc Q$ such that $\bar\gamma\in\GeodUS$, $\bar\gamma\subset M$, and for any $\gamma\in\GeodUS$ with $\bar\gamma\subset \gamma\subset M$, we have $\bar\gamma = \gamma$ (as sets).
	Assume there is $y\in M\setminus\bar\gamma$, then $y$, $\bar\gamma(a_\infty)$, $\bar\gamma(b_\infty)$ are connected by a geodesic $\gamma$. As $M$ is convex, $\gamma\subset M$. This contradicts the maximality property of $\bar\gamma$. Hence, $M\setminus\bar\gamma$ is empty. As $\bar\gamma\subset M$, we have $\bar\gamma = M$ and $M$ is a geodesic segment.
\end{proof}
\begin{proof}[Proof of \cref{thm:medianConcentration}]
	If $\#M = 1$, $M$ is a geodesic segment. From now on assume that $\#M > 1$. Then $\#M = \infty$ as $M$ is convex by \cref{lmm:closedBoundedConvex}. Let $m,p\in M$, $p\neq m$. Let $\gamma = \gamma_{m \to p}$. Define $G(t) := \Eof{\ol Y\gamma(t) - \ol Ym}$. We have $G(0) = G(\ol pm) = 0$ by definition of $M$. \cref{lmm:convexAffine} implies that $G(t) = 0$ for all $t\in [0, \ol pm]$ and there are $\R$-valued random variables $a_Y, b_Y$ such that $\ol Y\gamma(t) - \ol Ym = a_Y + b_Y t$ almost surely. By \cref{lmm:affineDist}, we know that $\PrOf{b_Y\in\cb{-1, 1}} = 1$. As $G(0) = G(\ol pm) = 0$, we must have $a_Y = 0$ almost surely and
	\begin{equation}\label{eq:byPmOne}
		\PrOf{b_Y = -1} = \PrOf{b_Y = 1} = \frac12
		\eqfs
	\end{equation}
	We will use this result shortly. First, another consequence of \cref{lmm:affineDist} is
	\begin{equation}\label{eq:yInGamma}
		\PrOf{Y\in \bigcup \setByEle{\bar\gamma\in\GeodUS}{\gamma_{p\to m} \subset \bar\gamma} \setminus\mathring\gamma_{p\to m}} = 1\eqfs
	\end{equation}
	For any three pairwise distinct points $p_1, p_2, p_3\in M$, we can apply \eqref{eq:yInGamma} to each pair $m=p_i, p=p_j$, $i,j\in\cb{1,2,3}$. This yields the existence of a point $y\in\mc Q$ and $\gamma_{ij}\in\GeodUS$ for $i,j\in\{1,2,3\}$ with following property: $y\not\in\mathring\gamma_{p_i\to p_j}$ and $y, p_i, p_j\in\gamma_{ij}$ for all $i,j\in\cb{1,2,3}$. With this, \cref{lmm:geodesicSegmentCharac} ensures that there is a common geodesic that contains $p_1,p_2, p_3$. Furthermore, by \cref{lmm:closedBoundedConvex}, $M$ is convex, bounded, and closed. Thus, \cref{lmm:geodesicSegmentFour} implies that $M$ is a geodesic segment. Hence, there are $m_{\ms{L}}, m_{\ms{R}} \in M$ such that $M = \geodft{m_{\ms{L}}}{m_{\ms{R}}} =: \gamma_M$. Applying \eqref{eq:yInGamma} to $m_{\ms{L}}, m_{\ms{R}}$, yields $\Prof{Y\in\mathring \gamma_M} = 0$, and our earlier observation \eqref{eq:byPmOne} implies the remaining part of \eqref{eq:medianLeftRightHalfProb}.
\end{proof}
\begin{remark}[on \cref{thm:medianConcentration}]\mbox{ }\label{rem:nearmedianunique}
    \begin{enumerate}[label=(\roman*)]
        \item
        On the real line, $m\in\R$ is a median of $Y$ if and only if $\Prof{Y\geq m}\geq \frac12$ and $ \Prof{Y\leq m} \geq \frac12$. Furthermore, if the median is not unique,  $\Prof{Y\geq m} = \Prof{Y\leq m} = \frac12$ for all medians $m$. The latter statement is generalized in \eqref{eq:medianLeftRightHalfProb} to Hadamard spaces.
        \item
        Assume there are $p, m\in M$ with $p\neq m$. Define
        \begin{equation}
            \mc Y := \bigcup \setByEle{\bar\gamma\in\GeodUS}{\gamma_{p\to m} \subset \bar\gamma}
            \eqfs
        \end{equation}
        Then by \cref{thm:medianConcentration}, $\Prof{Y\in\mc Y} = 1$. Furthermore, $\mc Y = L(\gamma_M) \cup \mathring\gamma_M \cup  R(\gamma_M)$, where $\gamma_M\in\GeodUS$ is the geodesic with $M = \gamma_M$ and the unions are disjoint.
    \end{enumerate}
\end{remark}
\begin{definition}
	Assume $(\mc Q, d)$ is a geodesic space, see \cref{def:geodesic}.
	We say that $(\mc Q, d)$ is \textit{non-branching} \cite[Definition 2.8]{sturm2006on1} if following condition is fulfilled:
	Let $y,p,q\in\mc Q$ and $\geodft yp$, $\geodft yq$ unit-speed geodesics connecting $y$ and $p$, and $y$ and $q$, respectively.
	If $\geodft yp(\ol yp/2) = \geodft yq(\ol yq/2)$, then $q = p$.
\end{definition}
According to \cite[Remark 2.9]{sturm2006on1}, if $(\mc Q, d)$ is a geodesic space with finite lower curvature bound, then it is non-branching.
\begin{corollary}\label{cor:unqiueCriteriaGeod}
	Assume $(\mc Q, d)$ is Hadamard.
	\begin{enumerate}[label=(\roman*)]
		\item If $Y$ is not concentrated on a union of geodesics that all intersect in a common geodesic segment of positive length, then the Fréchet median is unique.
		\item Assume $(\mc Q, d)$ is non-branching. If $Y$ is not concentrated on a geodesic, then the Fréchet median is unique.
	\end{enumerate}
\end{corollary}
\begin{proof}
	This is a direct consequence of \cref{thm:medianConcentration}.
\end{proof}
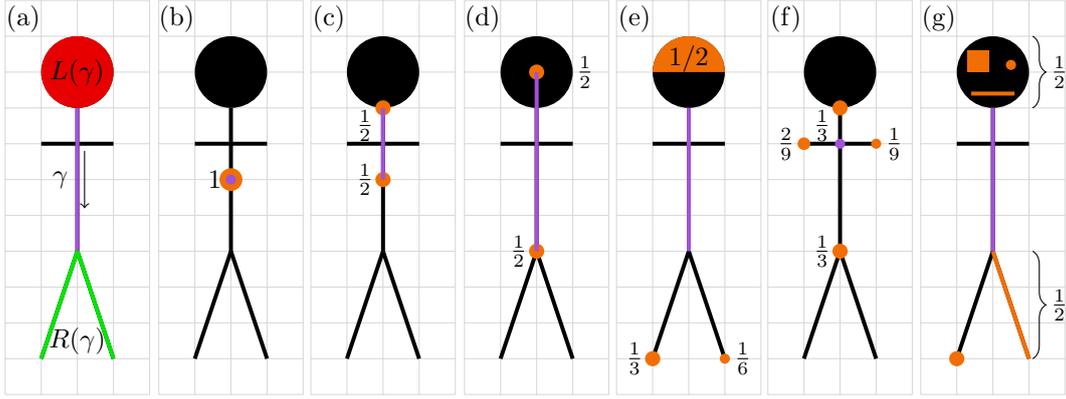
\begin{figure}
	\newcommand{\drawBase}[1]{%
		\draw[gray!30,step=0.5] (-1,-4.5) grid (1,1);
		\node at (-0.75,0.75) {(#1)};
		%
		\coordinate (headTop) at (0,+0.5);
		\fill[spacestyle] (0,0) coordinate (headCenter) circle (0.5);
		\coordinate (bodyCenter) at (0,-1.5);
		\draw[spacestyle] (0,-0.5) coordinate (bodyTop) -- (0,-2.5) coordinate (bodyBottom);
		\draw[spacestyle] (0,-1.) coordinate (rightArmInner) -- (0.5,-1.) coordinate (rightArmOuter) ;
		\draw[spacestyle] (-0.5,-1.) coordinate (leftArmOuter) -- (0,-1.) coordinate (leftArmInner);
		\draw[spacestyle] (0,-2.5) coordinate (leftLegTop) -- (-0.5,-4) coordinate (leftLegBottom);
		\draw[spacestyle] (0,-2.5) coordinate (rightLegTop) -- (0.5,-4) coordinate (rightLegBottom);
		\coordinate (bodyTopHalf) at (0,-1);
		\coordinate (bodyBottomHalf) at (0,-2);
	}%
	\begin{center}
	\begin{tikzpicture}[
		spacestyle/.style={line width=1.5pt,color=metricSpacepColor,fill=metricSpacepColor},
		pointmass/.style={fill=myOrange},
		meanseg/.style={draw=myBlue,line width=1.5pt},
		scale=0.948
		]
		\drawBase{a}
		\fill[meanseg] (bodyBottom) -- (bodyTop);
		\fill[spacestyle,fill=red!90!black] (0,0) coordinate (headCenter) circle (0.501);
		\draw[line width=1.51pt,color=green!90!black] (0,-2.5) coordinate (leftLegTop) -- (-0.5,-4) coordinate (leftLegBottom);
		\draw[line width=1.51pt,color=green!90!black] (0,-2.5) coordinate (rightLegTop) -- (0.5,-4) coordinate (rightLegBottom);
		\node at (0,0) {$L(\gamma)$};
		\node at (0,-3.75) {$R(\gamma)$};
		\node[left] at (bodyCenter) {$\gamma$};
		\draw[->] (0.1, -1.1) -- (0.1, -1.9);
	\end{tikzpicture}%
	\begin{tikzpicture}[
			spacestyle/.style={line width=1.5pt,color=metricSpacepColor,fill=metricSpacepColor},
			pointmass/.style={fill=myOrange},
			meanseg/.style={draw=myBlue,line width=1.5pt},
			scale=0.948
		]
		\drawBase{b}
		\fill[pointmass] (bodyCenter) circle (4.5pt);
		\node[left] at (bodyCenter) {$1$};
		\fill[myBlue] (bodyCenter) circle (2pt);
	\end{tikzpicture}%
	\begin{tikzpicture}[
		spacestyle/.style={line width=1.5pt,color=metricSpacepColor,fill=metricSpacepColor},
		pointmass/.style={fill=myOrange},
		meanseg/.style={draw=myBlue,line width=1.5pt},
		scale=0.948
		]
		\drawBase{c}
		\fill[pointmass] (bodyTop) circle (3pt);
		\fill[pointmass] (bodyCenter) circle (3pt);
		\node[below left,yshift=0.1cm] at (bodyTop) {$\frac12$};
		\node[left] at (bodyCenter) {$\frac12$};
		\fill[meanseg] (bodyCenter) -- (bodyTop);
	\end{tikzpicture}%
	\begin{tikzpicture}[
		spacestyle/.style={line width=1.5pt,color=metricSpacepColor,fill=metricSpacepColor},
		pointmass/.style={fill=myOrange},
		meanseg/.style={draw=myBlue,line width=1.5pt},
		scale=0.948
		]
		\drawBase{d}
		\fill[pointmass] (headCenter) circle (3pt);
		\fill[pointmass] (bodyBottom) circle (3pt);
		\node[right,xshift=0.4cm] at (headCenter) {$\frac12$};
		\node[left] at (bodyBottom) {$\frac12$};
		\fill[meanseg] (bodyBottom) -- (headCenter) ;
	\end{tikzpicture}%
	\begin{tikzpicture}[
		spacestyle/.style={line width=1.5pt,color=metricSpacepColor,fill=metricSpacepColor},
		pointmass/.style={fill=myOrange},
		meanseg/.style={draw=myBlue,line width=1.5pt},
		scale=0.948
		]
		\drawBase{e}
		\fill[pointmass] (-0.5,0) -- (0.5,0) arc(0:180:0.5) --cycle;
		\fill[pointmass] (leftLegBottom) circle (3pt);
		\fill[pointmass] (rightLegBottom) circle (2pt);
		\node[below] at (headTop) {$1/2$};
		\node[left] at (leftLegBottom) {$\frac13$};
		\node[right] at (rightLegBottom) {$\frac16$};
		\fill[meanseg] (bodyBottom) -- (bodyTop);
	\end{tikzpicture}%
	\begin{tikzpicture}[
		spacestyle/.style={line width=1.5pt,color=metricSpacepColor,fill=metricSpacepColor},
		pointmass/.style={fill=myOrange},
		meanseg/.style={draw=myBlue,line width=1.5pt},
		scale=0.948
		]
		\drawBase{f}
		\fill[pointmass] (bodyTop) circle (3pt);
		\fill[pointmass] (bodyBottom) circle (3pt);
		\fill[pointmass] (leftArmOuter) circle (2.5pt);
		\fill[pointmass] (rightArmOuter) circle (2pt);
		\node[below left,yshift=0.1cm] at (bodyTop) {$\frac13$};
		\node[left] at (bodyBottom) {$\frac13$};
		\node[left] at (leftArmOuter) {$\frac29$};
		\node[right] at (rightArmOuter) {$\frac19$};
		\fill[myBlue] (rightArmInner) circle (2pt);
	\end{tikzpicture}%
	\begin{tikzpicture}[
		spacestyle/.style={line width=1.5pt,color=metricSpacepColor,fill=metricSpacepColor},
		pointmass/.style={fill=myOrange},
		meanseg/.style={draw=myBlue,line width=1.5pt},
		scale=0.948
		]
		\drawBase{g}
		\fill[pointmass] (-0.35, 0.3) rectangle (-0.05, 0.0);
		\fill[pointmass] (0.25, 0.1) circle (2pt);
		\draw[pointmass,color=myOrange,line width=1.5pt] (-0.3, -0.3) -- (0.3, -0.3);
		\draw[pointmass,color=myOrange,line width=1.6pt] (rightLegTop) -- (rightLegBottom);
		\fill[pointmass] (leftLegBottom) circle (3pt);
		\fill[meanseg] (bodyBottom) -- (bodyTop);
		\draw [decorate,decoration={brace,amplitude=5pt}] (0.55, 0.5) -- (0.55,-0.5) node[midway,xshift=1em]{$\frac12$};
		\draw [decorate,decoration={brace,amplitude=5pt}] (0.55,-2.5) -- (0.55,-4) node[midway,xshift=1em]{$\frac12$};
	\end{tikzpicture}%
	\end{center}
	\caption{%
		The Fréchet median in the stick figure space. The stick figure shown here is to be taken as a subset of $\R^2$ with its intrinsic distance. It is a Hadamard space: Convex subsets of the Euclidean plane are Hadamard spaces. Moreover, gluing Hadamard spaces together yields a new Hadamard space according to Reshetnyak’s Gluing Theorem \cite[Theorem 3.12]{sturm03}. Thus, the stick figure, which is obtained from gluing together a disk and some line segments is Hadamard.
		For (a): Let $\gamma\in\GeodUS$ be the geodesic starting at the stick figures neck such that its image is the purple torso. Then the red head is the left set $L(\gamma)$ and the green legs make up  the right set $R(\gamma)$. The arms do not belong to either set.
		In (b) -- (f) the orange color depicts a distribution where small circles with a number on the side indicate a point mass with that probability and the orange half circle in (e) is to be taken as a uniform distribution on the colored area of one half of the total mass. In (f) the probability mass in the legs and in the head each make up one half in total. The Fréchet median set is shown in purple. It is either a geodesic segment of positive length in (c) -- (e), (g) or a single point in (b), (f).}\label{fig:stickfig}
\end{figure}
A simple criterion for uniqueness for any $\tran$-Fréchet mean with $\tran\in\setcciz$ is the following.
\begin{corollary}\label{cor:convexUnique}
	Let $\tran\in\setcciz$.
	Assume $(\mc Q, d)$ is Hadamard and separable.
    Assume $\Eof{\dtran(\ol Yo)} < \infty$.
    Assume that the support of $Y$ is convex.
    Then the $\tran$-Fréchet mean of $Y$ is unique.
\end{corollary}
\begin{proof}
	Let $x_0 = \inf\setByEleInText{x\in\Rpp}{\ddrtran(x) = 0}$.
	Let $M = \argmin_{q\in\mc Q}\Eof{\tran(\ol Yq) - \tran(\ol Yo)}$ be the set of $\tran$-Fréchet means of $Y$. Assume $\# M >1 $. Then \cref{cor:nonuniquetran} implies that $\PrOf{\ol Ym < x_0} = 0$ for all $m\in M$.
    Let $\mc Y$ denote the support of $Y$. It is closed by definition and assumed to be convex. From \cref{lmm:closedBoundedConvex}, we then obtain $M \subset \mc Y$. Thus, for all $\epsilon\in\Rpp$, $\PrOf{\ol Ym < \epsilon} > 0$. Hence, $x_0 = 0$, which means that $\tran$ is linear.

    By \cref{thm:medianConcentration}, $\mc Y$ is a subset of a union of geodesics that all intersect in the geodesics segment $\gamma_M = M$ for a $\gamma_M\in\GeodUS$. Furthermore, $M \subset \mc Y$ and $\Prof{Y \in \mathring\gamma_M} = 0$. Let $m \in \mathring\gamma_M$. On one hand, as $m \in \mc Y$, for all $\epsilon>0$, we have $\Prof{Y \in \ball m\epsilon} > 0$. On the other hand, there is $\epsilon>0$ such that $\ball{m}{\epsilon}\cap \mc Y \subset \mathring\gamma_M$, which implies
    \begin{equation}
        \PrOf{Y \in \ball{m}{\epsilon}} = \PrOf{Y \in \ball m\epsilon \cap \mc Y} \leq \PrOf{Y \in \mathring\gamma_M} = 0
        \eqfs
    \end{equation}
    As this is contradictory, our initial assumption that $\# M > 1$ must be wrong.
\end{proof}
\subsection{Variance Inequality}\label{ssec:median:varineq}
For $\eta\in[0,1]$ and $p,q\in\mc Q$, $q\neq p$, define the \textit{bowtie complement} (see \cref{fig:medianCutOut})
\begin{equation}\label{eq:bowtie}
    A(p, q, \eta) := \setByEle{y \in \mc Q}{\max\brOf{\ol y{\geodft pq}\rd(0)^2, \ol y{\geodft pq}\ld(\ol qp)^2} \leq 1-\eta^2}
    \eqfs
\end{equation}
\begin{theorem}\label{thm:near:median}
    Assume $(\mc Q, d)$ is Hadamard.
    Let $m \in \argmin_{q\in\mc Q}\Eof{\ol Yq - \ol Yo}$.
    Let $q\in\mc Q\setminus \cb{m}$.
    Then
    \begin{equation}\label{eq:near:median}
        \EOf{\ol Yq - \ol Ym} \geq \frac12 \eta^2 \,\ol qm^2\, \EOf{\max\brOf{\ol Ym, \ol Yq}^{-1} \indOf{A(m, q, \eta)}(Y)}
        \eqfs
    \end{equation}
\end{theorem}
\begin{proof}
    Set $\gamma := \geodft mq$.
    For $t\in[0, \ol qm]$, define
    \begin{align}
        G(t) &:= \EOf{\ol Y\gamma(t) - \ol Y\gamma(0)}
        \eqfs
    \end{align}
    We want to swap limit and integral to calculate $G\rd(0)$. For this we use dominated convergence, which we can apply as
    \begin{equation}
        \EOf{\sup_{t\in(0,\ol qm]}\abs{\frac{\ol Y\gamma(t)-\ol Y\gamma(0)}{t}}}
        \leq 1\eqcm
    \end{equation}
    as $\ol Y\gamma$ is 1-Lipschitz.
    As $G(t)$ is convex and minimized at $t=0$, we have
    \begin{equation}\label{eq:near:median:1}
        0 \leq  G\rd(0) = \Eof{\ol Y\gamma\rd(0)}
        \eqfs
    \end{equation}
    Applying \cref{lmm:semitaylor} \ref{lmm:semitaylor:median} yields, for $t\in(0, \ol qm]$,
    \begin{equation}\label{eq:near:median:2}
        G(t)
        \geq
        t  G\rd(0) + \frac12 t^2 \EOf{\frac{1 - \max\brOf{\ol Y\gamma\rd(0)^2, \ol Y\gamma\ld(t)^2}}{\max\brOf{\ol Y\gamma(0), \ol Y\gamma(t)}}}
        \eqfs
    \end{equation}
    Now fix $t = \ol qm$. As $\ol Y\gamma$ is $1$-Lipschitz, we have
    \begin{align}
        &\EOf{\frac{1 - \max\brOf{\ol Y\gamma\rd(0)^2, \ol Y\gamma\ld(t)^2}}{\max\brOf{\ol Y\gamma(0), \ol Y\gamma(t)}}}
        \\&\geq
        \EOf{\frac{1 - \max\brOf{\ol Y\gamma\rd(0)^2, \ol Y\gamma\ld(t)^2}}{\max\brOf{\ol Y\gamma(0), \ol Y\gamma(t)}} \indOf{A(m, q, \eta)}(Y)}
        \\&\geq
        \eta^2 \EOf{\max\brOf{\ol Y\gamma(0), \ol Y\gamma(t)}^{-1} \indOf{A(m, q, \eta)}(Y)}
        \eqfs
    \end{align}
    Thus, together with \eqref{eq:near:median:1} and \eqref{eq:near:median:2}, we obtain
    \begin{equation}
        G(t)
        \geq
        \frac12 t^2 \eta^2 \EOf{\max\brOf{\ol Y\gamma(0), \ol Y\gamma(t)}^{-1} \indOf{A(m, q, \eta)}(Y)}
        \eqfs
    \end{equation}
\end{proof}
\begin{remark}[on \cref{thm:near:median}]\label{rem:near:median}\mbox{ }
    \begin{enumerate}[label=(\roman*)]
        \item\label{rem:near:median:hilbert}
        Let $(\mc Q, d)$ be a separable Hilbert space with inner product $\ip{\cdot}{\cdot}$, induced norm $\norm$, and induced metric $d$. Let $v \in \mc Q$ with $\normof{v} = 1$.
        Let $\gamma\in\GeodUS$ be given by $\gamma \colon \R\to\mc Q, t\mapsto tv$. Let $y\in \mc Q$. Then
        \begin{equation}
            \ol y\gamma(t) = \sqrt{(t-t_0)^2 + h^2}\eqcm
        \end{equation}
        where
        \begin{align}
            t_0 &:= \ip{y}{v}\eqcm&	h &:= \sqrt{\normof{y}^2 - \ip{y}{v}^2}
            \eqfs
        \end{align}
        Then
        \begin{align}
            &\ol y{\gamma}\rd(t)^2 \leq 1-\eta^2
            \\\Leftrightarrow\ &\label{eq:rem:near:median:dist}
            \eta \,\ol y{\gamma}(t) \leq h
            \\\Leftrightarrow\ &
            \br{t - t_0}^2  \leq \frac{1 - \eta^2}{\eta^2} h^2
            \eqfs
        \end{align}
        As an example, $y \in A(\gamma(0), \gamma(t), \sqrt{2}/2)$ is equivalent to
        \begin{equation}
            \max\brOf{\abs{t - t_0}, \abs{t_0}} \leq h
            \eqcm
        \end{equation}
        where $\abs{t - t_0}$ is the distance between $\gamma(t)$ and $y$ projected to $\gamma$ and $h$ is the distance between $y$ and $\gamma$. See \cref{fig:medianCutOut} for an illustration.
        \item
        Assume $(\mc Q, d)$ is Hadamard and locally compact. Let $m,q,y\in\mc Q$, $q\neq m\neq y$. Let $\gamma := \geodft mq$. By the \emph{First Variational Formula} \cite[Corollary 4.5.7]{burago01}, we have $\ol y\gamma\rd(0) = -\cos(\alpha)$, where $\alpha$ is the angle \cite[Definition 3.6.26]{burago01} between $\gamma$ and $\geodft my$. Hence, we can interpret the bowtie complement $A(m, q, \eta)$ as a subset of $\mc Q$ where all geodesics are removed that intersect $\gamma$ at $\gamma(0)$ or $\gamma(\ol qm)$ with an angle of $\alpha_0$ or less, where $\alpha_0$ depends on $\eta$.
    \end{enumerate}
\end{remark}
\begin{example}
    We want to compare \cref{thm:near:median} with the variance inequality \cite[Theorem 2.3]{minsker2023geometric} for the Fréchet median in Euclidean spaces (geometric median). To this end, we apply our lower bound to the uniform distribution on a sphere $\mb S_{\sqrt{k}}^{k-1}$ of radius $\sqrt{k}$ in $\R^k$, $k \geq 2$, centered at the origin, as the authors do in \cite[Remark 2.4]{minsker2023geometric}.
    The Fréchet median is $m=0$. Let $q\in \R^k$ with $0 < \normof{q} \leq \sqrt{k}$. Because of symmetry, without loss of generality we can assume  $q = (r,0,\dots,0)\in\R^k$ with $r:=\normof{q}$. Let $s\in[0,1]$ and set $A_s := \setByEleInText{y \in \R^k}{\normof{y}^2 = k \text{ and } \abs{y_1} \leq s \sqrt{k}}$.
    If $y \in A_s$, we have $\max\brOf{\normof{y}, \normof{y-q}}^2 \leq 2(1+s)k$. Hence, \cref{thm:near:median} yields
    \begin{align}
        \EOf{\normof{Y-q} - \normof{Y}}
        &\geq
        \frac12 \eta^2 \normof q^2\, \EOf{\max\brOf{\normof{Y}, \normof{Y-q}}^{-1} \indOf{A(0, q, \eta)}(Y)}
        \\&\geq
        \frac1{2\sqrt{2(1+s)}} \eta^2 \normof q^2\, k^{-\frac12} \PrOf{Y \in A(0, q, \eta) \cap A_s}
        \eqfs
    \end{align}
    Using \eqref{eq:rem:near:median:dist}, we have
    \begin{align}
        A(0, q, \eta) \cap A_s
        &=
        \setByEle{y \in \R^k}{\eta^2 \max\brOf{\normof{y}^2, \normof{y - q}^2} \leq \normof{y}^2 - \ip{y}{v}^2}\cap A_s
        \\&\supseteq
        \setByEle{y \in \R^k}{2 \eta^2 (1+s)k \leq k - y_1^2} \cap A_s
        \\&=
        \setByEle{y \in \mb S_{\sqrt{k}}^{k-1}}{y_1^2 \leq \min\brOf{s^2 k, k - 2 \eta^2 (1+s)k}}
        \eqfs
    \end{align}
    The random variable $Y$ can be written as $Y = \sqrt{k} X / \normof{X}$ where $X$ is a standard normal random vector in dimension $k$. Let $U := \sum_{i=2}^k X_i^2$. Then $U$ and $X_1^2$ are independent and have a $\chi^2$ distribution with $k-1$ and $1$ degree of freedom, respectively. Hence,
    \begin{equation}
        B_k := k^{-1} Y_1^2 = \frac{X_1^2}{X_1^2 + U} \sim \mathrm{Beta}\brOf{\frac12, \frac{k-1}2}
        \eqcm
    \end{equation}
    where $\mathrm{Beta}(\alpha, \beta)$, $\alpha,\beta\in\Rpp$ is the Beta distribution with density
    \begin{equation}
        [0,1] \to \R,\, x\mapsto c_{\alpha,\beta} x^{\alpha-1} (1-x)^{\beta-1}
    \end{equation}
    for a constant $c_{\alpha,\beta}\in\Rpp$. Hence,
    \begin{align}
        \PrOf{Y_1^2 \leq \min\brOf{s^2 k, k - 2 \eta^2 (1+s)k}}
        &=
        \PrOf{B_k \leq  \min\brOf{s^2, 1 - 2 \eta^2 (1+s)}}
        \eqfs
    \end{align}
    We obtain the lower bound on the variance functional,
    \begin{align}
        \EOf{\normof{Y-q} - \normof{Y}}
        &\geq
        \frac1{2\sqrt{2+2s}} \eta^2 \normof q^2\, k^{-\frac12} \PrOf{B_k \leq  \min\brOf{s^2, 1 - 2 \eta^2 (1+s)}}
        \eqcm
    \end{align}
    for $q\in \R^k$ with $0 < \normof{q} < \sqrt{k}$, where $B_k \sim \mathrm{Beta}\brOf{\frac12, \frac{k-1}2}$. Let $\epsilon\in(0,1]$. Then
    \begin{equation}
        \PrOf{B_k \leq \epsilon} \xrightarrow{k\to\infty} 1
        \eqfs
    \end{equation}
    Set $s = \sqrt{\epsilon}$ and $\eta^2 = \frac{1-\epsilon}{2(1+s)}$ to obtain
    \begin{equation}\label{eq:exa:lowerepsi}
        \liminf_{k\to\infty}\inf_{q\in\ball{0}{\sqrt{k}}}
        \frac{\EOf{\normof{Y-q} - \normof{Y}}}
        {\normof q^2\, k^{-\frac12}}
        \geq
        \frac1{2\sqrt{2(1+\sqrt{\epsilon})}}  \frac{1-\epsilon}{2(1+\sqrt{\epsilon})}
        \eqfs
    \end{equation}
    As $\epsilon\in(0,1]$ is arbitrary, \eqref{eq:exa:lowerepsi} also holds with the lower bound $\frac{\sqrt{2}}8$.
    Thus, we obtain a better constant than derived in \cite[Remark 2.4]{minsker2023geometric}.
\end{example}
\begin{figure}
    \begin{center}
        \begin{tikzpicture}[blend group=normal]
            \fill[black!10] (-6,-3) rectangle (6,3);
            \draw[draw=none,pattern={Dots[radius=1pt,distance=5pt,xshift=2.5pt,yshift=2.5pt]},pattern color=myBlue] (-6,-2) -- (6,2) -- (6,-2) -- (-6,2) -- cycle;
            \draw[draw=none,pattern={Dots[radius=1pt,distance=5pt]},pattern color=myOrange] (-6,-8/3) -- (6,4/3) -- (6,-4/3) -- (-6,8/3) -- cycle;
            \draw[line width=0.5pt,myBlue] (-6,-2) -- (6,2);
            \draw[line width=0.5pt,myBlue] (6,-2) -- (-6,2);
            \draw[line width=0.5pt,myOrange] (-6,-8/3) -- (6,4/3);
            \draw[line width=0.5pt,myOrange] (6,-4/3) -- (-6,8/3);
            \draw[line width=1pt] (-6,0) -- (6,0);
            \fill (0,0) circle (2pt);
            \fill (2,0) circle (2pt);
            \node[above,fill=black!10,yshift=2pt, inner sep=1pt] at (-5,0) {$\gamma$};
            \node[below,fill=black!10,yshift=-2pt, inner sep=1pt] at (0,0) {$m$};
            \node[below,fill=black!10,yshift=-2pt, inner sep=1pt] at (2,0) {$q$};
            \node at (1,2) {$A(m, q, \eta)$};
        \end{tikzpicture}
        \caption{Visualization of the bowtie complement $A(m, q, \eta)$ in $\R^2$. The Euclidean plane is depicted as a gray area. The image of a geodesic $\gamma\colon\R\to\R^2$ with $\gamma(0) = m$ and $\gamma(\ol qm) = q$ is shown in black. The sets $\setByEleInText{y\in\R^2}{\ol y\gamma\rd(0)^2 > 1 - \eta^2}$ and $\setByEleInText{y\in\R^2}{\ol y\gamma\ld(\ol qm)^2 > 1 - \eta^2}$ are depicted as purple and orange dotted areas, respectively. The set $A(m, q, \eta)$ is the gray area without the dotted areas.}\label{fig:medianCutOut}
    \end{center}
\end{figure}
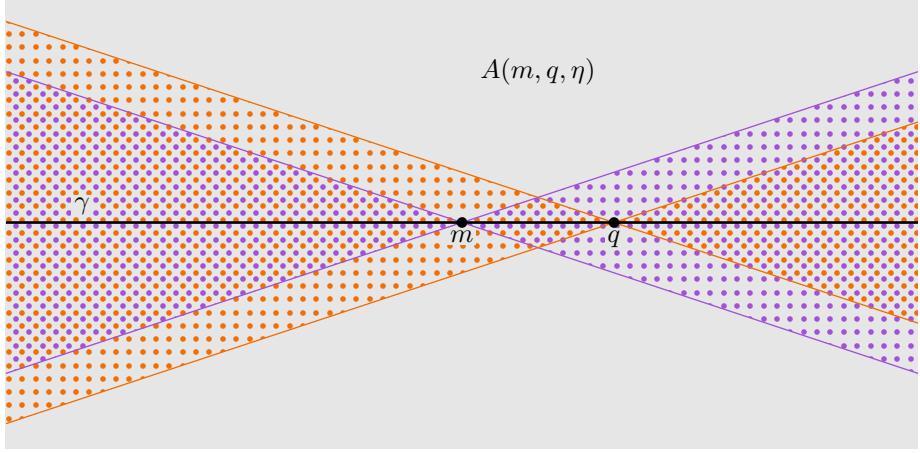
With \cref{thm:near:median}, we can show quadratic and faster variance inequalities for the Fréchet median in Hadamard spaces. To able to deal easily with the bowtie complement $A(m, q, \eta)$, we restrict ourselves to Euclidean spaces in the next corollary.
\begin{corollary}\label{cor:varinequ:median}
	Let $k\in\N$, $k\geq 2$, and $(\mc Q, d) = (\R^k, \normof{\cdot - \cdot})$ be Euclidean.
	Let $m \in \argmin_{q\in\R^k}\Eof{\normof{ Y-q} -\normof{Y}}$ be a Fréchet median of $Y$.
	Assume that $Y$ has a Lebesgue density $\rho \colon \R^k\to\Rp$ with constants $\delta,c \in\Rpp$, $\zeta\in[0,1)$ with one of the following properties:
    \begin{enumerate}[label=(\roman*)]
        \item
        We have $\zeta = 0$ and, for all $x\in \R^k$ with $\normof{x} = \delta$, $\rho$ is continuous at $m+x$ and
        \begin{equation}\label{eq:varinequ:median:density:const}
            \rho(m+x) \geq c
            \eqfs
        \end{equation}
        \item
        We have $\zeta \in (0,1)$ and
        \begin{equation}\label{eq:varinequ:median:density}
            \rho(m+x) \geq c \normof{x}^{-\zeta-k+1}
        \end{equation}
        for all $x\in\ball 0\delta\setminus\cb{0}$.
    \end{enumerate}
    Then $m$ is the only Fréchet median of $Y$ and there is $C \in \Rpp$ depending only on $k, c, \zeta$ such that
	\begin{equation}
		\EOf{\normOf{Y-q} -\normOf{Y-m}} \geq C \normOf{q-m}^{2-\zeta}
	\end{equation}
	for all $q\in\ball m\delta$.
\end{corollary}
\begin{proof}
    Let $q\in\ball m\delta\setminus m$.
    As $\max(\ol ym, \ol yq) \leq 2\delta$ for $y\in\ball m\delta$, \cref{thm:near:median} yields
    \begin{equation}\label{eq:euclidean:quadratic}
        \EOf{\normOf{Y-q} -\normOf{Y-m}}
        \geq
        \frac14 \delta^{-1} \eta^2 \ol qm^2  \PrOf{Y \in A(m, q, \eta) \cap \ball m\delta}
        \eqfs
    \end{equation}
    For the case $\zeta = 0$, note that $A(m, q, \eta) \subset A(m, \tilde q, \eta)$ for $\tilde q = m + t (q - m)$ with $t\in(0,1)$.
    Hence, for any constant $\eta \in (0,1)$, one can show a quadratic lower bound using \eqref{eq:euclidean:quadratic} and \eqref{eq:varinequ:median:density:const}.

    Now consider the case $\zeta > 0$.
	Without loss of generality, $m=0$ and $q=(r,0, \dots, 0)\tr$ for $r = \normof{q} \in (0 ,\delta)$.
	If $\normOf{y} \leq r$, then $\max\brOf{\normOf{y-m}, \normOf{y-q}}^{-1} \geq (2r)^{-1}$.
	Thus, we obtain from \cref{thm:near:median}
	\begin{equation}
		\EOf{\normOf{Y-q} -\normOf{Y}}
		\geq
		\frac14 \eta^2 r  \PrOf{\cb{\normOf{Y} \leq r}\cap \cb{Y \in A(m, q, \eta)}}
		\eqfs
	\end{equation}
	Fix $\eta = 5^{-\frac12}$.
	Let
	\begin{equation}
		\tilde A_r := \setByEle{y\in\R^k}{\normOf{y} \in \br{\frac 34 r, r}\eqcm\, \frac{y_1}{\normof{y}} \in \br{0, \frac2{\sqrt{5}}}}
		\eqfs
	\end{equation}
	One can show that $\tilde A_r \subset A(0, q, \eta) \cap \ball0{r}$, see \cref{fig:bowtie:ball}.
	Thus, using \eqref{eq:varinequ:median:density},
	\begin{align*}
		\PrOf{\cb{\normOf{Y} \leq r}\cap \cb{Y \in A(m, q, \eta)}}
		&\geq
		c \int_{y\in\tilde A_r}  \normOf{y}^{-\zeta-k+1}\dl  y
		\\&=
		c \int_{\frac 34r}^r s^{-\zeta-k+1} \int_{y\in\tilde A_r, \normof{y}=s}  1 \dl  x  \dl  s
		\\&\geq
		c c_{k,\zeta} r^{1-\zeta}
		\eqcm
	\end{align*}
	where $c_{k,\zeta} \in\Rpp$ only depends on $k$ and $\zeta$.
	Hence,
	\begin{equation}
		\EOf{\normOf{Y-q} -\normOf{Y}}
		\geq
		\frac1{20} c c_{k,\zeta} r^{2-\zeta}
		\eqfs
	\end{equation}
\end{proof}
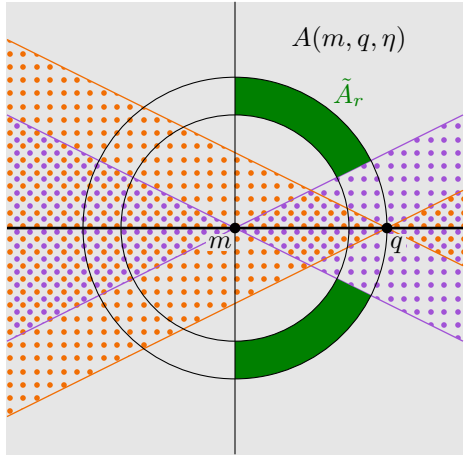
\begin{figure}
	\begin{center}
		\begin{tikzpicture}[blend group=normal]
			\fill[black!10] (-3,-3) rectangle (3,3);
			\draw[draw=none,pattern={Dots[radius=1pt,distance=5pt,xshift=2.5pt,yshift=2.5pt]},pattern color=myBlue] (-3,-1.5) -- (3,1.5) -- (3,-1.5) -- (-3,1.5) -- cycle;
			\draw[draw=none,pattern={Dots[radius=1pt,distance=5pt]},pattern color=myOrange] (-3,-2.5) -- (3,0.5) -- (3,-0.5) -- (-3,2.5) -- cycle;
			\draw[line width=0.5pt,myBlue] (-3,-1.5) -- (3,1.5);
			\draw[line width=0.5pt,myBlue] (3,-1.5) -- (-3,1.5);
			\draw[line width=0.5pt,myOrange] (-3,-2.5) -- (3,0.5);
			\draw[line width=0.5pt,myOrange] (3,-0.5) -- (-3,2.5);
			\draw[line width=1pt] (-3,0) -- (3,0);
			\fill (0,0) circle (2pt);
			\fill (2,0) circle (2pt);
			\node[below left,fill=black!10,yshift=-2pt, inner sep=1pt] at (0,0) {$m$};
			\node[below right,fill=black!10,yshift=-2pt, inner sep=1pt] at (2,0) {$q$};
			\node at (1.5,2.5) {$A(m, q, \eta)$};
			\node[green!50!black] at (1.5,1.8) {$\tilde A_r$};
			\fill[green!50!black] (0,2) -- (0, 1.5) arc (90:26.56505118:1.5) -- (1.788854382,0.8944271910)  arc (26.56505118:90:2);
			\fill[green!50!black] (0,-2) -- (0, -1.5) arc (-90:-26.56505118:1.5) -- (1.788854382,-0.8944271910)  arc (-26.56505118:-90:2);
			\draw (0,0) circle (2cm);
			\draw (0,0) circle (1.5cm);
			\draw (0,3) -- (0,-3);
		\end{tikzpicture}
		\caption{Construction of $\tilde A_r$ (green area) in the proof of \cref{cor:varinequ:median} illustrated in $\R^2$. Also confer \cref{fig:medianCutOut}.}\label{fig:bowtie:ball}
	\end{center}
\end{figure}
\begin{remark}\mbox{ }
    \begin{enumerate}[label=(\roman*)]
        \item
        Note that \eqref{eq:varinequ:median:density} implies
        \begin{equation}
            \PrOf{\ol Ym \leq x} \geq c_{k,\zeta} x^{1-\zeta}
        \end{equation}
        for $x \leq \delta$ and a constant $c_{k,\zeta}\in\Rpp$.
        \item
        Compare \cref{cor:varinequ:median} with \cref{cor:varinequ} using $\tran(x) := x^\alpha$, $\alpha \in (1, 2]$: Let $\delta\in\Rpp$. If
        \begin{equation}
            \PrOf{\ol Ym \leq x} \geq c x^{1-\zeta}
        \end{equation}
        for $x \leq \delta$ and a constant $c\in\Rpp$, we can choose $\beta = \zeta + 1 - \alpha$ if $\zeta \in [\alpha-1, 1)$, and \cref{cor:varinequ} yields
        \begin{equation}
            \EOf{\ol Yq^\alpha - \ol Ym^\alpha} \geq \tilde c \, \ol qm^{1 + \alpha - \zeta}
        \end{equation}
        for a constant $\tilde c\in\Rpp$.
        For $\alpha \searrow 1$, this variance inequality approaches the one obtained in \cref{cor:varinequ:median} (up to a constant factor). In this sense, the two results are consistent.
    \end{enumerate}
\end{remark}
\subsection{When Concentrated on a Geodesic}\label{ssec:median:geodesic}
We now discuss the growth behavior of the variance functional of the Fréchet median when the distribution of $Y$ is concentrated on a geodesic. This is almost equivalent to the case of the median on the real line as the image of a geodesic is isometric to a convex subset of $\R$. But we have to additionally establish a variance inequality for $q\in\mc Q$ not on the supporting geodesic. We first derive a lemma for the median on the real line.
\begin{lemma}\label{lmm:realMediVarIneq}
	Let $X$ be a real-valued random variable with median $0$.
	Define
	\begin{align}
		a_+ &:= \Prof{X > 0}\eqcm&
		a_0 &:= \Prof{X = 0}\eqcm&
		a_- &:= \Prof{X < 0}\eqfs
	\end{align}
	Then
	\begin{equation}
		\abs{a_- - a_+} \leq a_0
		\eqfs
	\end{equation}
	Let $t\in\R$. Then
	\begin{equation}
		\EOf{\abs{X-t} - \abs{X}} - \abs{t} a_0 - t \br{a_- - a_+} =
		\begin{cases}
			2\EOf{(t-X)\ind_{(0,t]}(X)} & \text{ if } t > 0\eqcm \\
			2\EOf{(X-t)\ind_{[t,0)}(X)} & \text{ if } t < 0\eqfs
		\end{cases}
	\end{equation}
\end{lemma}
\begin{proof}
	As $0$ is the median of $X$, $\Prof{X  \geq 0} \geq\frac12$ and $\Prof{X  \leq 0} \geq\frac12$. Thus,
	\begin{align}\label{eq:p0pm}
		a_- &\leq \frac12 \leq a_0 + a_+\eqcm&
		a_+ &\leq \frac12 \leq a_0 + a_-\eqfs
	\end{align}
	Therefore $\abs{a_+ - a_-} \leq a_0$.
	Assume $t > 0$.
	For all $x\in\R$, we have
	\begin{equation}
		\abs{x-t} - \abs{x} =
		\begin{cases}
			-t &\text{ if } x \geq t\eqcm\\
			t - 2x &\text{ if } x \in [0, t]\eqcm\\
			t &\text{ if } x \leq 0\eqfs
		\end{cases}
	\end{equation}
	Thus,
	\begin{equation}
		\EOf{\abs{X-t} - \abs{X}} - t a_0
		=
		- t \Prof{X > t} + \Eof{(t-2X)\ind_{(0,t]}(X)} + t a_-
		\eqfs
	\end{equation}
	Furthermore,
	\begin{align}
		- t \Prof{X > t}  + t a_- - t \br{a_- - a_+}
		&=
		t \br{a_+ - \Prof{X > t}}
		\\&=
		t \Prof{X \in (0, t]}
		\\&=
		\Eof{t\ind_{(0,t]}(X)}
		\eqfs
	\end{align}
	Taking the last two displays together, yields
	\begin{equation}
		\EOf{\abs{X-t} - \abs{X}} - t a_0 - t \br{a_- - a_+}
		=
		2\Eof{(t-X)\ind_{(0,t]}(X)}
		\eqfs
	\end{equation}
	The case $t < 0$ is similar.
\end{proof}
\begin{theorem}\label{thm:median:geodesic}
	Assume there is a geodesic $\gamma$ such that $\PrOf{Y\in\gamma}=1$. Let $m\in\argmin_{q\in\mc Q} \Eof{\ol Yq - \ol Yo}$.
	Let $q\in\mc Q$. Then the projection $p = \argmin_{z\in\gamma}\ol qz$ of $q$ onto $\gamma$ exists uniquely. Without loss of generality, assume $\gamma$ is unit-speed, $\gamma(0) = m$, and $\gamma^{-1}(p) \geq 0$.
	Define
	\begin{align}
		a_- &:= \PrOf{\gamma^{-1}(Y) < 0}\eqcm&
		a_0 &:= \Prof{Y = m}\eqcm&
		a_+ &:= \PrOf{\gamma^{-1}(Y) > 0}
		\eqfs
	\end{align}
	Then
	\begin{equation}
		\Eof{\ol Yq - \ol Ym}
		\geq
		\ol qm \, a_0 + \ol pm \br{a_- - a_+} + \EOf{\br{\ol qp + \ol pm - \ol Ym}\indOfOf{(0, \ol qp + \ol pm]}{\gamma^{-1}(Y)}}
		\eqfs
	\end{equation}
\end{theorem}
\begin{proof}
	According to \cite[Proposition 2.6]{sturm03}, the projection $p = \argmin_{z\in\gamma}\ol qz$ exists uniquely and fulfills
	\begin{align}
		\sqrt{\ol yp^2 + \ol qp^2} \leq \ol yq \leq \ol yp + \ol qp
	\end{align}
	for all $y\in \gamma$. We have
	\begin{equation}\label{eq:medianGeodSplit}
		\Eof{\ol Yq - \ol Ym}
		=
		\ol qm \,\Prof{Y = m} +  \EOf{\br{\ol Yq - \ol Yp} \indOf{Y\neq m}} +  \Eof{\br{\ol Yp - \ol Ym} \indOf{Y\neq m}}
		\eqfs
	\end{equation}
	For the third term, let $X := \gamma^{-1}(Y)$. Then $0 = \gamma^{-1}(m)$ is the median of $X$ as $\gamma$ is an isometry onto its image.
	Let $s := \gamma^{-1}(p) = \ol pm$. By \cref{lmm:realMediVarIneq},
	\begin{align}
		\Eof{\br{\ol Yp - \ol Ym} \indOf{Y\neq m}}
		&=
		\EOf{\abs{X-s} - \abs{X}} - \abs{s} a_0
		\\&=
		s \br{a_- - a_+} + 2\EOf{(s-X)\ind_{(0,s]}(X)}
		\eqfs
	\end{align}
	For the second term in \eqref{eq:medianGeodSplit},
	\begin{align}
		\EOf{\br{\ol Yq - \ol Yp} \indOf{Y\neq m}}
		&\geq
		\EOf{\br{\sqrt{\ol Yp^2 + \ol qp^2} - \ol Yp} \indOf{Y\neq m}}
		\\&\geq
		\EOf{\max\brOf{0, \ol qp - \ol Yp}\indOf{Y\neq m}}
		\eqfs
	\end{align}
	Let $h := \ol qp$. Then
	\begin{align}
		&\EOf{\br{\max\brOf{0, \ol qp - \ol Yp}} \indOf{Y\neq m}}
		\\&=
		\EOf{\max\brOf{0, h - \abs{X - s}}\indOf{X\neq 0}}
		\\&=
		\EOf{\br{h - s + X}\indOfOf{[s-h, s]}{X}\indOf{X\neq 0}} +
		\EOf{\br{h + s - X}\indOfOf{(s, s+h]}{X}}
		\eqfs
	\end{align}
	If $h \geq s \geq 0$,
	\begin{align}
		&\EOf{\max\brOf{0, h - \abs{X - s}}\indOf{X\neq 0}}  +
		2 \EOf{\br{s - X}\indOfOf{(0, s]}{X}}
		\\&=
		\EOf{\br{h + s - X}\indOfOf{(0, s+h]}{X}} +
		\EOf{\br{h - s + X}\indOfOf{[s-h, 0)}{X}}
		\\&\geq
		\EOf{\br{h + s - X}\indOfOf{(0, s+h]}{X}}
		\eqfs
	\end{align}
	If $s \geq h \geq 0$,
	\begin{align}
		&\EOf{\max\brOf{0, h - \abs{X - s}}\indOf{X\neq 0}}  +
		2 \EOf{\br{s - X}\indOfOf{(0, s]}{X}}
		\\&=
		\EOf{\br{h + s - X}\indOfOf{[s-h, s+h]}{X}} +
		2 \EOf{\br{s - X}\indOfOf{(0, s-h)}{X}}
		\\&\geq
		\EOf{\br{h + s - X}\indOfOf{(0, s+h]}{X}}
		\eqfs
	\end{align}
	Thus, \eqref{eq:medianGeodSplit} can be bounded from below by
	\begin{equation}
		\ol qm \, a_0 + s \br{a_- - a_+} + \EOf{\br{h + s - X}\indOfOf{(0, s+h]}{X}}
		\eqfs
	\end{equation}
	Finally, we obtain
	\begin{equation}
		\Eof{\ol Yq - \ol Ym}
		\geq
		\ol qm \, a_0 + \ol pm \br{a_- - a_+} + \EOf{\br{\ol qp + \ol pm - \ol Ym}\indOfOf{(0, \ol qp + \ol pm]}{\gamma^{-1}(Y)}}
		\eqfs
	\end{equation}
\end{proof}
\begin{example}
	Let $(\mc Q, d)$ be a Hilbert space with inner product $\ip{\cdot}{\cdot}$ and induced norm $\norm$.
	Let $U$ be an $\R$-valued random variable with a uniform distribution on $[-\frac12, \frac12]$. Let $v\in\mc Q$ have norm $\normof{v}=1$. Let $Y = U v$. As $Y$ has a convex support, its Fréchet median (or geometric median in this context),
	\begin{equation}
		m = \argmin_{q\in\mc Q} \Eof{\normof{Y-q}-\normof{Y}}
		\eqcm
	\end{equation}
	is unique by \cref{cor:convexUnique}. Because of symmetry, $m = 0$. Furthermore, $Y$ is concentrated on the unit-speed geodesic $\gamma\colon[-\frac12,\frac12] \to \mc Q, t\mapsto t v$. Let $q \in\mc Q$ with projection $p$ onto $\setByEleInText{tv}{t\in\R}$. Then $s := \ip{q}{v} = \normof{p}$ and  $h := \sqrt{\normof{q}^2 - \ip{q}{v}^2} = \normof{q - p}$. Set $r:=s+h$. From \cref{thm:median:geodesic}, we obtain
	\begin{align}
		\Eof{\normof{Y-q}-\normof{Y}}
		&\geq
		\EOf{\br{r -\normof{Y}}\indOfOf{(0,r]}{\gamma^{-1}(Y)}}
		\\&=
		r \PrOf{U \in [0, r]}
		-
		\EOf{U\indOfOf{[0, r]}{U}}
		\\&=
		\min\brOf{\frac12, r} \br{r - \frac12\min\brOf{\frac12, r}}
		\\&\geq
		\frac12 r \min\brOf{\frac12, r}
		\eqfs
	\end{align}
	Note that $\normof{q} \leq r \leq \sqrt{2}\normof{q}$.
	Thus, we obtain a quadratic lower bound for all $q\in\mc Q$ close to the median.
\end{example}
\begin{appendix}
\section{Reference Results}\label{sec:ref}
\begin{lemma}[Fundamental theorem of calculus for Lebesgue integrals]\label{lmm:integration}
	Let $a,b\in\R$ with $a < b$. Let $F\colon[a,b]\to\R$.
	\begin{enumerate}[label=(\roman*)]
		\item 
		Assume $F$ be nondecreasing. 
		Then $F$ is differentiable almost-everywhere with the derivative $F\pr(x) \geq 0$ almost everywhere. Furthermore,
		\begin{equation}
			\int_a^b F\pr(x) \dl x \leq F(b) -  F(a)\eqfs
		\end{equation}
		\item 
		Assume $F$ is absolutely continuous.
		Then $F$ is differentiable almost-everywhere, $F\pr$ is Lebesgue-integrable, and we have
		\begin{equation}
			\int_{a}^b F\pr(x) \dl x = F(b) - F(a)
			\eqfs
		\end{equation}
	\end{enumerate}
\end{lemma}
\begin{proof}
	See \cite[chapter 3]{folland99}.
\end{proof}
\begin{lemma}\label{thm:alexandrov}\mbox{ }
	\begin{enumerate}[label=(\roman*)]
		\item \emph{Rademacher's Theorem:} Every locally Lipschitz function $f\colon\R^k\to\R^n$ is almost everywhere differentiable.
		\item \emph{Alexandrov's Theorem:} Every convex function $f\colon\R^k\to\R$ is twice differentiable almost everywhere.
	\end{enumerate}
\end{lemma}
\begin{proof}
	See \cite[Theorems D.1.1 and D.2.1]{niculescu18}.
\end{proof}
\section{Auxiliary Results}\label{sec:aux}
\subsection{One-Sided Derivatives}
\begin{lemma}\label{lmm:middleDeriv}
	Let $I\subset \R$ be convex and $g\colon I \to \R$.
	Let $t_0\in I$. Assume $\semiDerivs g(t_0)$ exists. Then, for all $v_0\in \semiDerivs g(t_0)$, there are $(t_n^+)_{n\in\N}, (t_n^-)_{n\in\N} \subset I$ with $t_n^+ > t_n^-$ and $t_0\in[t_n^-, t_n^+]$ such that $t_n^\pm \xrightarrow{n\to\infty} t_0$, and
	\begin{equation}
		\lim_{n\to\infty} \frac{g(t_n^+) - g(t_n^-)}{t_n^+ - t_n^-} = v_0
		\eqfs
	\end{equation}
\end{lemma}
\begin{proof}
	Let $\alpha \in [0,1]$ so that $v_0 = \alpha g\rd(t_0) + (1-\alpha)  g\ld(t_0)$. We can clearly choose the sequences $(t_n^+)_{n\in\N}, (t_n^-)_{n\in\N} \subset I$ with $t_n^+ > t_n^-$, $t_0\in[t_n^-, t_n^+]$, and $t_n^\pm \xrightarrow{n\to\infty} t_0$ so that $\alpha = \frac{t_n^+ - t_0}{t_n^+ - t_n^-}$. Then
	\begin{align}
		\frac{g(t_n^+) - g(t_n^-)}{t_n^+ - t_n^-}
		&=
		\alpha \frac{g(t_n^+) - g(t_0)}{t_n^+ - t_0} + (1-\alpha)  \frac{g(t_0) - g(t_n^-)}{t_0 - t_n^-}
		\\&\xrightarrow{n\to\infty}
		\alpha g\rd(t_0) + (1-\alpha)  g\ld(t_0)
	\end{align}
\end{proof}
\subsection{Distance Functions}
\begin{lemma}\label{lmm:gconv:solve}
	Let $t_1, t_2\in\R$, $t_2\neq t_1$ and $x_1, x_2 \in \Rp$.
	There are $s\in\R$, $h \in \Rp$ with
	\begin{equation}\label{eq:gconv:solve:def}
		\sqrt{(t_1-s)^2+h^2} = x_1
		\qquad\text{and}\qquad
		\sqrt{(t_2-s)^2+h^2} = x_2
	\end{equation}
	if and only if
	\begin{equation}
		\abs{x_1 - x_2} \leq \abs{t_1 - t_2} \leq  x_1 + x_2
		\eqfs
	\end{equation}
	In this case,
	\begin{align}
		s &= \frac{t_2^2 - t_1^2 + x_1^2 - x_2^2}{2 \delta}\eqcm\\
		h &= \frac12 \sqrt{2\br{x_1^2 + x_2^2} - \delta^2 - \delta^{-2}(x_1^2-x_2^2)^2}\eqcm
	\end{align}
	where $\delta := t_2 - t_1$.
	Furthermore, $h =0$ if and only if $\abs{\delta} \in \{\abs{x_1 - x_2},  x_1 + x_2\}$.
\end{lemma}
\begin{proof}
	As $x_1, x_2 \geq 0$, \eqref{eq:gconv:solve:def} is equivalent to
	\begin{align}
		(t_1-s)^2+h^2 &= x_1^2\eqcm\label{eq:gconv:solve:def1}\\
		(t_2-s)^2+h^2 &= x_2^2
		\eqfs
	\end{align}
	Subtracting the two equations yields
	\begin{equation}
		(t_1-s)^2 - (t_2-s)^2 = x_1^2 -  x_2^2
		\eqfs
	\end{equation}
	Thus,
	\begin{align}
		0
		&=
		(t_1-s)^2 - (t_2-s)^2 - x_1^2 + x_2^2
		\\&=
		t_1^2 - 2 s t_1 + s^2 - t_2^2 + 2 s t_2 - s^2 - x_1^2 + x_2^2
		\\&=
		2 s (t_2-t_1) + t_1^2 - t_2^2  - x_1^2 + x_2^2
	\end{align}
	and we obtain
	\begin{equation}
		s = \frac{- t_1^2 + t_2^2 + x_1^2 - x_2^2}{2 (t_2-t_1)}
		\eqfs
	\end{equation}
	We plug this in to \eqref{eq:gconv:solve:def1} to obtain for the other parameter,
	\begin{align}
		h^2
		&=
		x_1^2 - (t_1-s)^2
		\\&=
		x_1^2 - t_1^2 + 2 s t_1 - s^2
		\\&=
		x_1^2 - t_1^2 + 2 t_1 \frac{- t_1^2 + t_2^2 + x_1^2 - x_2^2}{2 (t_2-t_1)} - \br{\frac{- t_1^2 + t_2^2 + x_1^2 - x_2^2}{2 (t_2-t_1)}}^2
		\eqfs
	\end{align}
	To ensure $h^2\geq0$, we calculate
	\begin{align}
		&\br{2 (t_2-t_1)}^2 h^2
		\\&=
		\br{2 (t_2-t_1)}^2 \br{x_1^2 - t_1^2} + 2 t_1 \br{2 (t_2-t_1)} \br{- t_1^2 + t_2^2 + x_1^2 - x_2^2}
		\\&\phantom{=}\,
		- \br{- t_1^2 + t_2^2 + x_1^2 - x_2^2}^2
		\\&=
		2 t_1^2 x_1^2 + 2 t_1^2 x_2^2 - 4 t_2 t_1 x_1^2 - 4 t_2 t_1 x_2^2 + 2 t_2^2 x_1^2 + 2 t_2^2 x_2^2
		\\&\phantom{=}\,
		- t_1^4 + 4 t_2 t_1^3 - 6 t_2^2 t_1^2 + 4 t_2^3 t_1 - t_2^4 - x_1^4 - x_2^4 + 2 x_1^2 x_2^2
		\\&=
		2 \br{t_1 - t_2}^2 \br{x_1^2 + x_2^2} -(t_1 - t_2)^4 - (x_1^2 - x_2^2)^2
		\\&= 2 \Delta \br{x_1^2 + x_2^2} - \Delta^2 - (x_1^2-x_2^2)^2
		\eqcm
	\end{align}
	where $\Delta := (t_1 - t_2)^2$.
	We need that $2 \Delta \br{x_1^2 + x_2^2} - \Delta^2 - (x_1^2-x_2^2)^2 \geq 0$ for $h$ to exist.
	In other words,
	\begin{equation}
		\br{x_1 - x_2}^2 \leq \Delta \leq  \br{x_1 + x_2}^2
	\end{equation}
	or
	\begin{equation}
		\abs{x_1 - x_2} \leq \abs{t_1 - t_2} \leq  x_1 + x_2
		\eqfs
	\end{equation}
	Then
	\begin{equation}
		h^2
		=
		\frac{2 \Delta \br{x_1^2 + x_2^2} - \Delta^2 - (x_1^2-x_2^2)^2}{4 \Delta}
		\eqcm
	\end{equation}
	which, for $h\geq0$, is equivalent to
	\begin{equation}
		h = \frac12 \sqrt{2\br{x_1^2 + x_2^2} - \Delta - \Delta^{-1}(x_1^2-x_2^2)^2}
		\eqfs
	\end{equation}
\end{proof}
\begin{lemma}\label{lmm:secondDerivInG}
	Let $g\in\mc G$ with parameters $t_0\in\R$ and $h\in\Rp$, i.e.,
	\begin{equation}
		g(t) = \sqrt{(t-t_0)^2 + h^2}
		\eqfs
	\end{equation}
	Let $s\in\R$. Assume $g(s) > 0$. Then $g$ is twice continuously differentiable at $s$ and
	\begin{equation}
		g\prr(s) g(s) = 1 - g\pr(s)^2
		\eqfs
	\end{equation}
\end{lemma}
\begin{proof}
	We calculate the first and second derivative of $g$ at $s$:
	\begin{align}
		g\pr(s) &= \frac{s-t_0}{\sqrt{(s-t_0)^2 + h^2}} =  \frac{s-t_0}{g(s)}\eqcm\\
		g\prr(s) &= \frac{g(s) - (s-t_0)g\pr(s)}{g(s)^2}\eqfs
	\end{align}
	Thus,
	\begin{equation}
		g\prr(s) g(s) = \frac{g(s) - (s-t_0)g\pr(s)}{g(s)} = 1 - g\pr(s)^2 \eqfs
	\end{equation}
\end{proof}
\subsection{Convex and Concave}
\begin{lemma}\label{lem:aux}
	Let $\tran\in\setcc$ not constant. Then, for $x\in\Rpp$,
	\begin{equation}
		1 \leq \frac{x \dtran(x)}{\tran(x)-\tran(0)} \leq 2
		\eqfs
	\end{equation}
\end{lemma}
\begin{proof}
	By \cref{lmm:ccdiff},
	\begin{equation}
		\frac{\dtran(x)+\dtran(0)}2 \leq \frac{\tran(x) - \tran(0)}{x} \leq \dtran\brOf{\frac{x}{2}}
		\eqfs
	\end{equation}
	As $\dtran(0) \geq 0$ and $\dtran$ is nondecreasing,
	\begin{equation}
		\frac12 x \dtran(x) \leq \tran(x) - \tran(0) \leq x\dtran\brOf{x}
		\eqfs
	\end{equation}
\end{proof}
\begin{lemma}\label{lmm:aux:redistri}\mbox{ }
	\begin{enumerate}[label=(\roman*)]
		\item Let $f\colon\Rp \to \R$. Assume $f$ is concave. Let $a,b\in\Rp$ with $a\geq b$. Then
		$x \mapsto f(a+x) + f(b-x)$ is nonincreasing on $[0, b]$. If additionally $f(0)\geq 0$, then $f$ is subadditive.
		\item
		Let $f\colon\Rp \to \R$. Assume $f$ is convex. Let $a,b\in\Rp$ with $a\geq b$. Then
		$x \mapsto f(a+x) + f(b-x)$ is nondecreasing on $[0, b]$.
	\end{enumerate}
\end{lemma}
\begin{proof}
	We prove the first part; the second part is similar.
	As $f$ is concave, we have
	\begin{align}
		f(a) &\geq \frac{a-b+x}{a-b+2x} f(a+x) + \frac{x}{a-b+2x} f(b-x)\eqcm\\
		f(b) &\geq \frac{x}{a-b+2x} f(a+x) + \frac{a-b+x}{a-b+2x} f(b-x)
	\end{align}
	for $x\in[0, b]$. Adding the two inequalities yields
	\begin{equation}
		f(a) + f(b) \geq f(a+x)  + f(b-x)
		\eqfs
	\end{equation}
	As this inequality also applies when $a$, $b$ are replaced by $\tilde a = a+\tilde x$, $\tilde b = b-\tilde x$ for $\tilde x \in [0, b]$, we have that $x \mapsto f(a+x) + f(b-x)$ is nonincreasing.
	Subadditivity follows by setting $x = b$.
\end{proof}
\begin{lemma}\label{lmm:tranconcave}
	Let $\tran\in\setcc$.
	\begin{enumerate}[label=(\roman*)]
		\item \label{lmm:tranconcave:add}
		Let $x,y\in\Rp$. Then
		\begin{equation}
			\dtran(x + y) \leq \dtran(x) +  \dtran(y) \leq 2 \dtran\brOf{\frac{x+y}2}
			\eqfs
		\end{equation}
		\item\label{lmm:tranconcave:factor} Let $a,x\in\Rp$. Then
		\begin{align}
			\dtran(ax) &\geq a\dtran(x) \text{ for } a\leq 1\eqcm\\
			\dtran(ax) &\leq a\dtran(x) \text{ for } a\geq 1
			\eqfs
		\end{align}
		\item\label{lmm:tranconcave:balance} Let $x,y\in\Rp$. Assume $y\geq x$. Then
		\begin{equation}
			x \dtran(y) \leq y \dtran(x)
			\eqfs
		\end{equation}
	\end{enumerate}
\end{lemma}
\begin{proof}
	These are all well-known properties of nonnegative, concave functions.
	\begin{enumerate}[label=(\roman*)]
		\item Use \cref{lmm:aux:redistri} and Jensen's inequality.
		\item Use $(1-t)\dtran(x_0) + t \dtran(x_1) \leq \dtran((1-t) x_0 + t x_1)$ on points $x_0=0$, $x_1=x$, $t = a$ and on $x_0=0$, $x_1=ax$, $t = 1/a$, respectively, and note that $\dtran(0)\geq0$.
		\item Apply \ref{lmm:tranconcave:factor} with $a = y/x$.
	\end{enumerate}
\end{proof}
\begin{lemma}\label{lmm:limittran}
	Let $\tran\in\setcc$, $\tran\not\equiv 0$. Let $r \in \Rp$. Then
	\begin{equation}
		\frac{\tran(x +  r)}{\tran(x)} \xrightarrow{x\to\infty} 1
		\eqfs
	\end{equation}
\end{lemma}
\begin{proof}
	The statement is trivial for $r = 0$. Assume $r>0$. By \cref{lmm:ccdiff},
	\begin{align}
		\frac{\dtran(x+r) + \dtran(x)}2 \leq \frac{\tran(x +  r) - \tran(x)} r \leq \dtran\brOf{x + \frac r2}
		\eqfs
	\end{align}
	From \cref{lem:aux}, we infer
	\begin{equation}
		\frac{\dtran(x)}{\tran(x)} \xrightarrow{x\to\infty} 0
		\eqfs
	\end{equation}
	Thus, on one hand,
	\begin{align}
		\frac{\tran(x +  r)}{\tran(x)}
		&=
		\frac{\tran(x +  r) - \tran(x)}{\tran(x)} + 1
		\\&\leq
		\frac{r \dtran\brOf{x + \frac r2}}{\tran(x)} + 1
		\\&\leq
		r \frac{\dtran(x)}{\tran(x)} + \frac{r \dtran(r/2)}{\tran(x)} + 1
		\\&\xrightarrow{x\to\infty} 1
		\eqfs
	\end{align}
	On the other hand, as $\tran$ is nondecreasing,
	\begin{equation}
		\frac{\tran(x +  r)}{\tran(x)}
		\geq
		1
		\eqfs
	\end{equation}
\end{proof}
\begin{lemma}\label{lmm:limittran:finiteT}
    Let $\tran\in\setcciz$ with $T := \sup_{x\in\Rp}\dtran(x) < \infty$. Let $a \in \Rpp$. Then
    \begin{equation}
        \frac{ax}{\tran(x)} \xrightarrow{x\to\infty} \frac aT \qquad\text{and}\qquad
        \frac{\tran(ax)}{\tran(x)} \xrightarrow{x\to\infty} a
        \eqfs
    \end{equation}
\end{lemma}
\begin{proof}
    On one hand, for all $x_0\in\Rp$,
    \begin{equation}
        \tran(x) \leq \tran(x_0) + T(x - x_0)\eqfs
    \end{equation}
    On the other hand, $\dtran(x) \xrightarrow{x\to\infty} T$ as $\tran$ is convex. Hence, for all $\delta\in\Rpp$, there is $x_0\in \Rp$ such that
    \begin{equation}
        \tran(x) \geq \tran(x_0) + (1-\delta)T(x - x_0)
    \end{equation}
    for all $x \geq x_0$.
    Thus, if $a>0$ and $\delta \in (0,1)$, we can choose $x_0\in\Rp$ large enough so that
    \begin{equation}
        \frac{\tran(ax_0) + (1-\delta) T a (x-x_0)}{\tran(x_0) + T (x-x_0)} \leq \frac{\tran(ax)}{\tran(x)} \leq \frac{\tran(ax_0) + T a (x-x_0)}{\tran(x_0) + (1-\delta) T (x-x_0)}
    \end{equation}
    for all $x \geq x_0$.
    Hence, we obtain $\tran(ax) / \tran(x) \xrightarrow{x\to\infty} a$. The argument for $\frac{ax}{\tran(x)} \xrightarrow{x\to\infty} \frac aT$ is similar.
\end{proof}
\section{Omitted Proofs}\label{apendix:omitted}
\subsection{From Section \ref{sec:ncfcd}}\label{apendix:omitted:ncfcd}
\begin{proof}[Proof of \cref{lmm:tran:continuous}]\mbox{ }
	\begin{enumerate}[label=(\roman*)]
		\item As $\tran$, $\dtran$ are convex and concave, respectively, we immediately obtain continuity on $\Rpp$. As $\tran$ is nondecreasing and convex, it must also be continuous at 0. By the definition of $\dtran(0)$, it is continuous at $0$.
		\item The statements are well-known properties of finite, continuous, concave functions such as $\dtran$, see, e.g., \cite[Theorem 24.1]{rockafellar70}.
		\item
		As $\tran$ is convex and nondecreasing on $\Rp$,
		\begin{equation}
			h\mapsto \frac{\tran(h)-\tran(0)}{h}
		\end{equation}
		is nondecreasing and nonnegative. Thus, $\tran\rd(0)$ exists.

		Fix $\epsilon >0$. As $\dtran$ is continuous, it is uniformly continuous on the compact interval $[0,2]$. Thus, we can find $\delta\in(0,1]$ such that $\abs{\dtran(x)-\dtran(y)} \leq \epsilon$ for all $x,y\in[0,2]$ with $0 < \abs{x-y} \leq \delta$. Thus, for any $h\in (0, \delta]$ and $x \in (0, 1]$,
		\begin{equation}
			\abs{\frac{\tran(x+h)-\tran(x)}{h} - \dtran(x)}
			=
			\abs{\frac1h \int_0^h \dtran(x+z) - \dtran(x) \dl z}
			\leq
			\epsilon
			\eqfs
		\end{equation}
		Choose $h\in (0, \delta]$ small enough such that,
		\begin{equation}
			\abs{\frac{\tran(h)-\tran(0)}{h} - \tran\rd(0)}
			\leq
			\epsilon
			\eqfs
		\end{equation}
		As $\tran$ is continuous and $h$ is fixed, we can find $x\in(0,\delta]$ small enough such that
		\begin{align}
			\abs{\tran(x + h) - \tran(h)} &\leq h\epsilon\eqcm\\
			\abs{\tran(x) - \tran(0)} &\leq h\epsilon\eqfs
		\end{align}
		Thus, using the triangle inequality and putting the bounds together, we obtain
		\begin{align}
			&\abs{\dtran(0) - \tran\rd(0)}
			\\&\leq
			\abs{\dtran(0) - \dtran(x)} +
			\abs{\dtran(x) - \frac{\tran(x+h)-\tran(x)}{h}}
			\\&\phantom{\leq}\,
			+
			\abs{\frac{\tran(x+h)-\tran(x)}{h} - \frac{\tran(h)-\tran(0)}{h}}+
			\abs{\frac{\tran(h)-\tran(0)}{h} - \tran\rd(0)}
			\\&\leq
			3\epsilon +
			\frac1h\br{\abs{\tran(x+h)-\tran(h)}+\abs{\tran(x)-\tran(0)}}
			\\&\leq 5\epsilon
			\eqfs
		\end{align}
		As $\epsilon > 0$ can be chosen arbitrarily small, $\dtran(0) = \tran\rd(0)$.
	\end{enumerate}
\end{proof}
\begin{proof}[Proof of \cref{lmm:ccdiff}]
	Let $x,y\in\Rp$ with $x > y$. For the lower bound, as $\dtran$ is concave,
	\begin{align}
		\tran(x) - \tran(y)
		&=
		\int_y^x \dtran(u) \dl u
		\\&\geq
		(x-y)\int_0^1 (1-t)\dtran(y) + t \dtran(x)\, \dl t
		\\&=
		\frac {x-y}2 \br{\dtran(x)+\dtran(y)}
		\eqfs
	\end{align}
	For the upper bound, concavity of $\dtran$ implies the existence of an affine linear function $h$ with $h(u) \geq \dtran(u)$ for all $u \in\Rp$ and
	\begin{equation}
		h\brOf{\frac{x+y}{2}} = \dtran\brOf{\frac{x+y}{2}}
		\eqfs
	\end{equation}
	Thus,
	\begin{align}
		\tran(x) - \tran(y)
		&\leq
		\int_y^x h(u) \dl u
		\\&=
		\frac {x-y}2\br{h(x)+h(y)}
		\\&=
		(x-y) h\brOf{\frac{x+y}{2}}
		\eqfs
	\end{align}
\end{proof}
\begin{proof}[Proof of \cref{lmm:tranlower}]
	First, consider the case $x\geq y$.
	Define $f(x,y) = \tran(x-y) - \tran(x)- \tran(y) + 2 y \dtran(x)$.
	We want to show $f(x, y)\geq0$.
	The derivative of $f$ with respect to $y$ is
	\begin{equation}
		\partial_y f(x, y) = -\dtran(x-y) - \dtran(y) + 2 \dtran(x)\eqfs
	\end{equation}
	As $\dtran$ is nondecreasing and $x\geq \max(y, x-y)$, we obtain $\partial_y f(x, y)  \geq 0$.
	Hence, $f(x, y) \geq f(x, 0) = 0$, as $\tran(0) = 0$.

	Now, consider the case $x \leq y$. Set $g(x,y) = \tran(y-x) - \tran(x) - \tran(y) + 2 y \dtran(x)$, which yields
	\begin{equation}
		\partial_y g(x, y) = \dtran(y-x) - \dtran(y) + 2 \dtran(x)\eqfs
	\end{equation}
	As $\dtran$ is subadditive (\cref{lmm:tranconcave}), we obtain $\partial_y g(x, y)  \geq 0$.
	Thus, $g(x, y) \geq g(x, x) = -2 \tran(x) + 2 x \dtran(x)$ as $\tran(0) = 0$. As $\dtran$ is nondecreasing and $\tran(0) = 0$, we have $\tran(x) \leq x \dtran(x)$. Hence, $g(x, y) \geq 0$.
\end{proof}
\subsection{From Section \ref{sec:quadLowerBound}}\label{apendix:omitted:quadLowerBound}
\begin{proof}[Proof of \cref{prp:gconv:set}]\mbox{ }
	\begin{enumerate}[label=(\roman*)]
		\item
		The quadratic polynomial $t \mapsto t^2 + at + b$ attains its minimum $-\frac{a^2}4+b$ at $t_0 = -\frac a2$. It is nonnegative if and only if $4b \geq a^2$. A function of the form $t\mapsto (t-t_0)^2+h^2$ can be written as
		$t^2 + \tilde a t + \tilde b$ with $\tilde a = -2t_0$ and $\tilde b = t_0^2 + h^2$. Thus, $4 \tilde b \geq \tilde  a^2$. Furthermore, we can clearly choose $h\in\Rp$ and $t_0\in\R$ to obtain any $\tilde a, \tilde b\in\R$ with $4 \tilde b \geq \tilde a^2$.
		\item
		Let $\normof{\cdot}$ be the norm of the Hilbert space $\mc Q$. By the Pythagorean theorem,
		\begin{align}
			\normof{y - \gamma(t)}^2
			&=
			\normof{y - \gamma(t_0)}^2 + \normof{\gamma(t_0) - \gamma(t)}^2
			\\&=
			\normof{y - \gamma(t_0)}^2 + (t-t_0)^2
			\eqcm
		\end{align}
		where $\gamma(t_0)$ is the orthogonal projection of $y$ onto $\gamma$. If the dimension of $\mc Q$ is at least $2$, we can choose $y$ to obtain any value for $t_0\in\R$ and $\normof{y - \gamma(t_0)}\in\Rp$.
	\end{enumerate}
\end{proof}
\begin{proof}[Proof of \cref{prp:gconv:elem}]\mbox{ }
	\begin{enumerate}[label=(\roman*)]
		\item Trivial.
		\item If $g_1, g_2$ are parameterized by $t_{0,1}, h_1$ and $t_{0,2}, h_2$, respectively, then squaring the two equations yields
		\begin{align}
			(s - t_{0,1})^2 + h_1^2 &= (s - t_{0,2})^2 + h_2^2\eqcm\\
			(t - t_{0,1})^2 + h_1^2 &= (t - t_{0,2})^2 + h_2^2\eqfs
		\end{align}
		The difference of these two equations yields
		\begin{equation}
			(s - t_{0,1})^2 - (t - t_{0,1})^2 = (s - t_{0,2})^2 + (t - t_{0,2})^2 \eqcm
		\end{equation}
		which is equivalent to $(t-s) (t_{0,1}-t_{0,2}) = 0$. Using $t_{0,1}=t_{0,2}$ in $f_1(s) = f_2(s)$ yields $h_1^2 = h_2^2$.
		\item
		Set $h(x) =  g_2(x)^2 - g_1(x)^2$. Then $h$ is an affine linear function as the squared terms cancel. Furthermore, $h(r)\geq 0$, $h(s)\leq 0$, and $h(t) \geq 0$. Thus $h \equiv 0$ and we have $g_1 = g_2$.
		\item
		Follows directly from \cref{lmm:gconv:solve}.
	\end{enumerate}
\end{proof}
\begin{proof}[Proof of \cref{prp:gconv:fun}]\mbox{ }
	\begin{enumerate}[label=(\roman*)]
		\item All functions in $\mc G$ are nonnegative.
		\item All functions in $\mc G$ are $1$-Lipschitz. Let $t_1, t_2 \in \R$. Without loss of generality, assume $f(t_2) \geq f(t_1)$. Let $g\in\mc G$ be a $\mc G$-tangent of $f$ at $t_2$. Then $f(t_2) - f(t_1) \leq g(t_2) - g(t_1) \leq \abs{t_2 - t_1}$.
		\item All functions in $\mc G$ are convex. Thus, \cref{prp:gconv:characterization} \ref{prp:gconv:characterization:aboveBelow} yields convexity of $f$ (we do not use \ref{prp:gconv:fun:convex} in the proof of \cref{prp:gconv:characterization} \ref{prp:gconv:characterization:aboveBelow}).
		\item Let $g \in \mc G$ be a $\mc G$-tangent of $f$ at $t_1$. Then $f(t_1) + f(t_2) \geq g(t_1) + g(t_2)$. By \cref{lmm:gconv:solve}, $g(t_1) + g(t_2) \geq \abs{t_1 - t_2}$.
		\item The only function $g\in\mc G$ with $g(t_0) = 0$ is $\abs{t-t_0}$. As $f$ is $1$-Lipschitz and $t \mapsto \abs{t-t_0}$ is a lower bound of $f$ with equality at $t_0$, $f$ must also be the function $t \mapsto \abs{t-t_0}$.
	\end{enumerate}
\end{proof}
\begin{proof}[Proof of \cref{prp:gconvVSstrong}]\mbox{ }
	\begin{enumerate}[label=(\roman*)]
		\item
		The functions $g, g^2 \colon \R\to\R$ are convex. Thus, their subdifferentials are the closed intervals between the respective left and right derivative.
		Let $t_0\in\R$. By \cref{prp:gconv:characterization} \ref{prp:gconv:characterization:lower}, for all $s\in\R$,
		\begin{equation}
			g(s)^2 \geq g(t_0)^2 + 2 (s-t_0) g(t_0)g\rd(t_0) + (s-t_0)^2
			\eqfs
		\end{equation}
		This inequality holds if we replace $g\rd(t_0)$ by $g\ld(t_0)$ (follow the proof of \cref{prp:gconv:characterization}  \ref{prp:gconv:characterization:lower}). Thus it is true for every subderivative (subgradient in one dimension) in $\partial g(t_0)$. Furthermore,
		\begin{equation}
			\partial g^2 (t_0) = \setByEle{2 g(t_0) v}{v \in \partial g(t_0)}
			\eqfs
		\end{equation}
		Thus, by \cref{prp:strongconv:chara}, $g^2$ is strongly convex with modulus $1$.
		\item
		Set $g := \sqrt{f}$. As $f$ is strongly convex with modulus $1$, by \cref{prp:strongconv:chara},
		\begin{equation}\label{eq:gconvVSstron:1}
			g(s)^2 \geq g(t_0)^2 + (s-t_0) (g^2)\rd(t_0) + (s-t_0)^2 := \tilde f(s)
			\eqfs
		\end{equation}
		As the square root of a degree two polynomial with positive second order coefficient is convex, $g$ is convex. Thus, its one-sided derivatives exist.
 		There is $v_0 \in \semiDerivs g(t_0)$ such that $(g^2)\rd(t_0) = 2 g(t_0)v_0$.
		As we assume $g$ to be $1$-Lipschitz, we have $v_0^2 \leq 1$ and $(1 - v_0^2) g(t_0)^2 \geq 0$.
		We use the first of these two inequalities to show $\tilde f\geq0$. Hence, we can define $\tilde g := \sqrt{\tilde f}$.
		The second inequality allows us to write
		\begin{equation}
			\tilde g(s) = \sqrt{\br{s - s_0}^2 + h^2}
		\end{equation}
		with $s_0 = t_0 + g(t_0)v_0$ and $h^2 = (1 - v_0^2) g(t_0)^2$.
		Now we have $g(s) \geq \tilde g(s)$ for all $s\in I$ by \eqref{eq:gconvVSstron:1}. Furthermore, we calculate $g(t_0) = \tilde g(t_0)$. Clearly, $\tilde g\in\mc G$. Hence, $\tilde g$ is a $\mc G$-tangent of $g$ at $t_0$.
	\end{enumerate}
\end{proof}

\end{appendix}
\printbibliography
\end{document}